\documentclass[a4paper,10pt]{elsarticle}

\newif\ifusetikzexternalize
\usetikzexternalizetrue

\usepackage{amsmath,amsthm,amssymb,scalerel}
\usepackage{bm}
\usepackage{mathtools}

\usepackage{thmtools, thm-restate}

\declaretheorem{corollary}
\declaretheorem{remark}

\newcommand\mtext[1]{\text{#1}}
\newcommand{\+}[1]{\boldsymbol{\ensuremath{\mathbf{#1}}}}

\newcommand\roundpar[1]{\left( #1 \right)}
\newcommand\squarepar[1]{\left[ #1 \right]}

\newcommand\oneDIntegral[4]{\int_{#1}^{#2} #3 \hspace{1ex} d#4}
\newcommand\integral[3]{\int_{#1} #2 \hspace{1ex} d#3}

\newcommand\setextrusion[1]{{\omega(\mathcal{#1})}}

\newcommand\defeq{\mathrel{\mathop:}=}

\newcommand\half{\frac{1}{2}}

\usepackage{stackengine}
\DeclareFontFamily{U}{mathx}{\hyphenchar\font45}
\DeclareFontShape{U}{mathx}{m}{n}{
      <5> <6> <7> <8> <9> <10>
      <10.95> <12> <14.4> <17.28> <20.74> <24.88>
      mathx10
      }{}
\DeclareSymbolFont{mathx}{U}{mathx}{m}{n}
\DeclareMathAccent{\widecheck}    {0}{mathx}{"71}
\newcommand\approximate[1]{\widecheck{#1}}

\newcommand\stackrelwidth[2]{\stackrel{\mathmakebox[\widthof{#2}]{#1}}{#2}}

\def\connectionc{c{(f,h)}}

\def\onefluid{one-velocity }
\def\twofluid{two-velocity }
\def\Onefluid{One-velocity }
\def\Twofluid{Two-velocity }

\def\stagger{\widetilde}
\def\volfrac{\alpha}
\def\volfracstag{\stagger\alpha}
\def\orientation{o}

\def\tvar{T}
\def\avar{\varphi}
\def\avarstag{\stagger\varphi}

\def\onevelo{\bar{u}}
\def\Onevelo{\bar{\+u}}
\def\volumeflux{\upsilon}
\def\volumefluxOne{\bar\volumeflux}
\def\volumefluxOneStag{\stagger{\bar\volumeflux}}

\def\volumefluxTwo{\widehat\volumeflux}

\def\totalfluxctu{V}

\def\massflux{m}
\def\massfluxOne{\bar{m}}
\def\massfluxOneStag{\stagger{\bar{m}}}
\def\massfluxTwo{\widehat{m}}
\def\dt{\delta t}

\newcommand\dr{{\Delta}}
\newcommand\drOne{{\bar\dr}}
\newcommand\drOneStag{\stagger{\bar\dr}}
\newcommand\drTwo{{\widehat\dr}}

\newcommand\interpolanthalf{\mathfrak{E}}
\newcommand\interpolanthalfstag{\stagger{\interpolanthalf}}
\newcommand\interpolantsbp{\mathfrak{I}}
\newcommand\interpolantflux{\stagger{\interpolantsbp}}
\newcommand\fluxinterp{{\mathfrak{F}}}
\newcommand\fluxinterpstag{\stagger{\mathfrak{F}}}
\newcommand\tanginterp{\stagger{\mathfrak{U}}}

\newcommand\flowmap[2]{\Phi^{#1}_{#2}}

\newcommand\gradient{\nabla}
\newcommand\divergence{\gradient \cdot}

\newcommand\advection[2]{\mathfrak{A}[#1]#2}
\newcommand\advectionstag[2]{\stagger{\mathfrak{A}}[#1]#2}
\newcommand\divh{\mathfrak{D}}
\newcommand\gradh{\mathfrak{S}}

\newcommand\abs[1]{\left\lvert #1 \right\rvert}

\newcommand\ratio[1]{#1^g/#1^l}

\newcommand\mathcfl{\nu}

\newcommand\posflux[1]{\squarepar{#1}^+}
\newcommand\negflux[1]{\squarepar{#1}^-}
\newcommand\minfracval{\volfrac^\dagger}
\newcommand\wycflval{\nu^\dagger}

\usepackage{stmaryrd}

\newcommand\jump[1]{\left\llbracket#1\right\rrbracket}

\newcommand\mean[1]{\dgal*{#1}}
\usepackage{xparse}
\NewDocumentCommand{\dgal}{sO{}m}{%
  \IfBooleanTF{#1}
    {\dgalext{#3}}
    {\dgalx[#2]{#3}}%
}

\NewDocumentCommand{\dgalext}{m}{%
  \sbox0{%
    \mathsurround=0pt 
    $\left\{\vphantom{#1}\right.\kern-\nulldelimiterspace$%
  }%
  \sbox2{\{}%
  \ifdim\ht0=\ht2
    \{\kern-.625\wd2 \{#1\}\kern-.625\wd2 \}%
  \else
    \left\{\kern-.7\wd0\left\{#1\right\}\kern-.7\wd0\right\}%
  \fi
}

\NewDocumentCommand{\dgalx}{om}{%
  \sbox0{\mathsurround=0pt$#1\{$}%
  \sbox2{\{}%
  \ifdim\ht0=\ht2
    \{\kern-.625\wd2 \{#2\}\kern-.625\wd2 \}%
  \else
    \mathopen{#1\{\kern-.7\wd0 #1\{}
    #2
    \mathclose{#1\}\kern-.7\wd0 #1\}}
  \fi
}

\newcommand\set[2]{
  \{\,#1 \mid #2\,\}
}

\usepackage{geometry}
\usepackage{graphicx}

\usepackage{tikz}
\usetikzlibrary{
  arrows.meta,
  positioning}

\ifusetikzexternalize
  \usetikzlibrary{
    arrows.meta,
    external}
  \tikzset{external/system call={pdflatex \tikzexternalcheckshellescape -halt-on-error
    -interaction=batchmode -jobname "\image" "\texsource" && 
    pdfseparate -f 1 -l 1 "\image".pdf "\image"_tmp.pdf &&
    mv "\image"_tmp.pdf "\image".pdf}}
\else

  \newcommand\tikzsetnextfilename[1]{}
  \newcommand\tikzexternalenable{}
  \newcommand\tikzexternaldisable{}
\fi

\usepackage[utf8]{inputenc}
\usepackage{pgfplots}
\usepackage{pgfgantt}
\usepackage{pdflscape}
\pgfplotsset{compat=newest}
\pgfplotsset{plot coordinates/math parser=false}
\pgfplotsset{
    legend image with text/.style={
        legend image code/.code={%
            \node[anchor=center] at (0.3cm,0cm) {#1};
        }
    },
}
\usetikzlibrary{
    matrix,
}
\pgfplotsset{
    compat=1.3,
}

\usepgfplotslibrary{fillbetween}

\newcommand\gls[1]{\MakeUppercase{#1}}
\newcommand\glsplural[1]{\MakeUppercase{#1}s}

\usepackage{subcaption}

\def\twofigwidth{0.475\textwidth}
\def\threefigwidth{0.3\textwidth}

\usepackage[]{natbib}
\usepackage[hidelinks]{hyperref}
\usepackage{cleveref}
\usepackage{autonum} 
\newcommand\caref[1]{eq.~\eqref{#1}} 
\usepackage{import}
\usepackage{stackengine}

\usepackage{booktabs}
\usepackage{array}
\usepackage{varwidth} 
\usepackage{multirow}

\usepackage{enumitem}
\newtheorem{challenge}{Challenges}
\crefname{challenge}{challenge}{challenges}
\newlist{chalenum}{enumerate}{1} 
\setlist[chalenum]{label=\arabic*), ref=\thechallenge.\arabic*}
\crefalias{chalenumi}{challenge}
\crefformat{appendix}{#2#1#3}

\begin{document}
  \begin{frontmatter}

    \author{
      Ronald A. Remmerswaal
    }
    \author{
      Arthur E.P. Veldman
    }

    \address{
      Bernoulli Institute, University of Groningen\\
      PO Box 407, 9700 AK Groningen, The Netherlands
    }

    \begin{abstract}
      The numerical modelling of convection dominated high density ratio two-phase flow poses several challenges, amongst which is resolving the relatively thin shear layer at the interface.
To this end we propose a sharp discretisation of the \twofluid model of the two-phase Navier--Stokes equations. 
This results in the ability to model the shear layer, rather than resolving it, by allowing for a velocity discontinuity in the direction(s) tangential to the interface.

In this paper we focus our attention on the transport of mass and momentum in the presence of such a velocity discontinuity.
We propose a generalisation of the dimensionally unsplit geometric volume of fluid (VOF) method for the advection of the interface in the \twofluid formulation.
Sufficient conditions on the construction of donating regions are derived that ensure boundedness of the volume fraction for dimensionally unsplit advection methods.
We propose to interpolate the mass fluxes resulting from the dimensionally unsplit geometric VOF method for the advection of the staggered momentum field, resulting in semi-discrete energy conservation.
Division of the momentum by the respective mass, to obtain the velocity, is not always well-defined for nearly empty control volumes and therefore care is taken in the construction of the momentum flux interpolant: our proposed flux interpolant guarantees that this division is always well-defined without being unnecessarily dissipative.

Besides the newly proposed \twofluid model we also detail our exactly conservative (mass per phase and total linear momentum) implementation of the \onefluid formulation of the two-phase Navier--Stokes equations, which will be used for comparison.

The discretisation methods are validated using classical time-reversible flow fields, where in this paper the advection is uncoupled from the Navier--Stokes solver, which will be developed in a later paper.

    \end{abstract}

    \title{Towards a sharp, structure preserving \twofluid model for two-phase flow: transport of mass and momentum}

  \end{frontmatter}

  \section{Introduction}\label{sec:introduction}
The numerical simulation of two-phase flow is of great interest to engineering problems (e.g. liquid sloshing), as well as to fundamental research into fluid flow (e.g. to study the physics of liquid impacts).
Discretisation methods used for the transport of mass and momentum play a particularly crucial role, as an erroneous approximation quickly leads to an unphysical exchange of momentum between the phases.
For convection dominated high density ratio two-phase flow in particular, it is known that momentum should be conserved~\citep{Arrufat2021,Desmons2021,Owkes2017,Rudman1998} while it is transported, for obtaining a robust and accurate model.

The aforementioned references use methods based on the \onefluid formulation of the two-phase NSE, wherein the velocity field is assumed to be continuous at the interface
\begin{equation}\label{eqn:intro:full_smoothness}
  \jump{\+u} = 0.
\end{equation}
Here the jump in a quantity is defined as the difference between the gas and liquid value
\begin{equation}
  \jump{\avar} \defeq \avar^g - \avar^l, \quad \mean{\avar} \defeq \avar^g + \avar^l,
\end{equation}
where we have introduced a complimentary notation for the sum over the phases.
This \onefluid formulation results in a single set of equations defined in the entire computational domain $\Omega$, where the fluid properties (i.e. the density and dynamic viscosity) change when going from the liquid to the gas domain.
\begin{figure}
  \centering

  \subcaptionbox{Fine mesh using the \onefluid formulation.\label{fig:intro:lgpi:1fluid_fine}}
  [\threefigwidth]{
    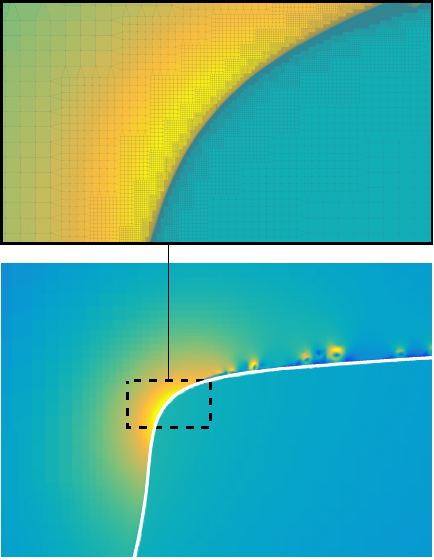
  }
  \hfill
  \subcaptionbox{Coarse mesh using the \onefluid formulation.\label{fig:intro:lgpi:1fluid_coarse}}
  [\threefigwidth]{
    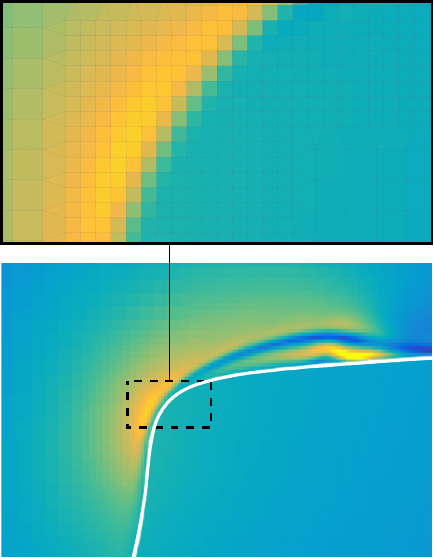
  }
  \hfill
  \subcaptionbox{Coarse mesh using our proposed \twofluid formulation.\label{fig:intro:lgpi:2fluid_coarse}}
  [\threefigwidth]{
    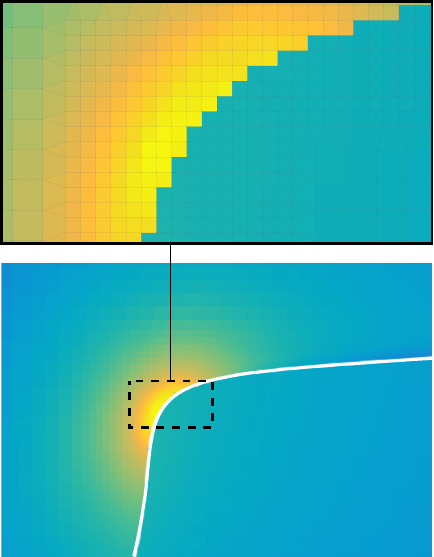
  }
  \caption{Example simulation of the large gas-pocket impact as defined by~\citet{Etienne2018} at scale $1 : 5$ (thus resulting in a $\lambda=4$m long wave) and $t /T_g \approx 0.45$ (where $T_g = 2\pi / \sqrt{gk}$). 
  We show the velocity magnitude, clipped to $\abs{\+u}_2 \in [0, 8]$, at the wave crest and compare three different approximations. 
  The `fine' mesh yields a resolved shear layer with an interface resolution of $h \approx 0.1$mm whereas the `coarse' mesh refers to $h \approx 3.1$mm ($32$ times larger).
  Note that the colouring is determined per control volume (depending on whether the control volume centroid lies in the liquid or gas phase), and therefore the visible `staircase' approximation of the interface is entirely due to postprocessing.}
  \label{fig:intro:lgpi}
\end{figure}
The simulation of convection dominated high density ratio two-phase flow, however, can be quite challenging due to the formation of a viscous boundary layer at the interface, which we will refer to as a `shear layer'.
In~\cref{fig:intro:lgpi:1fluid_fine,fig:intro:lgpi:1fluid_coarse} we show two approximate solutions according to the \onefluid formulation on a fine as well as relatively coarse (but still fine enough to resolve the interface) mesh.
Perhaps unsurprisingly we find, due to the assumption of velocity continuity at the interface, that the thickness of the shear layer is determined by the mesh resolution until it is entirely resolved.
This implies that the dynamics in the lighter gas phase are incorrectly approximated until the shear layer is resolved, which raises the question of whether such a shear layer can successfully be modelled. 
Thus resulting in a numerical model for two-phase flow in which the shear layer need not be resolved, from which it follows that the resolution requirements become independent of viscous effects.

We will propose a numerical model for incompressible two-phase flow that is based on a \twofluid formulation of the two-phase NSE~\citep{Favrie2014,Preisig2010,Saurel1999}. 
This means that each phase gets its own momentum (and thus velocity) field, as well as the corresponding equation which governs its conservation.
Rather than imposing continuity of the velocity field, we now merely impose that the interface normal component of the velocity is continuous (i.e., there is no phase change)
\begin{equation}\label{eqn:intro:normal_smoothness}
  \jump{u_\eta} = 0, 
\end{equation}
where $\+\eta$ denotes the interface normal (pointing into the gas phase), and the interface normal component of the velocity is denoted by $u_\eta = \+\eta \cdot \+u$.
The resulting numerical model therefore aims at replacing the unresolved shear layer by a velocity discontinuity in the direction(s) tangential to the interface, as shown in~\cref{fig:intro:lgpi:2fluid_coarse}.

Several multi-equation models for two-phase flow can be found in the literature, resulting in at most two mass conservation equations, two momentum conservation equations, two energy equations and an equation that governs the advection of the interface.
In~\citet{Saurel1999} a hyperbolic seven equation model (2-2-2-1) is considered for the simulation of two-phase compressible flow.
The energy equations are omitted in~\citet{Preisig2010}, who therefore consider a five equation model (2-2-0-1), and in~\citet{Favrie2014} a six equation model (2-1-2-1) is considered where a momentum equation for the mixture is considered.

For incompressible two-phase flow not much can be found in the literature regarding multi-equation models (in particular two-velocity models).
Note that for incompressible flow the mass conservation equations are implied by the equation that governs the advection of the interface, together with a suitable incompressibility constraint (which we will count as a mass conservation equation).
In~\citet{Desjardins2010} and \citet{Vukcevic2017} two separate equations are used for the transport of momentum, but after each time step the velocity jump is removed in the pressure Poisson problem by imposing~\cref{eqn:intro:full_smoothness}. 
In our terminology this constitutes as a \onefluid model (1-1-0-1).
Our proposed \twofluid model is a sharp four equation model (1-2-0-1) wherein the velocity is permitted to be discontinuous according to~\cref{eqn:intro:normal_smoothness}.

The numerical modelling of such a \twofluid formulation poses many challenges, but we will focus our attention on the following three.
\begin{challenge}[\Twofluid formulation]\leavevmode
  In the presence of a velocity discontinuity, how to\ldots
  \begin{chalenum}
    \item transport mass and momentum\label{chal:mass_mom}
    \item impose continuity of the interface normal component of velocity\label{chal:continuity}
    \item ensure that the viscous \onefluid solution is obtained under mesh refinement\label{chal:refinement}
  \end{chalenum}
\end{challenge}

For now we will consider~\cref{chal:mass_mom}, which means that the two fluids are not yet coupled.
The coupling of the fluids, via the interface normal component of velocity, as well as viscous diffusion, is considered in a future paper where the entire Navier--Stokes solver will be detailed.

\begin{figure}
  \centering
  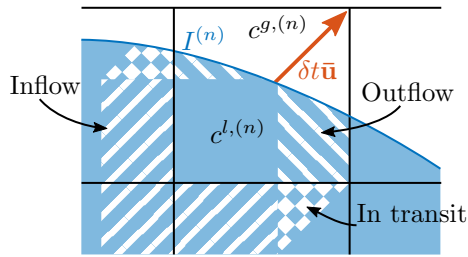
  \caption{Illustration of~\cref{eqn:intro:conservation_eqn}.
  A centred control volume $c$ is shown, separated by the interface $I^{(n)}$, resulting in two parts $c^{\pi,(n)}$.
  The hatched regions correspond to liquid that flows into, through (referred to as `in transit') and out of the control volume $c$ during the time interval $[t^{(n)}, t^{(n+1)}]$.
  Here the advecting velocity is proportional to $[1, 1]^T$ and shown by the arrow.}
  \label{fig:interface:mass_transport}
\end{figure}
The transport of mass ($\avar^\pi = 1$) and momentum ($\+\avar^\pi = \+u^\pi$) is modelled, in conservative integral form, by
\begin{equation}\label{eqn:intro:conservation_eqn}
  \integral{\omega^{\pi,(n+1)}}{\rho^\pi \avar^{\pi}}{V} - \integral{\omega^{\pi,(n)}}{\rho^\pi \avar^{\pi}}{V} + \oneDIntegral{t^{(n)}}{t^{(n+1)}}{\integral{\partial \omega^\pi(t) \setminus I(t)}{\rho^\pi u_n^\pi \avar^\pi }{S}}{t} = 0,
\end{equation}
which is illustrated in~\cref{fig:interface:mass_transport}, with $\omega = c$.
Here $\omega$ denotes some fixed in time control volume, $\omega^{\pi,(n)}$ denotes the part of this control volume which is occupied by the $\pi$-phase at $t = t^{(n)}$ and $I(t)$ denotes the interface, where $\pi \in \{l, g\}$ denotes the respective phase.
Furthermore, $u^\pi$ denotes the velocity of the $\pi$-phase and $\rho^\pi$ the corresponding (constant) density.

The advection of the phase interface is closely related to the conservation of mass, and more specifically, the conservation of mass is implied (our fluids are assumed incompressible) by the conservation of phase volume.
Therefore we focus our attention on geometric volume of fluid (VOF) methods, see e.g. the works of~\citet{Hirt1981,Youngs1982,Rider1997,Puckett1997,Rudman1998,Harvie2001,Lopez2004,Weymouth2010} and~\citet{Owkes2014}.
Amongst those references, we find that only~\citep{Owkes2014} results in a sharp dimensionally unsplit interface advection method which exactly conserves volume, and therefore mass, without the need of using unphysical mass redistribution algorithms, and moreover does not produce any wisps, flotsam and/or jetsam.

Several approaches can be taken in the discretisation of~\cref{eqn:intro:conservation_eqn} for $\avar^\pi \neq 1$ (not constant).
In~\citet{Desjardins2010} and~\citet{Desmons2021} the conservative integral form is first converted into a strong form
\begin{equation}\label{eqn:intro:scalar_advec_strong}
  \partial_t (\rho\avar)^\pi + \divergence\roundpar{\rho \+u \avar}^\pi = 0,
\end{equation}
where for simplicity in notation we write $(\rho\avar)^\pi = \rho^\pi\avar^\pi$.
This strong formulation is then discretised using one-sided extrapolation~\citep{Desjardins2010} (reminiscent of the ghost fluid method~\citep{Liu2000}) or using standard methods for hyperbolic problems~\citep{Desmons2021}.
Neither of such approaches results in exact conservation of mass and linear momentum combined with a sharp interface representation.

Instead, we follow the approach which was first introduced in~\citet{Rudman1998}, and later adopted by~\citet{Chenadec2013,Owkes2017,Zuzio2020} and~\citet{Arrufat2021} for \onefluid formulations.
This approach is based on an approximate space-time integration of the space-time integral in~\cref{eqn:intro:conservation_eqn}, and directly results in conservative transport on a collocated variable layout.
Our proposed advection method differs in the way this approach is generalised for a staggered momentum field.

\subsection*{Overview}
We will first introduce some frequently occurring notation in~\cref{sec:notation}.
In~\cref{sec:mass} we will discuss the dimensionally unsplit discretisation of the interface advection equation, for both the one- and \twofluid formulation, and derive sufficient conditions that guarantee boundedness of the volume fraction.
As the fluids are considered incompressible, the advection equation of the interface directly yields the centred mass conservation equations.
A dimensionally unsplit advection method for the staggered momentum field is proposed and analysed in~\cref{sec:momentum}, and is furthermore compared with methods found in the literature.
Results are shown and discussed in~\cref{sec:validation} and we conclude with a preliminary discussion in~\cref{sec:discussion}.

  \section{Notation}\label{sec:notation}

\subsection{Computational mesh}
The mesh we use is \emph{locally} rectilinear in the sense that the mesh is an adaptive block-based mesh, where each block is rectilinear.
Details on the block-based mesh, as well as modifications of the operators presented here near refinement interfaces, can be found in~\citet{VanderPlas2017}.
Even though we consider a rectilinear mesh, we choose to use a general notation suitable for an arbitrary tessellation of the domain, as presented in~\citet{Lipnikov2014}.
The reason for using this notation is that most of the discretisation methods we propose can in fact be applied to an arbitrary tessellation of the domain.

Rather than using indices $i,j,k$ (which imply the use of a logically rectilinear mesh) to denote some control volume $c_{i,j,k}$ we simply refer to some control volume $c \subset \Omega \subset \mathbb{R}^d$, where $d = 2,3$ denotes the dimensionality of the domain $\Omega$.
The set of all control volumes is denoted by $\mathcal{C}$.
Each control volume $c$ has faces which are denoted by $f \in \mathcal{F}(c)$, where $\mathcal{F}(c)$ is the set of all faces of the control volume $c$, see also~\cref{fig:notation:notation:mesh_c}.
Similarly, the control volumes adjacent to some face $f$ are denoted by the set $\mathcal{C}(f)$.
The set of all faces is denoted by $\mathcal{F}$.
\begin{figure}
  \captionbox{A control volume $c \in \mathcal{C}$ with boundary $\partial c = \bigcup_{f \in \mathcal{F}(c)} f$, where $\mathcal{F}(c) = \{f_1, f_2, f_3, f_4\}$.
  Here $\orientation_{c,f_1} = \orientation_{c,f_2} = +1$ and $\orientation_{c,f_3} = \orientation_{c,f_4} = -1$.\label{fig:notation:notation:mesh_c}}
  [\twofigwidth]{
    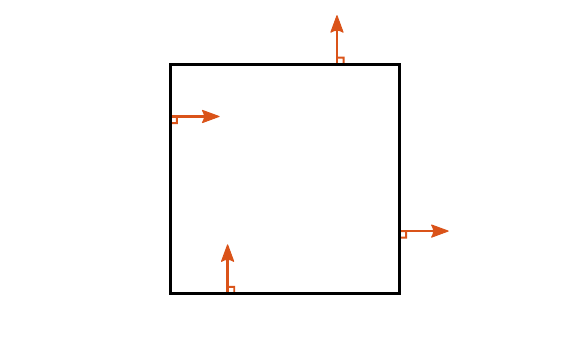
  }
  \hfill
  \captionbox{A staggered control volume $\omega_f$ based on the extrusion of the face $f \in \mathcal{F}$ (dashed line) to the centroids of the neighbouring control volumes $\mathcal{C}(f) = \{c_1, c_2\}$. 
  The faces of the staggered control volume are given by $\mathcal{G}(\omega_f) = \{g_1, g_2, g_3, g_4\}$.\label{fig:notation:notation:mesh_f}}
  [\twofigwidth]{
    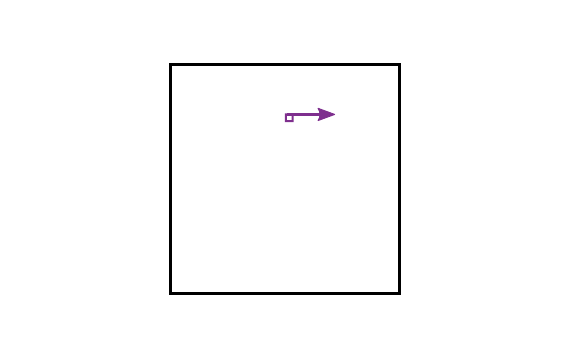 
  }
\end{figure}

The staggered control volume is defined as the extrusion of the face $f$ to the neighbouring control volume centroids $\+x_c$ for $c \in \mathcal{C}(f)$, and is denoted by $\omega_f$, see also~\cref{fig:notation:notation:mesh_f}.
The centroid of a face $f$ is denoted by $\+x_f$.
The boundary of a staggered control volume consists of faces $g \in \mathcal{G}(\omega_f)$, where $\mathcal{G}$ is the set of all such staggered faces.

\subsection{Function spaces and operators}\label{sec:notation:notation:operators}
For the centred and staggered control volumes we define corresponding function spaces.
The centred functions, such as the pressure $p \in \mathcal{C}^h: \mathcal{C} \rightarrow \mathbb{R}$, are such that the value of $p_c$ approximates the value of $p$ in the centroid of the control volume
\begin{equation}
  p^{(n)}_c \approx p(t^{(n)},\+x_c).
\end{equation}
Similarly, the staggered functions, such as the velocity and momentum, $u \in \mathcal{F}^h: \mathcal{F} \rightarrow \mathbb{R}$ approximate the face-normal component of velocity at the face centroid
\begin{equation}\label{eqn:notation:notation:facefunction}
  u^{(n)}_f \approx \+n_f \cdot \+u(t^{(n)},\+x_f),
\end{equation}
where $\+n_f$ is normal to the face $f$.
The tensor valued functions are defined at the centroids of the faces of the staggered control volumes (denoted by $\+x_g$ as shown in~\cref{fig:notation:notation:mesh_f}), that is, $\tvar \in \mathcal{G}^h: \mathcal{G} \rightarrow \mathbb{R}$ is defined such that
\begin{equation}
  \stagger\orientation_{f,g} \tvar^{(n)}_g \approx \+n_f \cdot \+\tvar (t^{(n)}, \+x_g) \cdot \stagger{\+n}_g.
\end{equation}

The divergence operator $\divh: \mathcal{F}^h \rightarrow \mathcal{C}^h$ is defined as a boundary integral divided by the control volume size $|c|$
\begin{equation}\label{eqn:notation:notation:divergence}
  |c|(\divh u)_c \defeq \sum_{f \in \mathcal{F}(c)} |f| \orientation_{c,f} u_f \approx \integral{\partial c}{\+u \cdot \+n}{S},
\end{equation}
where $\orientation$ encodes the local orientation of the face normal $\+n_f$ such that $\orientation_{c,f}\+n_f$ is pointing out of the centred control volume $c$ and $|f|$ denotes the area (length in 2D) of the face $f$.
See also~\cref{fig:notation:notation:mesh_c}.
The gradient $\gradh : \mathcal{C}^h \rightarrow \mathcal{F}^h$ is defined as the negative adjoint (w.r.t. the $L^2$ inner product) of $\divh$.
That is, $\gradh$ is defined exactly such that
\begin{equation}
  \sum_{c\in\mathcal{C}} |c| p_c (\divh u)_c = -\sum_{f \in \mathcal{F}} |\omega_f| u_f (\gradh p)_f, \quad \forall p \in \mathcal{C}^h, u \in \mathcal{F}^h,
\end{equation}
where we have assumed that the boundary contributions vanish (periodic domain or no-slip/slip boundary conditions).
This results in the following definition of the gradient operator
\begin{equation}\label{eqn:notation:notation:gradient}
  (\gradh p)_f \defeq -\frac{1}{h_f}\sum_{c \in \mathcal{C}(f)} \orientation_{c,f} p_c \approx \+n_f \cdot \gradient p (\+x_f),
\end{equation}
where $h_f \defeq |\omega_f| / |f|$ denotes the face projected distance between the centroids of the neighbouring control volumes of the face $f$.
Furthermore we consider the interpolant $\interpolantsbp: \mathcal{C}^h \rightarrow \mathcal{F}^h$, which weighs each scalar value by their respective control volume size
\begin{equation}\label{eqn:notation:notation:interpolantsbp}
  |\omega_f|(\interpolantsbp \volfrac)_f \defeq \half\sum_{c\in\mathcal{C}(f)} |c| \volfrac_c.
\end{equation}
This interpolant is consistent because $|\omega_f| = \half\sum_{c\in\mathcal{C}(f)} |c|$.

\begin{figure}
  \centering 
  \subcaptionbox{Illustration of the set of neighbouring faces $\mathcal{F}(g)$.
  The set of faces that overlap with $g_2$ is given by $\mathcal{F}(g_2) = \{f_1, f_2\}$, whereas the set of neighbouring faces with the same face normal direction as $g_1$ is given by $\mathcal{F}(g_1) = \{f, f_3\}$.\label{fig:notation:notation:set_stag}}
  [\twofigwidth]{
    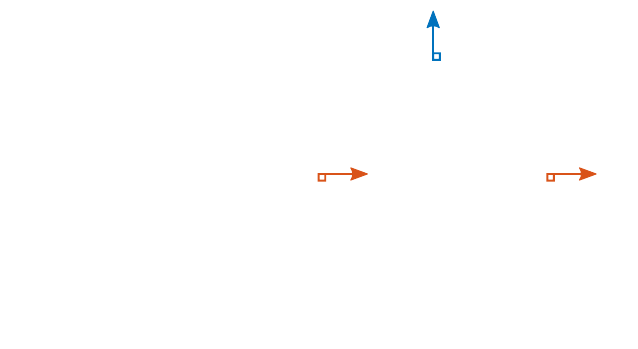
  }\hfill
  \subcaptionbox{Illustration of the set of neighbouring faces $\mathcal{F}^\omega(g)$.
  The set of faces whose staggered control volume share $g_1$ as a boundary is given by $\mathcal{F}^\omega(g_1) = \{f_1, f_2\}$.
  Similarly, for $g_2$ this set is given by $\mathcal{F}^\omega(g^2) = \{f_3, f_4\}$.\label{fig:notation:notation:set_stag_omega}}
  [\twofigwidth]{
    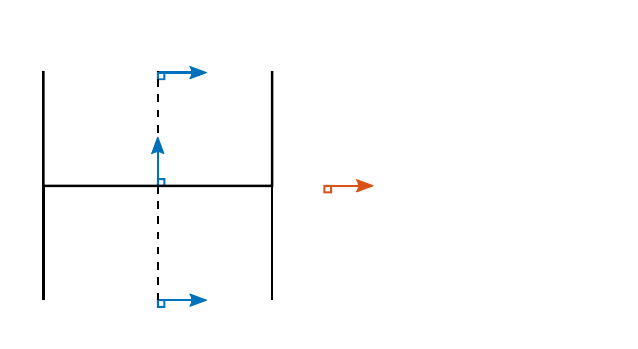
  }
  \caption{The solid boundaries correspond to staggered control volumes $\omega_f$, whereas the dashed boundaries correspond to the centred control volumes $c$.}\label{fig:notation:notation:sets}
\end{figure}
For the staggered control volume $\omega_f$ we define the staggered divergence operator $\stagger \divh: \mathcal{G}^h \rightarrow \mathcal{F}^h$ in terms of a boundary integral over the faces $g \in \mathcal{G}(\omega_f)$ divided by the staggered control volume size $|\omega_f|$
\begin{equation}\label{eqn:notation:notation:divergence_stag}
  |\omega_f| (\stagger \divh \tvar)_f \defeq \sum_{g \in \mathcal{G}(\omega_f)} \stagger\orientation_{f,g} |g| \tvar_g,
\end{equation} 
where $\stagger\orientation_{f,g}$ is such that $\stagger\orientation_{f,g} \stagger{\+n}_g$ points out of $\omega_f$.
The area of the face $g \subset \partial \omega_f$ is defined as
\begin{equation}
  |g| \defeq \half \sum_{f \in \mathcal{F}(g)} |f|,
\end{equation}
where $\mathcal{F}(g)$ is the set of faces $f \in \mathcal{F}$ neighbouring $g$ with the same face normal direction, as illustrated in~\cref{fig:notation:notation:set_stag}.
As with the standard divergence and gradient operators $\divh, \gradh$ we define staggered gradient operator $\stagger \gradh: \mathcal{F}^h \rightarrow \mathcal{G}^h$ as the negative adjoint of $\stagger \divh$. 
This results in
\begin{equation}\label{eqn:notation:notation:gradient_stag}
  (\stagger \gradh u)_g = -\frac{1}{\stagger h_g}\sum_{f \in \mathcal{F}^\omega(g)} \stagger\orientation_{f,g} u_f,
\end{equation}
where $\mathcal{F}^\omega(g)$ denotes the set of neighbouring staggered control volumes which have $g$ as part of their boundary, as shown in~\cref{fig:notation:notation:set_stag_omega}.
The distance $\stagger h_g$ is defined such that $\stagger \gradh$ is exact for linear functions.
Finally we denote by $\interpolantflux: \mathcal{F}^h \rightarrow \mathcal{G}^h$ the area weighted interpolant for interpolation of staggered fields to tensor valued functions
\begin{equation}\label{eqn:notation:notation:interpolantsbp_stag}
  |g|(\interpolantflux \massflux)_g \defeq \half \sum_{f\in\mathcal{F}(g)} (\stagger{\+n}_g \cdot \+n_f) |f| \massflux_f.
\end{equation}

A relation between all of the previously introduced operators, which will be important later on, is stated in~\cref{lem:notation:connection} and proven in~\cref{sec:app:connection}.
\begin{restatable}[Operator connection]{lemma}{lemmaoperatorconnection}\label{lem:notation:connection}
  For the previously introduced operators~\cref{eqn:notation:notation:interpolantsbp,eqn:notation:notation:interpolantsbp_stag,eqn:notation:notation:divergence,eqn:notation:notation:divergence_stag} it holds that
  \begin{eqnarray}\label{eqn:notation:notation:operator_connection}
    \stagger \divh \interpolantflux{} = \interpolantsbp \divh.
  \end{eqnarray}
\end{restatable}
  \section{Transport of mass}\label{sec:mass}
The equation that governs conservation of the centred mass can be obtained from~\cref{eqn:intro:conservation_eqn} by letting $\avar^\pi = 1$ and letting the control volume coincide with a centred control volume $\omega = c$
\begin{equation}\label{eqn:mass:exact_mass_cons}
  \integral{c^{\pi,(n+1)}}{\rho^\pi}{V} - \integral{c^{\pi,(n)}}{\rho^\pi}{V} + \sum_{f\in\mathcal{F}(c)}\oneDIntegral{t^{(n)}}{t^{(n+1)}}{\integral{f^\pi(t)}{\rho^\pi u^\pi}{S}}{t} = 0.
\end{equation}
Here we moreover replaced the integration over the boundary $\partial c \setminus I(t)$ by a sum of integrals over the $\pi$-phase part of the faces, which follows from
\begin{equation}
  \partial c^\pi \setminus I(t) = \bigcup_{f\in\mathcal{F}(c)} f^\pi(t).
\end{equation}
As briefly motivated in the introduction, we will use the geometric VOF method for the advection of the interface.
Hence the $d$-dimensional phase domain $\Omega^\pi \subset \Omega \subset \mathbb{R}^d$, for $\pi \in \{l, g\}$, will implicitly be represented by the volume fraction function $\volfrac^\pi \in \mathcal{C}^h$, where the volume fraction $\volfrac^\pi_c$ is defined as the fraction of volume in the control volume $c$ occupied by the $\pi$-phase
\begin{equation}
  \volfrac^\pi_c \defeq \frac{|c^\pi|}{|c|} \in [0, 1].
\end{equation}
Since the fluids are assumed incompressible we can factor out the constant density $\rho^\pi$ in~\cref{eqn:mass:exact_mass_cons}, resulting in
\begin{equation}\label{eqn:interface:mass_cons}
  |c|\volfrac^{\pi,(n+1)}_c - |c|\volfrac^{\pi,(n)}_c + \sum_{f\in\mathcal{F}(c)}\oneDIntegral{t^{(n)}}{t^{(n+1)}}{\integral{f^\pi(t)}{u^\pi_n}{S}}{t} = 0.
\end{equation}
This is a volumetric constraint on the advection of the interface, and in particular, summation of~\cref{eqn:interface:mass_cons} over all control volumes $c \in \mathcal{C}$ shows that the volume of each of the phases should be constant in time (provided an absence of in- or outflow at the boundary).
Adding the volumetric constraints for both phases results in (after letting the time step $\dt = t^{(n+1)} - t^{(n)}$ tend to zero)
\begin{equation}\label{eqn:interface:mix_divergence}
  \integral{\partial c}{u_n}{S} = 0,
\end{equation}
where the boundary integral contains contributions from the liquid as well as the gas velocity.

Sharp interface approximations of~\cref{eqn:interface:mass_cons} are obtained by a sharp approximation of the space-time integral, which often rely on the geometric intersection of an approximate {donating region} (DR)~\citep{Zhang2019} with an approximate phase domain representation $\Omega^\pi$.
The geometric VOF method traditionally relies on the piecewise linear approximation of the interface within each interface control volume, i.e.
\begin{equation}\label{eqn:interface:cutcell}
  c^l \defeq \set{\+x \in c}{\+\eta_c \cdot (\+x - \+x_c) + s(\+\eta_c; \volfrac^l_c) \le 0}, \quad c^g \defeq c \setminus c^l.
\end{equation}
We choose to approximate the interface normal $\+\eta \in [\mathcal{C}^h]^d$ by making use of local height-functions (LHFs)~\citep{gerrits2003dynamics,Veldman2007} when possible, and if this is not possible we resort to the efficient least-squares VOF interface reconstruction algorithm (ELVIRA)~\citep{Pilliod2004}.
Provided with the interface normal and volume fraction, the shift $s(\+\eta_c; \volfrac^l_c)$ is defined by ensuring that the volume of the reconstructed liquid part coincides with $|c| \volfrac^l_c$, which can be achieved using the methods described by~\citet{Scardovelli2000}.

In what follows we will discuss the geometric VOF method for the \onefluid formulation in detail, i.e. we discuss the discretisation of the space-time integral in~\cref{eqn:interface:mass_cons}.
We will first define some convenient notation for the oriented DR.
Then we derive sufficient conditions for the construction of approximate DRs which will ensure that the resulting volume fraction is bounded between zero and one.
Subsequently we will consider some DR approximation methods found in the literature in light of these conditions.

Provided with a discretisation of~\cref{eqn:interface:mass_cons}, the transport of the centred mass is modelled by denoting the mass per phase as $|c|\volfrac^{\pi,(n)}_c \rho^\pi$, whose evolution equation directly follows from the evolution of the volume fraction function $\volfrac^{\pi,(n)}$.
We will furthermore propose a generalisation of the interface advection method for the \twofluid model.

\subsection{The oriented donating region}\label{sec:interface:oriented_dr}
\begin{figure}
  \subcaptionbox{Absolute orientation.
  \label{fig:interface:dr:self_intersect:absolute}}
  [\twofigwidth]{
\begingroup%
  \makeatletter%
  \providecommand\color[2][]{%
    \errmessage{(Inkscape) Color is used for the text in Inkscape, but the package 'color.sty' is not loaded}%
    \renewcommand\color[2][]{}%
  }%
  \providecommand\transparent[1]{%
    \errmessage{(Inkscape) Transparency is used (non-zero) for the text in Inkscape, but the package 'transparent.sty' is not loaded}%
    \renewcommand\transparent[1]{}%
  }%
  \providecommand\rotatebox[2]{#2}%
  \newcommand*\fsize{\dimexpr\f@size pt\relax}%
  \newcommand*\lineheight[1]{\fontsize{\fsize}{#1\fsize}\selectfont}%
  \ifx\svgwidth\undefined%
    \setlength{\unitlength}{116.17691482bp}%
    \ifx\svgscale\undefined%
      \relax%
    \else%
      \setlength{\unitlength}{\unitlength * \real{\svgscale}}%
    \fi%
  \else%
    \setlength{\unitlength}{\svgwidth}%
  \fi%
  \global\let\svgwidth\undefined%
  \global\let\svgscale\undefined%
  \makeatother%
  \begin{picture}(1,0.78917499)%
    \lineheight{1}%
    \setlength\tabcolsep{0pt}%
    \put(0,0){\includegraphics[width=\unitlength,page=1]{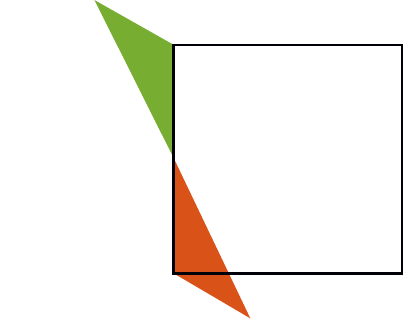}}%
    \put(-0.00357247,0.49370448){\color[rgb]{0.46666667,0.67843137,0.18823529}\makebox(0,0)[lt]{\lineheight{1.25}\smash{\begin{tabular}[t]{l}$\posflux{\drOne_f^{(n)}}$\end{tabular}}}}%
    \put(0.4383954,0.56245695){\color[rgb]{0,0,0}\makebox(0,0)[lt]{\lineheight{1.25}\smash{\begin{tabular}[t]{l}$f$\end{tabular}}}}%
    \put(0,0){\includegraphics[width=\unitlength,page=2]{dr_self_intersect.pdf}}%
    \put(0.46855603,0.43121753){\color[rgb]{0,0.44705882,0.74117647}\makebox(0,0)[lt]{\lineheight{1.25}\smash{\begin{tabular}[t]{l}$\+n_f$\end{tabular}}}}%
    \put(0.09586427,0.27642684){\color[rgb]{0.85098039,0.3254902,0.09803922}\makebox(0,0)[lt]{\lineheight{1.25}\smash{\begin{tabular}[t]{l}$\negflux{\drOne_f^{(n)}}$\end{tabular}}}}%
    \put(0,0){\includegraphics[width=\unitlength,page=3]{dr_self_intersect.pdf}}%
    \put(0.73082078,0.3294591){\color[rgb]{0,0,0}\makebox(0,0)[lt]{\lineheight{1.25}\smash{\begin{tabular}[t]{l}$\+x_c$\end{tabular}}}}%
  \end{picture}%
\endgroup%

  }
  \hfill 
  \subcaptionbox{Orientation relative to $c$.
  \label{fig:interface:dr:self_intersect:relative}}
  [\twofigwidth]{
    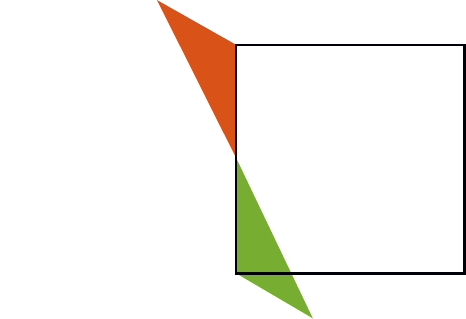
  }
  \caption{Illustration of an oriented and self-intersecting DR.
  The positively oriented part consists of the points that flow through the face $f$ in the direction of the face normal, whereas the negatively oriented part consists of points that flow through the face $f$ in the direction opposite to the face normal.
  When the face normal $\+n_f$ is considered we refer to the orientation of the part as the absolute orientation (left), and if the face normal $\orientation_{c,f}\+n_f$ is considered, we refer to it as the relative orientation instead (right).}
  \label{fig:interface:dr:self_intersect}
\end{figure}
For the moment we will assume that our velocity field is continuous, as is the case in the \onefluid model.
We emphasise this by denoting the velocity by $\onevelo \in \mathcal{F}^h$.
For the evaluation of the space-time integral in~\cref{eqn:interface:mass_cons} we will use a geometric intersection of a donating region with the phase domain.
The DR, which will be denoted by $\drOne^{(n)}_f$, is defined as the set of points for which the following equality holds
\begin{equation}\label{eqn:interface:dr:defining_property} 
  M_0(\drOne^{(n)}_f \cap \Omega^{\pi,(n)}) = \oneDIntegral{t^{(n)}}{t^{(n+1)}}{\integral{f^{\pi}(t)}{u^\pi_n}{S}}{t},
\end{equation}
where $M_0$ denotes the signed volume of a set, and will be defined in~\cref{eqn:interface:dr:signed_vol}.

The DR is endowed with an orientation such that a volume flux can be both positively as well as negatively contributing.
Note that a DR may self intersect, resulting in a DR which contains a positively as well as a negatively oriented part.
The orientation of a part of the DR $\drOne^{(n)}_f$ is defined as negative if the face normal $\+n_f$ points into it, see also~\cref{fig:interface:dr:self_intersect:absolute}.

Instead of separately having to keep track of each of the parts of the DR per face, we introduce the following notation: an oriented set $\dr$ is defined as a pair of non-oriented sets
\begin{equation}
  \dr = (\dr^+, \dr^-), \quad \dr^\pm \subset \mathbb{R}^d,
\end{equation}
where $\dr^+, \dr^-$ equals the positively and negatively oriented part respectively.
The signed volume is then defined as the difference between the unsigned volumes
\begin{equation}\label{eqn:interface:dr:signed_vol}
  M_0(\dr) \defeq |\dr^+| - |\dr^-|.
\end{equation}
Intersection of an oriented set $\dr$ with a non-oriented set $A$ is defined per oriented part
\begin{equation}
  \dr \cap A \defeq (\dr^+ \cap A, \dr^- \cap A),
\end{equation}
and the union, intersection as well as set difference of two oriented sets are defined `elementwise', e.g.
\begin{equation}
  \dr \cap \nabla \defeq (\dr^+ \cap \nabla^+, \dr^- \cap \nabla^-).
\end{equation}
Moreover, we allow the orientation of the parts of an oriented set to be swapped, which we denote as the multiplication by minus one (for two non-oriented sets $A, B$)
\begin{equation}\label{eqn:interfacee:dr:mult}
  -(A, B) \defeq (B, A).
\end{equation}
We refer to the orientation of $\orientation_{c,f}\drOne^{(n)}_f$ (recall that $\orientation_{c,f} = \pm 1$) as the orientation of the DR relative to the control volume $c$, or simply as the `relative orientation'.
This is illustrated in~\cref{fig:interface:dr:self_intersect:relative}.
Note that DRs of opposite relative orientation may overlap, resulting in a volume flux that cancels due to the opposite relative orientation, and which corresponds to fluid that merely passes through the control volume $c$: the fluid is `in transit' while in $c$ (see also~\cref{fig:interface:mass_transport}).

\begin{remark}[Definition of the donating region]
  Often the DR is defined as the set of points that are fluxed through the face $f$ during the time interval $[t^{(n)}, t^{(n+1)}]$
  \begin{equation}\label{eqn:interface:onefluid:dr}
    \drOne^{(n)}_f \defeq \set{\flowmap{-\tau}{t^{(n+1)}}f}{\tau \in [0, \dt]},
  \end{equation} 
  where the flow map $\flowmap{\tau}{t_0}$ is the solution operator of the following initial value problem
  \begin{equation}
    \frac{d}{dt} \+x(t) = \Onevelo(t, \+x(t)), \quad \+x(t_0) = \+x_0,
  \end{equation}
  such that $\flowmap{\tau}{t_0} \+x_0 = \+x(t_0 + \tau)$.

  This definition however does not in general satisfy~\cref{eqn:interface:dr:defining_property}, but it is sufficient for our discussion and we will therefore make use of it.
  For a precise definition we refer to~\citet{Zhang2019}.
\end{remark}

\subsection{Geometric VOF for the \onefluid model}\label{sec:interface:onefluid}
\begin{figure}
  \centering
  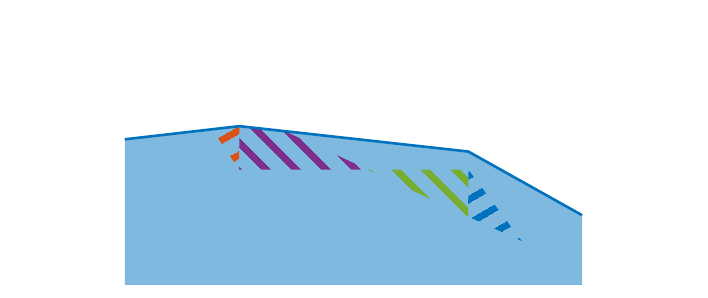
  \caption{Example of a self-intersecting DR (dashed), intersected with the neighbouring phase domain (shaded), resulting in the partial volume fluxes (hatched regions).
  The neighbouring control volumes are given by $\mathcal{C}^2(f) = \{c_1, \ldots, c_6\}$, and four out of six of the partial volume fluxes are nontrivial.}
  \label{fig:interface:onefluid:flux}
\end{figure}
Provided with the defining equation~\eqref{eqn:interface:dr:defining_property} of the DR, we can rewrite~\cref{eqn:interface:mass_cons} as follows
\begin{equation}\label{eqn:interface:onefluid:mass_cons}
  \frac{\volfrac^{\pi,(n+1)} - \volfrac^{\pi,(n)}}{\dt} + \divh \volumefluxOne^{\pi,(n)} = 0,
\end{equation}
where the divergence operator $\divh$ is as defined in~\cref{eqn:notation:notation:divergence}, and the signed volume flux $\volumefluxOne^{\pi,(n)} \in \mathcal{F}^h$ is given by the intersection volume of the DR with the phase domain
\begin{equation}\label{eqn:interface:onefluid:volflux}
  \dt |f|\volumefluxOne^{\pi,(n)}_f \defeq M_0(\drOne^{(n)}_f \cap \Omega^{\pi,(n)}).
\end{equation}
As the phase domain is approximated {piecewise} linearly, the volume flux is also computed in a piecewise manner
\begin{equation}\label{eqn:centred:volume_flux_per_cell_equivalence}
  \volumefluxOne^{\pi,(n)}_f = \sum_{b \in \mathcal{C}} \volumefluxOne^{\pi,(n)}_{f,b},
\end{equation}
where $|f|\volumefluxOne^{\pi,(n)}_{f,b}$ is the `partial volume flux' and results from the fluid that flows through the face $f$ and comes from the neighbouring control volume $b$. 
The partial volume flux is defined as
\begin{equation}\label{eqn:centred:volume_flux_per_cell}
  \dt|f|\volumefluxOne^{\pi,(n)}_{f,b} \defeq M_0(\drOne^{(n)}_f \cap b^{\pi,(n)}),
\end{equation}
as illustrated in~\cref{fig:interface:onefluid:flux}.
Note that under a suitable CFL constraint the use of the set $\mathcal{C}$ in~\cref{eqn:centred:volume_flux_per_cell_equivalence} can be replaced by $\mathcal{C}^2(f)$, which is the set of all control volumes that share at least one node with $f$ (i.e. $2 \times 3^{d-1}$ control volumes for a rectilinear mesh).

For the {approximate} advection of the interface we will approximate both terms on the right-hand side of~\cref{eqn:interface:onefluid:volflux}: the fluid domain is approximated using a piecewise linear reconstruction of the interface, whereas the DR is replaced by a polytopal approximation.
In what follows we will discuss types of so called `fluxing errors' that result from poorly approximated DRs, and subsequently we show that the absence of such fluxing errors results in a dimensionally unsplit advection method for which boundedness of the volume fraction can be guaranteed.
Moreover we discuss several advection methods from the literature, and check which of the fluxing errors are made.

\subsubsection{Fluxing errors in approximate donating regions}\label{sec:interface:onefluid:dr}
\begin{figure}
  \centering
  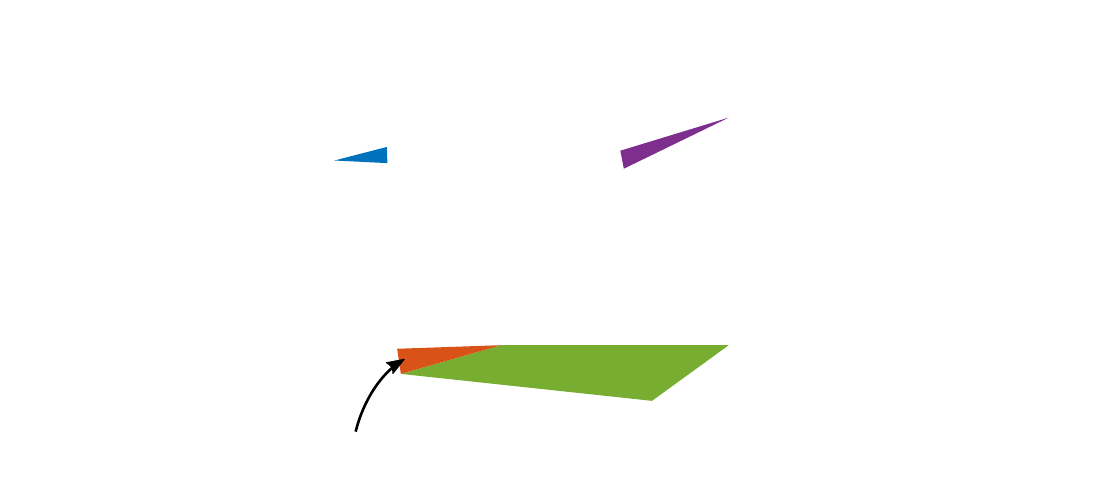
  \caption{Example of four two-dimensional donating regions (black dashed lines, each defined by four vertices) defined for some control volume $c$ (black solid line) resulting in four fluxing errors: the flux overlap error results from erroneously overlapping (i.e. of same relative orientation) DRs, the flux gap error results from a gap between two neighbouring DRs, and the flux transit error results from two neighbouring DRs with no overlap or gap, but which are defined using a non-matching corner position.
  Finally, it must be ensured that the volume of the DR matches the total volume flux according to $\dt |f| \onevelo_f$, as stated in~\cref{eqn:interface:onefluid:volume_enforcement}.
  Note that the orientation shown here is relative to the control volume $c$.
  }
  \label{fig:int:advect:dr:errors}
\end{figure}
When approximating a DR it is important to keep in mind that, besides the approximation errors, boundedness of the resulting volume fraction is of importance.
Some approximation errors may result in what we call fluxing errors, which for example lead to a piece of fluid that is fluxed twice during a single time step, potentially resulting in unboundedness of the volume fraction.
An absence of fluxing errors does not imply an absence of approximation errors, but we will show that an absence of fluxing errors results in guaranteed boundedness of the volume fraction.
This result will be formalised in~\cref{cor:interface:onefluid:boundedness}, which to the best of our knowledge is new, and can be used as a guide in the design of approximate DR methods.
We will now discuss four different fluxing errors, three of which will turn out to be essential for boundedness, and which are illustrated in~\cref{fig:int:advect:dr:errors}.

\paragraph{Flux overlap error}
Erroneously constructed DRs may result in overlap of two neighbouring DRs with equal relative orientation, as shown by the purple shaded region in~\cref{fig:int:advect:dr:errors}.
This implies that the corresponding fluid (if any) will leave the control volume $c$ twice, therefore potentially resulting in unboundedness of the volume fraction.
We refer to this as a flux overlap error, which is defined as
\begin{equation}\label{eqn:interface:onefluid:dr:overlap}
  \posflux{\orientation_{c,f}\drOne^{(n)}_f} \cap \posflux{\orientation_{c,g}\drOne^{(n)}_g} \neq \emptyset, \text{ for some } f \neq g \in \mathcal{F}(c),
\end{equation}
and similarly for the negatively oriented counterpart.

\paragraph{Flux gap error}
There may also be a gap between two neighbouring DRs, which implies that the fluid (if any) will erroneously remain inside the neighbouring control volume, potentially resulting in wisps~\citep{Tryggvason2011}, which are tiny fragments of fluid which are erroneously jettisoned from the bulk.
See e.g. the red shaded region in~\cref{fig:int:advect:dr:errors}.
Such an error is referred to as flux gap error.

\paragraph{Flux transit error}
While overlapping DRs are permissible if they have opposite relative orientation, care should be taken that they do not result in what we refer to as a flux transit error, shown in~\cref{fig:int:advect:dr:errors} as the blue shaded region.
Here a mismatch in the DR corner position results in fluid that leaves $c$ through $f_2$, but does not enter $c$ via $f_1$.
This can lead to unboundedness of the volume fraction.
A flux transit error (referred to as the existence of `nonconforming flux polyhedra' in~\citet{Ivey2017}) is defined as
\begin{equation}\label{eqn:interface:onefluid:dr:transit}
  \underbrace{\posflux{\orientation_{c,f}\drOne_f^{(n)}}}_\text{leaving} \setminus \bigcup_{g\in\mathcal{F}(c)} \underbrace{\negflux{\orientation_{c,g} \drOne_g^{(n)}}}_\text{entering} \not\subset c, \text{ for some } f \in \mathcal{F}(c).
\end{equation}

\paragraph{Flux volume error}
Finally, note that if~\cref{eqn:interface:onefluid:mass_cons} is summed over $\pi = l, g$ we find that
\begin{equation}\label{eqn:interface:err:div_vol}
  \mean{\volfrac}^{(n+1)} - \mean{\volfrac}^{(n)} + \divh\mean{\volumefluxOne}^{(n)} = 0 \quad\implies\quad \divh\mean{\volumefluxOne}^{(n)} = 0,
\end{equation}
since both phases together fill up a control volume $\mean{\volfrac} = 1$.
By definition of the volume flux in~\cref{eqn:interface:onefluid:volflux} we find that the sum of the volume fluxes equals the signed volume of the DR
\begin{equation}
  \dt|f|\mean{\volumefluxOne}^{(n)}_f = M_0(\drOne^{(n)}_f).
\end{equation}
and therefore the volume of the exact DR should equal the exact total volume flux
\begin{equation}
  M_0(\drOne^{(n)}_f) = \oneDIntegral{t^{(n)}}{t^{(n+1)}}{\integral{f}{\onevelo_n}{S}}{t} \approx \dt |f| \onevelo^{(n)}_f.
\end{equation}
Hence we will impose that the volume of the DR should equal the total volume flux~\citep{Lopez2004}
\begin{equation}\label{eqn:interface:onefluid:volume_enforcement}
  M_0(\drOne^{(n)}_f) \approx  \dt |f| \onevelo^{(n)}_f,
\end{equation}
such that we find that~\cref{eqn:interface:err:div_vol} holds automatically by the incompressibility constraint $\divh\onevelo^{(n)} = 0$.
Note that~\cref{eqn:interface:onefluid:volume_enforcement} approximately holds for exact DRs, but we impose it to hold exactly for approximate DRs.
The fourth error that we consider is the flux volume error, which occurs when the signed volume of the DR (as defined in~\cref{eqn:interface:dr:signed_vol}) does not match the total volume flux $\dt |f| \onevelo^{(n)}_f$.
This can be expressed as
\begin{equation}
  \mean{\volumefluxOne}^{(n)}_f \neq \onevelo^{(n)}_f.
\end{equation}

Imposing that all types of fluxing errors are absent does not leave a lot of freedom for constructing an approximate polytopal DR.
In~\cref{sec:int:onefluid:two:twoddr,sec:int:onefluid:two:threeddr} we will consider DR approximations found in the literature in light of the four aforementioned types of fluxing errors.
But first we will analyse how the fluxing errors relate to the boundedness of the volume fraction.

\subsubsection{Boundedness of the volume fraction}
The partial volume fluxes, which were defined in~\cref{eqn:centred:volume_flux_per_cell}, allow the interface advection equation~\eqref{eqn:interface:onefluid:mass_cons} to be restated as
\begin{equation}
  |c|(\volfrac^{\pi,(n+1)} - \volfrac^{\pi,(n)}) = - \dt\sum_{f \in \mathcal{F}(c)} \orientation_{c,f}|f|\sum_{b \in \mathcal{C}^2(f)} \volumefluxOne^{\pi,(n)}_{f,b} = |c|(V^{+,\pi}_c - V^{-,\pi}_c),
  \label{eqn:centred:advection_reorderedsummation}
\end{equation}
where the non-negative in- and outgoing volume through the boundary of the control volume $c$, denoted by $V^{+,\pi}_c$ and $V^{-,\pi}_c$ respectively, are defined as
\begin{equation}\label{eqn:centred:advection_inoutgoing_volume}
  V^{\pm,\pi}_c \defeq \pm \frac{\dt}{|c|} \sum_{b \in \mathcal{C}(c)} \squarepar{-\sum_{f \in \mathcal{F}(c)} \orientation_{c,f}|f| \volumefluxOne^{\pi,(n)}_{f,b}}^\pm.
\end{equation}
Here $\mathcal{C}(c)$ are the control volumes neighbouring $c$ that share at least one node with $c$ (i.e. $3^d$ control volumes if the mesh is rectilinear), and $\posflux{x}, \negflux{x}$ denote the positive and negative part respectively of $x$.
Note that we have swapped the order of summation in~\cref{eqn:centred:advection_reorderedsummation}, which implies that fluid that is merely in transit, and therefore results in two contributions of opposite sign, does not affect the in- and outgoing volumes.

The following theorem, which will also be crucial in~\cref{sec:momentum}, shows that bounded outgoing flow can be guaranteed provided that the approximate DRs do not contain some of the errors that were discussed in~\cref{sec:interface:onefluid:dr}.
A proof is found in~\cref{sec:app:dr_analysis}.
\begin{restatable}[Bounded outflow]{theorem}{lemmactuboundedoutflow}\label{thm:mass:bounded_outflow}
  The outgoing flow is bounded by the $\pi$-phase volume fraction contained in the control volume $c$ at $t = t^{(n)}$
  \begin{eqnarray}\label{eqn:centred:bounded_outflow}
    V^{-,\pi} \le \volfrac^{\pi,(n)},
  \end{eqnarray}
  provided that the approximate DRs do not contain any flux overlap nor transit errors.
\end{restatable}

Provided with the result of~\cref{thm:mass:bounded_outflow} we can now easily show that bounded outflow, combined with an absence of flux volume errors, results in boundedness of the volume fraction.
\begin{corollary}[Boundedness of the volume fraction]\label{cor:interface:onefluid:boundedness}
  Suppose that we advect the liquid phase.
  An absence of flux overlap and transit errors implies that the liquid volume fraction is bounded from below
  \begin{equation}\label{eqn:interface:properties:bound_below}
    \volfrac^{l,(n+1)} \ge V^{+,l} \ge 0.
  \end{equation}
  If additionally no flux volume errors are made then the liquid volume fraction is bounded from above as well
  \begin{equation}\label{eqn:interface:properties:bound_above}
    \volfrac^{l,(n+1)} \le 1 - V^{+,g} \le 1.
  \end{equation}
\end{corollary}
\begin{proof}
  The assumption that no flux overlap nor transit errors are made allows the use of~\cref{thm:mass:bounded_outflow} (with $\pi = l$),
  which implies that the non-negative ingoing flow is bounded by the phase volume at $t = t^{(n+1)}$
  \begin{equation}\label{eqn:interface:properties:intermediate}
    V^{+,l} \stackrel{\eqref{eqn:centred:advection_reorderedsummation}}{=} \volfrac^{l,(n+1)} - (\volfrac^{l,(n)} - V^{-,l}) \stackrel{\eqref{eqn:centred:bounded_outflow}}{\le} \volfrac^{l,(n+1)},
  \end{equation}
  and thus shows that~\cref{eqn:interface:properties:bound_below} holds.

  Satisfying~\cref{eqn:interface:onefluid:volume_enforcement} exactly, and thereby preventing flux volume errors to be made, ensures that implementing~\cref{eqn:interface:onefluid:mass_cons} for the liquid phase yields the same interface evolution as one would obtain for the gas phase, since
  \begin{align}\label{eqn:interface:onefluid:symmetry}
    \volfrac^{g,(n+1)} - \volfrac^{g,(n)} + \dt \divh\volumefluxOne^{g,(n)} &= (1-\volfrac^{l,(n+1)}) - (1-\volfrac^{l,(n)}) + \dt (\divh\mean{\volumefluxOne}^{(n)} -\divh\volumefluxOne^{l,(n)})\\
    &\stackrelwidth{\eqref{eqn:interface:err:div_vol}}{=}-\squarepar{\volfrac^{l,(n+1)} - \volfrac^{l,(n)} + \dt \divh\volumefluxOne^{l,(n)}}\\
    &\stackrelwidth{\eqref{eqn:interface:onefluid:mass_cons}}{=} 0.
  \end{align}
  Hence we can again make use of~\cref{thm:mass:bounded_outflow}, but now with $\pi = g$, to find that the gas volume fraction is bounded from below as well
  \begin{equation}
    \volfrac^{g,(n+1)} \ge {V^{+,g}},
  \end{equation}
  where we have used the same argument as we did in~\cref{eqn:interface:properties:intermediate}.
  We then use that the volume fractions add up to one, $\mean{\volfrac^{(n+1)}} = 1$, to find that the liquid volume fraction is indeed bounded from above
  \begin{equation}
    \volfrac^{l,(n+1)} = 1 - \volfrac^{g,(n+1)} \le 1 - {V^{+,g}}.
  \end{equation}
\end{proof}

We will now consider a few DR approximation methods from the literature, and discuss their properties regarding the different types of fluxing errors.
This then allows us, by making use of~\cref{cor:interface:onefluid:boundedness}, to guarantee boundedness of the volume fraction for some of these methods. 

\subsubsection{Two-dimensional approximate donating regions}\label{sec:int:onefluid:two:twoddr}
\begin{figure}
  \subcaptionbox{EMFPA.\label{fig:int:advect:dr:emfpa}}
  [\twofigwidth]{
    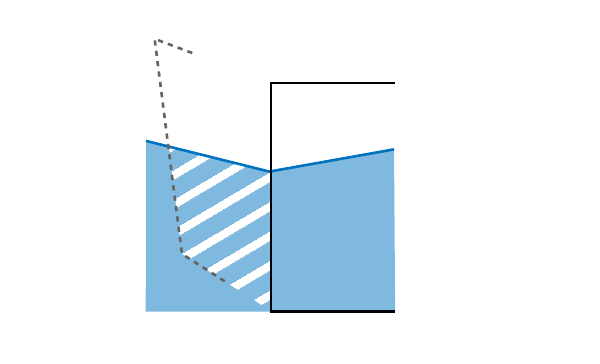
  }
  \subcaptionbox{MEMFPA.\label{fig:int:advect:dr:memfpa}}
  [\twofigwidth]{
    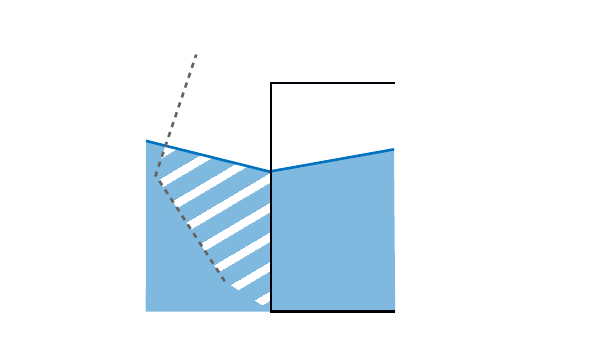
  }
  \caption{Example two-dimensional DRs with an exaggerated correction.
  The black dashed polygon corresponds to the DR without enforcing~\cref{eqn:interface:onefluid:volume_enforcement} (i.e. $\dt^* = 0$ in~\cref{eqn:int:advect:dr:emfpa_remap}), whereas the gray polygon corresponds to the correction.}
  \label{fig:int:advect:dr:twod}
\end{figure}
Let's first consider the naive construction of a two-dimensional donating region.
A face $f$ is defined by its two vertices $\+x_{v_1}, \+x_{v_2}$, for $v_1, v_2 \in \mathcal{V}(f)$, where $\mathcal{V}(f)$ are the nodes of $f$.
Using bilinear interpolation we approximate the velocity at each of the vertices, resulting in $\Onevelo^{(n)} \in [\mathcal{V}^h]^d$.
Simply using $\+x^*_v = \+x_v - \dt \Onevelo_v^{(n)}$ as the remapped (i.e. integrated backwards in time) vertices to complete the definition of the DR (see the black dashed polygon in~\cref{fig:int:advect:dr:emfpa}) would result in a DR which avoids the flux overlap, gap and transit errors (assuming a simple CFL condition is satisfied).
However a flux volume error is still made, and therefore the result of~\cref{cor:interface:onefluid:boundedness} does not guarantee that the resulting volume fraction is bounded from above.

The edge-matched flux polygon advection (EMFPA)~\citep{Lopez2004} method proposes to fix the aforementioned flux volume error by a novel modification of the DR.
They propose to shift the remapped face, i.e. the face of the DR which is defined by $\+x^*_{v_1}, \+x^*_{v_2}$, in such a way that the volume of the resulting DR matches the total volume flux, as stated in~\cref{eqn:interface:onefluid:volume_enforcement}.
The proposed modification yields the following definition of the remapped vertices
\begin{equation}\label{eqn:int:advect:dr:emfpa_remap}
  \+x^*_{v,f} = \+x_v - (\dt + \dt^*_{v,f}) \Onevelo_v^{(n)},
\end{equation}
where $\dt^*_{v,f}$ is a correction which is defined such that the remapped face tangent is invariant under this correction
\begin{equation}\label{eqn:int:advect:dr:emfpa_parallel}
  \+x^*_{v_2,f} - \+x^*_{v_1,f} \parallel \+x^*_{v_2} - \+x^*_{v_1},
\end{equation}
and such that~\cref{eqn:interface:onefluid:volume_enforcement} holds, see also~\cref{fig:int:advect:dr:emfpa}.
Hence per face $f$ we find two equations (given by~\cref{eqn:int:advect:dr:emfpa_parallel,eqn:interface:onefluid:volume_enforcement}) for the two unknown values of $\dt^*_{v,f}$.
This results in a scalar quadratic polynomial which has real roots except in some degenerate cases. 

The EMFPA method can however commit a flux transit error because $\+x^*_{v}$ is not uniquely defined: the position of the corner of the DR depends on which face $f$ is being considered.
It follows that the EMFPA method does not necessarily lead to bounded volume fractions (neither from below nor from above) since the assumptions of~\cref{cor:interface:onefluid:boundedness} do not hold.

In~\citet{Owkes2014} a modification to the EMFPA method is proposed.
Rather than imposing~\cref{eqn:interface:onefluid:volume_enforcement} by making use of~\cref{eqn:int:advect:dr:emfpa_remap}, a fifth vertex is added at the centroid of the remapped face
\begin{equation}\label{eqn:int:advect:dr:memfpa_enforce}
  \+x^*_f = \frac{1}{|\mathcal{V}(f)|}\sum_{v \in \mathcal{V}(f)} \+x^*_v + \dt^*_f \+n^*_f,
\end{equation}
where $\+n^*_f$ is the vector normal to $\+x^*_{v_2} - \+x^*_{v_1}$.
The signed area of the resulting polygon, consisting of five vertices, is then, up to a constant, a linear function of $\dt^*_f$ and therefore we can easily find a $\dt^*_f$ for which~\cref{eqn:interface:onefluid:volume_enforcement} holds.
A suitable $\dt^*_f$ exists if and only if $\+x^*_{v_1} \neq \+x^*_{v_2}$.
An example is shown in~\cref{fig:int:advect:dr:memfpa}.

This approach alleviates any flux transit errors because the remapped vertex is now uniquely defined (contrary to the remapped vertex, given by~\cref{eqn:int:advect:dr:emfpa_remap}, which was used in the EMFPA method).
Moreover it is still ensured that no flux volume errors are made by ensuring that~\cref{eqn:interface:onefluid:volume_enforcement} holds. 
Hence none of the discussed fluxing errors are made and therefore~\cref{cor:interface:onefluid:boundedness} ensures boundedness of the volume fraction.
There was no name provided to the methods proposed in~\citep{Owkes2014}, and therefore we will refer to both the 2D and 3D method as the modified EMFPA (MEMFPA) method.

\subsubsection{Three-dimensional approximate donating regions}\label{sec:int:onefluid:two:threeddr}
The face matched flux polyhedron advection (FMFPA) method~\citep{hernandez2008} is a 3D generalisation of the EMFPA method.
The additional challenge in 3D is to ensure that the faces of the approximate DR remain planar.
The FMFPA method succeeds in doing so, however this is at the cost of making flux overlap, gap as well as transit errors.
For this reason we use the 3D equivalent of the MEMFPA method instead, as proposed by~\citet{Owkes2014}.

In 3D, the resulting DR is in general no longer convex, and therefore the authors of~\citep{Owkes2014} propose a tetrahedralisation of the non-convex polyhedron such that routines suitable only for convex polyhedra can be used to compute the intersection volume~\cref{eqn:interface:onefluid:volflux}.
We instead use the VOFTools 5 toolbox~\citep{Lopez2020}, which was kindly provided to us by its authors.
This toolbox is able to perform intersections of non-convex polyhedra and therefore we do not need to tetrahedralise our polyhedron prior to intersection.

\subsubsection{The CFL constraint}
We will use the following CFL condition
\begin{equation}\label{eqn:interface:onefluid:dr:cfl}
  \max_{c\in\mathcal{C}} \mathcfl_c < \wycflval,
\end{equation}
where the CFL limit is usually set to $\wycflval = \frac{3}{4}$ and $\mathcfl_c$ denotes the local control volume CFL number, which is defined as
\begin{equation}\label{eqn:interface:onefluid:dr:cfl_cell}
  \mathcfl_c \defeq \frac{\dt}{|c|}\sum_{f\in\mathcal{F}(c)} \posflux{-\orientation_{c,f}|f| \onevelo^{(n)}_f}.
\end{equation}
 
If a DR does not self-intersect, and hence has only one part of a single orientation, then the volume flux is bounded by the velocity $\abs{\volumefluxOne^\pi} \le \abs{\onevelo}$.
The CFL condition~\cref{eqn:interface:onefluid:dr:cfl} then ensures that the change in volume fraction does not exceed $\mathcfl_c$
\begin{equation}\label{eqn:interface:onefluid:dr:cfl_bound}
  \pm(\volfrac^{\pi,(n+1)}_c - \volfrac^{\pi,(n)}_c) = \frac{\dt}{|c|}\sum_{f \in \mathcal{F}(c)} \mp\orientation_{c,f} |f| \volumefluxOne^{\pi,(n)}_f \le \pm\frac{\dt}{|c|}\sum_{f \in \mathcal{F}(c)} \squarepar{-\orientation_{c,f} |f| \volumefluxOne^{\pi,(n)}_f}^\pm \le \mathcfl_c,
\end{equation}
which follows from~\cref{eqn:interface:onefluid:mass_cons}.
These bounds do not always hold (a DR is permitted to self-intersect), but it illustrates the usefulness of the CFL condition~\eqref{eqn:interface:onefluid:dr:cfl}.
We find that in practise the bound~\cref{eqn:interface:onefluid:dr:cfl_bound} holds almost always, and always holds when replacing $\mathcfl_c$ by $\wycflval$.

\subsection{Application to mass transport (\onefluid model)}
Provided with a discretisation of~\cref{eqn:interface:mass_cons} we can now formulate the resulting equation which governs mass transport.
We define the mass flux $\massfluxOne^\pi \in \mathcal{F}^h$ as follows
\begin{equation}\label{eqn:mass:onefluid:massflux_def}
  \massfluxOne^\pi \defeq \rho^\pi \volumefluxOne^\pi.
\end{equation}
We find that conservation of the centred mass is given by
\begin{equation}\label{eqn:mass:onefluid:mass_cons}
  \frac{(\volfrac\rho)^{\pi,(n+1)} - (\volfrac\rho)^{\pi,(n)}}{\dt} + \advection{\massfluxOne^\pi}1 = 0,
\end{equation}
which follows from multiplying~\cref{eqn:interface:onefluid:mass_cons} by the constant $\rho^\pi$.
Here we have defined the advection operator $\advection{m}: \mathcal{C}^h \rightarrow \mathcal{C}^h$ as 
\begin{equation}\label{eqn:mass:advection_def}
  \advection{m}\avar \defeq \divh(m \fluxinterp \avar),
\end{equation}
where $\fluxinterp: \mathcal{C}^h \rightarrow \mathcal{F}^h$ denotes the flux interpolant for which the staggered equivalent will be defined in~\cref{sec:momentum}.
For now it is sufficient to know that this flux interpolant will interpolate a constant field exactly: $\fluxinterp 1 = 1$.

\subsection{Geometric VOF for the \twofluid model}\label{sec:interface:twofluid}
We will now consider the advection of the interface in the presence of a velocity discontinuity. 
That is, we have two staggered velocity fields $u^l, u^g \in \mathcal{F}^h$ where a value $u_f^\pi$ exists if and only if the corresponding staggered volume fraction $\volfracstag^\pi_f \in \mathcal{F}^h$ is non-zero. 
The staggered volume fraction $\volfracstag^\pi \in \mathcal{F}^h$ is defined as the volume weighted average of the centred volume fraction
\begin{equation}\label{eqn:interface:twofluid:volfracstag}
  \volfracstag^\pi \defeq \interpolantsbp\volfrac^\pi.
\end{equation}

The two velocity fields are continuous in the interface normal direction (see~\cref{eqn:intro:normal_smoothness}), and the movement of the interface depends only on this interface normal component. 
Therefore we are free to use either the liquid velocity, gas velocity, or a linear combination thereof for the advection of the interface. 
We choose to advect the interface using the liquid velocity field, corresponding to the heaviest of the two phases, and track the liquid volume fraction field.

\begin{figure}
  \subcaptionbox{Definition of the face aperture.\label{fig:interface:twofluid:cutcell_aperture}}
  [\twofigwidth]{
\begingroup%
  \makeatletter%
  \providecommand\color[2][]{%
    \errmessage{(Inkscape) Color is used for the text in Inkscape, but the package 'color.sty' is not loaded}%
    \renewcommand\color[2][]{}%
  }%
  \providecommand\transparent[1]{%
    \errmessage{(Inkscape) Transparency is used (non-zero) for the text in Inkscape, but the package 'transparent.sty' is not loaded}%
    \renewcommand\transparent[1]{}%
  }%
  \providecommand\rotatebox[2]{#2}%
  \newcommand*\fsize{\dimexpr\f@size pt\relax}%
  \newcommand*\lineheight[1]{\fontsize{\fsize}{#1\fsize}\selectfont}%
  \ifx\svgwidth\undefined%
    \setlength{\unitlength}{132.40209698bp}%
    \ifx\svgscale\undefined%
      \relax%
    \else%
      \setlength{\unitlength}{\unitlength * \real{\svgscale}}%
    \fi%
  \else%
    \setlength{\unitlength}{\svgwidth}%
  \fi%
  \global\let\svgwidth\undefined%
  \global\let\svgscale\undefined%
  \makeatother%
  \begin{picture}(1,0.50693777)%
    \lineheight{1}%
    \setlength\tabcolsep{0pt}%
    \put(0,0){\includegraphics[width=\unitlength,page=1]{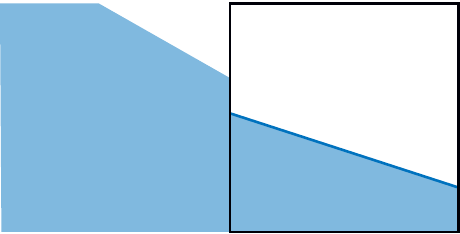}}%
    \put(0.55800033,0.38402477){\color[rgb]{0,0,0.10980392}\makebox(0,0)[lt]{\lineheight{1.25}\smash{\begin{tabular}[t]{l}$f^g$\end{tabular}}}}%
    \put(0,0){\includegraphics[width=\unitlength,page=2]{cutcell_aperture.pdf}}%
    \put(0.38562297,0.13654514){\color[rgb]{0,0,0.10980392}\makebox(0,0)[lt]{\lineheight{1.25}\smash{\begin{tabular}[t]{l}$f^l$\end{tabular}}}}%
    \put(0.40350435,0.42550979){\color[rgb]{0,0.44705882,0.74509804}\makebox(0,0)[lt]{\lineheight{1.25}\smash{\begin{tabular}[t]{l}$\+\eta_{c_1}$\end{tabular}}}}%
    \put(0.85697869,0.17946067){\color[rgb]{0,0.44705882,0.74509804}\makebox(0,0)[lt]{\lineheight{1.25}\smash{\begin{tabular}[t]{l}$\+\eta_{c_2}$\end{tabular}}}}%
  \end{picture}%
\endgroup%

  }
  \hfill
  \subcaptionbox{Boundary integral split into two parts.\label{fig:interface:twofluid:cutcell_divergence}}
  [\twofigwidth]{
\begingroup%
  \makeatletter%
  \providecommand\color[2][]{%
    \errmessage{(Inkscape) Color is used for the text in Inkscape, but the package 'color.sty' is not loaded}%
    \renewcommand\color[2][]{}%
  }%
  \providecommand\transparent[1]{%
    \errmessage{(Inkscape) Transparency is used (non-zero) for the text in Inkscape, but the package 'transparent.sty' is not loaded}%
    \renewcommand\transparent[1]{}%
  }%
  \providecommand\rotatebox[2]{#2}%
  \newcommand*\fsize{\dimexpr\f@size pt\relax}%
  \newcommand*\lineheight[1]{\fontsize{\fsize}{#1\fsize}\selectfont}%
  \ifx\svgwidth\undefined%
    \setlength{\unitlength}{67.23459213bp}%
    \ifx\svgscale\undefined%
      \relax%
    \else%
      \setlength{\unitlength}{\unitlength * \real{\svgscale}}%
    \fi%
  \else%
    \setlength{\unitlength}{\svgwidth}%
  \fi%
  \global\let\svgwidth\undefined%
  \global\let\svgscale\undefined%
  \makeatother%
  \begin{picture}(1,1.00000017)%
    \lineheight{1}%
    \setlength\tabcolsep{0pt}%
    \put(0,0){\includegraphics[width=\unitlength,page=1]{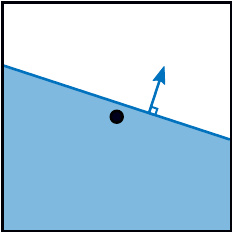}}%
    \put(0.71325294,0.5592951){\color[rgb]{0,0.44705882,0.74509804}\makebox(0,0)[lt]{\lineheight{1.25}\smash{\begin{tabular}[t]{l}$\+\eta_{c}$\end{tabular}}}}%
    \put(0,0){\includegraphics[width=\unitlength,page=2]{cutcell_divergence.pdf}}%
    \put(0.06049272,0.14162407){\color[rgb]{0.85490196,0.3254902,0.09803922}\makebox(0,0)[lt]{\lineheight{1.25}\smash{\begin{tabular}[t]{l}$\integral{\partial c^l \setminus I}{u^l_n}{S}$\end{tabular}}}}%
    \put(0.13997184,0.82765447){\color[rgb]{0.49411765,0.18431373,0.55686275}\makebox(0,0)[lt]{\lineheight{1.25}\smash{\begin{tabular}[t]{l}$\integral{\partial c^g \setminus I}{u^g_n}{S}$\end{tabular}}}}%
    \put(11.0319699,-8.75232437){\color[rgb]{0,0,0}\makebox(0,0)[lt]{\begin{minipage}{6.15543658\unitlength}\raggedright \end{minipage}}}%
  \end{picture}%
\endgroup%

  }
  \caption{Illustration of the cut-cell method applied to the approximation of the divergence constraint given by~\cref{eqn:interface:twofluid:cutcell:div_approx}.}
  \label{fig:interface:twofluid:cutcell}
\end{figure}
The velocities $u^l, u^g$ are divergence free in the sense of~\cref{eqn:interface:mix_divergence}.
The divergence constraint~\cref{eqn:interface:mix_divergence} can be written as a sum of two boundary integrals, where we sharply distinguish the integration over the liquid and gaseous parts of the boundary
\begin{equation}
  \integral{\partial c}{u_n}{S} = \integral{\partial c^l \setminus I}{u^l_n}{S} + \integral{\partial c^g \setminus I}{u^g_n}{S}.
\end{equation}
Hence for the sharp discretisation of the boundary integral we need to sharply identify which part of a face $f \subset \partial c$ is inside the liquid, and which part is in the gas phase.
To this end we introduce the face aperture $a^\pi \in \mathcal{F}^h$ which equals the fraction of the face that is occupied by the $\pi$-phase
\begin{equation}
  a^{\pi,(n)}_f \defeq \frac{|f \cap \Omega^{\pi,(n)}|}{|f|},
\end{equation}
with\footnote{Our proposed methods allow for the presence of arbitrary geometry, which is modelled implicitly using the cut-cell method~\citep{Kleefsman2005,Droge2007,Crockett2011}.
This means that in such cut-cells we find $\mean{a^{(n)}} < 1$, as the remainder of the face is filled by the geometry.
Throughout this paper it will however be assumed that each control volume consists entirely of liquid and/or gas for ease of discussion.} $\mean{a^{(n)}} = 1$.
Note that since a {piecewise} reconstruction of the interface is used, the face aperture is averaged from the two approximations from each side of the face, see also~\cref{fig:interface:twofluid:cutcell_aperture}.

The face apertures allow for the following cut-cell~\citep{Crockett2011} approximation of~\cref{eqn:interface:mix_divergence}
\begin{equation}\label{eqn:interface:twofluid:cutcell:div_approx}
  |c|\divh(\mean{a u})_c = 0,
\end{equation}
which will be used in a future paper to impose the divergence constraint on the velocity field.
The cut-cell divergence operator is illustrated in~\cref{fig:interface:twofluid:cutcell_divergence}.
Using this divergence constraint implies that the divergence per phase does not vanish at the interface: $\divh u^\pi \neq 0$, and we can therefore not directly re-use the previously discussed interface advection method, because an absence of flux volume errors cannot be guaranteed, resulting in a loss of boundedness according to~\cref{cor:interface:onefluid:boundedness}.

Instead, we propose to solve the following additional Poisson problem, which projects the liquid velocity field $u^l$ into the space of divergence free velocity fields
\begin{equation}
  \def\arraystretch{1.35}
  \left\lbrace
  \begin{array}{rlll}
    \widehat{u}_f^{l,(n)} &= u_f^{l,(n)} - (\gradh\widehat{p})_f & \forall f \in \mathcal{F} &\text{for which $\volfracstag^{l,(n)}_f > 0$} \\
    (\divh\widehat{u}^{l,(n)})_c &= 0 & \forall c \in \mathcal{C} &\text{for which $\volfrac^{l,(n)}_c > 0$}\\
    \widehat{p}_c &= 0 & \forall c \in \mathcal{C} &\text{for which $\volfrac^{l,(n)}_c = 0$}
  \end{array}\right.,
\end{equation}
resulting in the `extrapolated' velocity field $\widehat{u}^{l,(n)}$.
Recall that the gradient operator $\gradh$ was defined in~\cref{eqn:notation:notation:gradient}.
We will now show that the use of this Poisson problem is sufficient for guaranteeing boundedness of the volume fraction. 
If the centred volume fraction is non-zero $\volfrac^{l,(n)}_c > 0$ then $(\divh\widehat{u})_c^{l,(n)} = 0$ and therefore no flux volume errors are made, resulting in boundedness of the volume fraction according to~\cref{cor:interface:onefluid:boundedness}.
On the other hand, if the centred volume fraction is initially zero $\volfrac^{l,(n)}_c = 0$, then flux volume errors are made because the velocity field is not divergence free (note that missing velocities whose staggered volume fraction is zero are computed using constant extrapolation), but~\cref{thm:mass:bounded_outflow} still guarantees that the volume fraction will be bounded from below.
Moreover, by the CFL constraint~\cref{eqn:interface:onefluid:dr:cfl,eqn:interface:onefluid:dr:cfl_cell} as well as~\cref{eqn:interface:onefluid:dr:cfl_bound} we find that $\volfrac^{\pi,(n+1)}_c \le \wycflval \le 1$.

We can therefore apply the dimensionally unsplit interface advection method, as discussed in~\cref{sec:interface:onefluid}, where we use the extrapolated velocity field $\widehat{u}^{l,(n)}$.
The resulting DRs are denoted as $\drTwo^{l,(n)}_f$.
From these DRs we define the following liquid volume flux (cf.~\cref{eqn:interface:onefluid:volflux})
\begin{equation}
  \dt |f|\volumefluxTwo^{l,(n)}_f \defeq M_0(\drTwo^{l,(n)}_f \cap \Omega^{l,(n)}),
\end{equation}
resulting in the following interface advection equation (cf.~\cref{eqn:interface:onefluid:mass_cons})
\begin{equation}\label{eqn:interface:twofluid:mass_cons}
  \frac{\volfrac^{l,(n+1)} - \volfrac^{l,(n)}}{\dt} + \divh\volumefluxTwo^{l,(n)} = 0.
\end{equation}

\subsection{Application to mass transport (\twofluid model)}\label{sec:mass:twofluid}
For the \twofluid formulation we find that the transport of the centred liquid mass is given by
\begin{equation}\label{eqn:mass:twofluid:mass_cons}
  \frac{(\volfrac\rho)^{l,(n+1)} - (\volfrac\rho)^{l,(n)}}{\dt} + \advection{\massfluxTwo^l}1 = 0,
\end{equation}
which follows from multiplying~\cref{eqn:interface:twofluid:mass_cons} by $\rho^l$.
Here the liquid mass flux is defined as
\begin{equation}
  \massfluxTwo^l \defeq \rho^l \volumefluxTwo^l.
\end{equation}
Note that the divergence free liquid velocity $\widehat{u}^{l,(n)}$ is available only inside the liquid phase (only on those faces $f$ for which $\volfracstag_f^{l,(n)}>0$), and therefore the gas mass fluxes inside the gas phase cannot be obtained using $\widehat{u}^{l,(n)}$.
On the other hand we do know the gas mass inside the gas phase
\begin{equation}\label{eqn:mass:twofluid:gas_mass}
  (\volfrac\rho)^{g,(n+1)} = (1-\volfrac^{l,(n+1)})\rho^g,
\end{equation}
which is conserved as a consequence of $\volfrac^{l,(n+1)}$ being conserved.

For the transport of momentum, which will be discussed in~\cref{sec:momentum}, we will need to know not only the gas mass, but also the gas mass fluxes inside the gas phase, and to this end we define the gas mass flux as
\begin{equation}
  \massflux^g \defeq \rho^g \volumeflux^g,
\end{equation}
where the volume fluxes $\volumeflux^g$ follow from using the DR method inside the gas phase with the gas velocity field $u^g$.
Hence on faces at the interface we compute two DRs: one DR is constructed using the divergence free liquid velocity $\widehat{u}^l$, resulting in the liquid mass fluxes $\massfluxTwo^l$, and another DR is constructed using the gas velocity $u^g$, resulting in the gas mass fluxes $\massflux^g$.
The resulting centred mass, according to the gas mass fluxes, is denoted by $(\volfrac \rho)^{*,\pi}$, for which
\begin{equation}\label{eqn:mass:twofluid:mass_cons_gas}
  \frac{(\volfrac\rho)^{g,*} - (\volfrac\rho)^{g,(n)}}{\dt} + \advection{\massflux^g}1 = 0.
\end{equation}
The gas velocity field is not divergence free (contrary to $\widehat{u}^l$), but if we do not enforce the volume of the resulting DRs, and thereby permit flux volume errors, then this is not a problem: the presence of flux volume errors means that the volume fraction may exceed one, but it is still ensured to be bounded from below by zero according to~\cref{cor:interface:onefluid:boundedness}.
Note that $\volfrac^{g,*}$ as well as $1-\volfrac^{l,(n+1)}$ are consistent estimates of the gas volume fraction at $t = t^{(n+1)}$.
However only the latter, by which the interface is defined, is bounded from above as well as below.
  \section{Transport of momentum}\label{sec:momentum}
We will now focus our attention on the conservative advection of some (possibly discontinuous) staggered (i.e. defined at the faces of the control volume) scalar $\avar$, which is modelled by~\cref{eqn:intro:conservation_eqn}.
The control volume now coincides with the staggered control volume $\omega = \omega_f$, resulting in
\begin{equation}\label{eqn:momentum:scalar_advec}
  \integral{\omega_f^{\pi,(n+1)}}{\rho^\pi \avar^{\pi}}{V} - \integral{\omega_f^{\pi,(n)}}{\rho^\pi \avar^{\pi}}{V} + \sum_{g\in\mathcal{G}(\omega_f)}\oneDIntegral{t^{(n)}}{t^{(n+1)}}{\integral{g^\pi(t)}{\rho^\pi u_n^\pi \avar^\pi }{S}}{t} = 0.
\end{equation}
We denote the staggered \emph{advected} quantity by $\avarstag^\pi \in \mathcal{F}^h$ to emphasise the distinction with the \emph{advecting} velocity field $u^\pi$.
For momentum transport in the \onefluid formulation we let $\avarstag^\pi = \onevelo$, and in the \twofluid formulation we use $\avarstag^\pi = u^\pi$.
We will first consider the \onefluid formulation.

An approximation of~\cref{eqn:momentum:scalar_advec} consists of the approximation of both the volume integral of $(\rho\avar)^\pi$ as well as the space-time integral of $\rho^\pi\onevelo_n \avar^\pi$.
We choose to approximate the former by the product of the integration volume (given by $|\omega_f| \volfracstag_f^{\pi,(n)}$) with the value of $(\rho\avarstag)^\pi$ in the control volume centroid
\begin{equation}\label{eqn:momentum:int_approx}
  |\omega_f|(\volfracstag\rho\avarstag)_f^{\pi,(n)} \approx \integral{\omega_f^{\pi,(n)}}{(\rho\avar)^{\pi,(n)}}{V},
\end{equation}
where $\avarstag^\pi \in \mathcal{F}^h$ is such that
\begin{equation}
  \avarstag^{\pi,(n)}_f \approx \avar^\pi(t^{(n)}, \+x_f),
\end{equation}
and $\volfracstag$ denotes the staggered volume fraction function as defined in~\cref{eqn:interface:twofluid:volfracstag}.
The space-time integral is similarly approximated as the product of the mass flux $\massfluxOneStag^{\pi,(n)} \in \mathcal{G}^h$ (recall that $\mathcal{G}$ is the set of faces of the staggered control volumes) with the value of $\avar^\pi$ at the centroid of the DR (cf.~\cref{eqn:interface:dr:defining_property})
\begin{equation}\label{eqn:momentum:flux_approx}
  \dt |f|\massfluxOneStag_g^{\pi,(n)} (\fluxinterpstag\avarstag^{\pi,(n)})_g \approx \oneDIntegral{t^{(n)}}{t^{(n+1)}}{\integral{g^\pi(t)}{\rho^\pi \onevelo_n \avar^\pi }{S}}{t},
\end{equation}
where $\fluxinterpstag: \mathcal{F}^h \rightarrow \mathcal{G}^h$ is the flux interpolant which interpolates to the centroid of the DR (see~\cref{fig:momentum:flux_interp:example})
\begin{equation}\label{eqn:momentum:unsplit_dr:approx_position_flux}
  (\fluxinterpstag\avarstag^{\pi,(n)})_g \approx \avar^{\pi}(t^{(n)}, \+C^{(n)}_g).
\end{equation}
The DR, which will be denoted by $\drOneStag{}^{(n)}_g$, is now based on the face $g$ of a staggered control volume $\omega_f$, and is formally defined such that (cf.~\cref{eqn:interface:dr:defining_property})
\begin{equation}
  M_0(\drOneStag{}^{(n)}_g \cap \Omega^{\pi,(n)}) = \oneDIntegral{t^{(n)}}{t^{(n+1)}}{\integral{g^{\pi}(t)}{\onevelo_n}{S}}{t},
\end{equation}
holds.
The centroid of the DR $\drOneStag{}^{(n)}_g$ is denoted by $\+C^{(n)}_g$.

In what follows we discuss the approximation of the newly introduced terms in~\cref{eqn:momentum:flux_approx}: the mass flux $\massfluxOneStag$ as well as the flux interpolant $\fluxinterpstag$.
We will then first apply this to the \onefluid formulation, and subsequently generalise the proposed method to the \twofluid formulation.

\subsection{Mass flux computation}
There are several approaches to compute the staggered mass flux, or equivalently, the staggered volume flux.
Originally,~\citet{Rudman1998} proposed a dimensionally split advection method wherein the volume fraction field was defined on a refined grid: every control volume was split into $2^d$ control volumes.
Subsequently the mass fluxes were geometrically computed on the faces of each of the refined control volumes.
Provided with the mass fluxes on a refined grid, one can compute the mass fluxes on the original grid by simply adding the corresponding refined mass fluxes.
In this way the mass fluxes of the centred as well as staggered control volumes can all be computed directly from the refined mass fluxes.
A similar approach was followed by~\citet{Zuzio2020} for a dimensionally split advection method, and by~\citet{Owkes2017} for a dimensionally unsplit advection method.

Alternatively one can construct DRs directly on the faces of the staggered control volume $\omega_f$, and define the mass fluxes for the staggered control volume in this way, resulting in
\begin{equation}
  \dt|g|\stagger\massfluxOne^\pi_g = \rho^\pi M_0(\drOneStag{}^{(n)}_g \cap \Omega^{\pi,(n)}).
\end{equation}
This would result in the following staggered advection of $\avarstag^\pi$
\begin{equation}\label{eqn:momentum:geometric}
  \frac{(\volfracstag\rho)^{\pi,*}\avarstag^{\pi,(n+1)} - (\volfracstag\rho\avarstag)^{\pi,(n)}}{\dt} + \stagger \divh\roundpar{\stagger\massfluxOne^{\pi,(n)}\fluxinterpstag{}\avarstag^{\pi,(n)}} = 0,
\end{equation}
where the staggered divergence operator $\stagger \divh$ is as defined in~\cref{eqn:notation:notation:divergence_stag}.
This approach is followed by~\citet{Arrufat2021} where it is applied to a dimensionally split advection method.
Note that in following this approach, the staggered volume fraction must be defined geometrically from the intersection of the staggered control volume with the phase domain.
The downside of this approach is that the staggered mass, which follows from letting $\avarstag = 1$ in~\cref{eqn:momentum:geometric}, will become out of sync with the centred mass which itself evolves according to~\cref{eqn:mass:onefluid:mass_cons}. 
This means that the staggered mass must be reset after every time step, resulting in a loss of momentum conservation.

We instead propose a third alternative which is based on simply averaging the mass fluxes (recall that the interpolant $\interpolantflux$ is as defined in~\cref{eqn:notation:notation:interpolantsbp_stag})
\begin{equation}\label{eqn:momentum:onefluid:staggered_massflux}
  \massfluxOneStag^\pi \defeq \interpolantflux\massfluxOne^\pi,
\end{equation}
for which the staggered advection equation is given by
\begin{equation}\label{eqn:momentum:onefluid:transport}
  \frac{(\volfracstag\rho\avarstag)^{\pi,(n+1)} - (\volfracstag\rho\avarstag)^{\pi,(n)}}{\dt} + \advectionstag{\massfluxOne^{\pi,(n)}}\avarstag^{\pi,(n)} = 0,
\end{equation}
where we define the staggered advection operator $\advectionstag{\massflux}: \mathcal{F}^h \rightarrow \mathcal{F}^h$ as
\begin{equation}\label{eqn:mass:advection_stagdef}
  \advectionstag{\massflux} \avarstag \defeq \stagger \divh\roundpar{(\interpolantflux \massflux) (\fluxinterpstag \avarstag)}.
\end{equation}
In~\cref{sec:app:quadratic} we show that for this advection operator the corresponding semi-discrete formulation preserves quadratic invariants, such as kinetic energy, provided that the LW flux interpolant (which will be introduced next) is used.

This approach is reminiscent of the flux interpolation used in the cut-cell method from~\citet{Droge2005}. 
Therein a symmetry preserving convection operator near a solid boundary (rather than a moving phase interface) is proposed where the volume fluxes used in the convection operator are simply averaged from the volume fluxes that are defined on the faces of the centred control volumes.
The advantage of doing so, is that the divergence operator, which is based on the volume fluxes of the centred control volumes, shows up in the convection operator, resulting in a skew-symmetric operator if the volume fluxes are divergence free.
Simply averaging the volume fluxes may seem inconsistent at first, but a geometric interpretation is provided in~\citep{Droge2005,Cheny2010}.

Recall from~\caref{eqn:notation:notation:operator_connection} that the centred and staggered divergence operators are related, and therefore (moreovoer using~\cref{eqn:mass:advection_def,eqn:mass:advection_stagdef})
\begin{equation}\label{eqn:mass:onefluid:advection_property}
  \advectionstag{m}1 = \interpolantsbp{}(\advection{m}1), \quad \forall m \in \mathcal{F}^h.
\end{equation}
This implies that the staggered mass transport equation (which follows from substituting $\avarstag = 1$ in~\cref{eqn:momentum:onefluid:transport}) can also be written as
\begin{equation}\label{eqn:mass:onefluid:mass_cons_stag}
  \frac{(\volfracstag\rho)^{\pi,(n+1)} - (\volfracstag\rho)^{\pi,(n)}}{\dt} + \interpolantsbp{}(\advection{\massfluxOne^{\pi,(n)}}1) = 0,
\end{equation}
and therefore the evolution equation for the staggered mass can equivalently be obtained from application of the interpolation operator $\interpolantsbp{}$ to~\cref{eqn:mass:onefluid:mass_cons}.
This means that, unlike the method proposed in~\citet{Arrufat2021}, we need not reset the staggered mass after every time step to ensure that it remains in sync with the centred mass, thereby obtaining exact conservation of linear momentum, without the need for subgrid mass fluxes as was used by e.g.~\citet{Rudman1998}.

\subsection{Flux interpolation}
All that remains for the discretisation of~\cref{eqn:momentum:scalar_advec} is to define the flux interpolant, which interpolates the value of $\avarstag^\pi$ to the centroid of the DR for each of the faces of a control volume.
We let the centroid of the DR be approximated by
\begin{equation}\label{eqn:momentum:dr_centroid}
  \+C^{(n)}_g \defeq \+x_g - \frac{\dt}{2} \Onevelo^{(n)}_g,
\end{equation}
where $\Onevelo^{(n)} \in [\mathcal{G}^h]^d$ denotes an interpolated vector-valued velocity.

From~\cref{eqn:momentum:onefluid:transport} it follows that the value of $\avarstag^{\pi,(n+1)}$ can be computed as follows
\begin{equation}\label{eqn:centred:approx_dr_solve}
  \avarstag^{\pi,(n+1)} = \frac{(\volfracstag\rho\avarstag)^{\pi,(n)} - \dt \advectionstag{\massfluxOne^{\pi,(n)}}\avarstag^{\pi,(n)}}{(\volfracstag\rho)^{\pi,(n+1)}}.
\end{equation}
Hence if a staggered control volume is drained of the $\pi$-phase during a single time step, i.e. $\volfracstag^{\pi,(n+1)}_f \rightarrow 0$, we must somehow guarantee that the numerator in the right-hand side of~\cref{eqn:centred:approx_dr_solve} tends to zero as well to make sure that $\avarstag^{\pi,(n+1)}_f$ is well-defined.
In terms of the transport of momentum this relates to ensuring that the velocity can be obtained from the momentum and mass without any problems.
We will refer to this as the `boundedness property' of the flux interpolant

We will first discuss the construction of the flux interpolant away from the interface, that is, where we can guarantee that $\volfracstag^{\pi,(n+1)}_f$ does not vanish.
Then we will consider a different (more dissipative) flux interpolant which is to be used at the interface, and is guaranteed to be bounded.
Finally we will switch between both flux interpolants in such a way that we only use the more dissipative flux interpolant when necessary, while still guaranteeing boundedness.

\subsubsection{Flux interpolant away from the interface}\label{sec:momentum:unsplit_flux_interp}
Away from the interface we need not worry about boundedness, and therefore the construction of the flux interpolant is focused on obtaining a second-order accurate approximation of $\avar^\pi$ at the centroid of the DR.
The flux interpolant that we propose is constructed in two steps: first we compute one downwind and two upwind approximations of $\avar^\pi$ that all lie on a single line containing the centroid of the DR, then we interpolate along this line to obtain a second-order approximation of $\avar^\pi$ at the centroid of the DR.

\begin{figure}
  \centering
  \def\fluid{}
  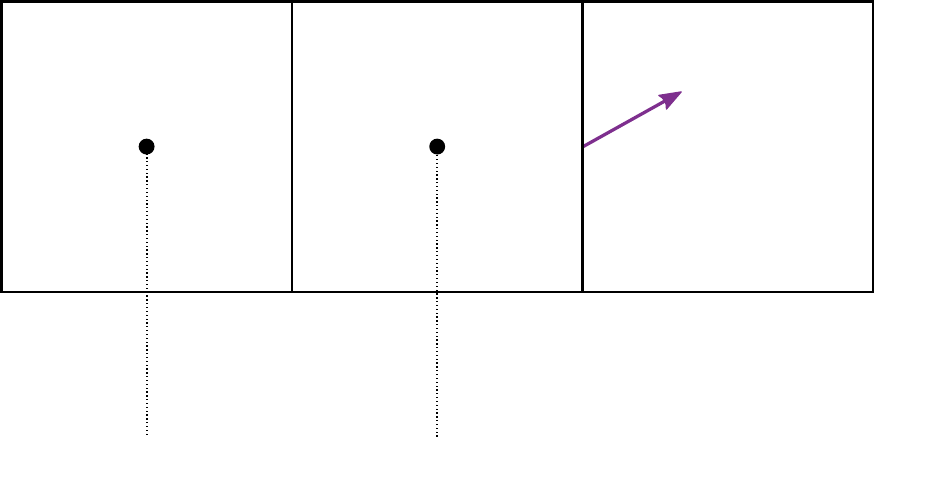
  \caption{Example of a DR (black dashed lines) defined on some face $g \in \mathcal{G}(\omega_f)$.
  The downwind interpolant $\tanginterp_0$ simply yields the value downwind of the face $g$, whereas the $i$-th upwind interpolant $\tanginterp_i$ is defined as the interpolant that linearly interpolates along the $i$-th black dotted line to the point of intersection with the red dashed line, for $i = 1, 2$.
  The red dashed interpolation line is defined by the downwind node $\+x_f$ as well as the centroid $\+C^{(n)}_g$ of the DR.}
  \label{fig:momentum:flux_interp:example}
\end{figure}
We let the downwind approximation, which we denote by $(\tanginterp_0\avarstag^\pi)_g$, coincide with the value $\avarstag_f^\pi$, where $f \in \mathcal{F}^\omega(g)$ (the set $\mathcal{F}^\omega(g)$ is illustrated in~\cref{fig:notation:notation:set_stag_omega}) corresponds to the staggered control volume $\omega_f$ which is downwind from $g$, as illustrated in~\cref{fig:momentum:flux_interp:example}.
Since the centroid of the DR is assumed to lie on the interpolation line, it follows that $(\tanginterp_0\+x)_g$ and $\+C^{(n)}_g$ entirely define the interpolation line as indicated by the dashed line in~\cref{fig:momentum:flux_interp:example}.
The first upwind value $(\tanginterp_1\avarstag^\pi)_g$ is obtained via (bi)linear interpolation using the two (four in 3D) neighbouring values of $\avarstag^{\pi}$ that lie in the plane with normal $\+n_g$ and passes through the centroid of the upwind control volume.
The linear interpolation is indicated by the dotted vertical lines in~\cref{fig:momentum:flux_interp:example}.
The second upwind value $(\tanginterp_2\avarstag^\pi)_g$ is computed similarly (when needed).
Note that it may happen that the upwind value $(\tanginterp_2\avarstag^\pi)_g$ is missing because the corresponding staggered control volumes do not contain the $\pi$-phase, in this case we resort to using only $(\tanginterp_1\avarstag^\pi)_g$.

We are now ready to perform the second interpolation step.
The simplest second-order accurate approximation to the value of $\avarstag^\pi$ at the centroid of the DR is given by (assuming a uniform mesh)
\begin{equation}\label{eqn:centred:flux_interp:interp_lw}
  \stagger\fluxinterp^{\mtext{LW}}\avarstag \defeq \frac{1}{2}\roundpar{1 - \mathcfl} \tanginterp_0\avarstag + \frac{1}{2}\roundpar{1 + \mathcfl}\tanginterp_1\avarstag,
\end{equation}
where $\mathcfl^{(n)} \in \mathcal{G}^h$ denotes the CFL number per face and is given by
\begin{equation}\label{eqn:interface:onefluid:dr:cfl_face}
  \mathcfl^{(n)} \defeq \frac{\dt \interpolantflux\onevelo^{(n)}}{\stagger h}.
\end{equation}
\Cref{eqn:centred:flux_interp:interp_lw} can be seen as a multidimensional generalisation of the Lax--Wendroff (LW) flux interpolant (see e.g.~\citet{Leveque2004}).
The second upwind value can be included in the approximation in the following way
\begin{equation}\label{eqn:centred:unsplit:hi_res_interpolant}
  \stagger\fluxinterp^{\mtext{Method}}\avarstag \defeq \tanginterp_1^{}\avarstag + \frac{1}{2}\roundpar{1 - \mathcfl} \Psi^\mtext{Method}\roundpar{\iota} (\tanginterp_0^{}\avarstag - \tanginterp_1^{}\avarstag),
\end{equation}
where $\Psi^\mtext{Method}\roundpar{\iota}$ denotes the flux limiter function which is a function of the ratio of the slopes
\begin{equation}
  \iota = \frac{\tanginterp_1\avarstag - \tanginterp_2\avarstag}{\tanginterp_0\avarstag - \tanginterp_1\avarstag} \frac{\abs{\tanginterp_0\+x - \tanginterp_1\+x}_2}{\abs{\tanginterp_1\+x - \tanginterp_2\+x}_2}.
\end{equation}
The following flux limiter functions~\citep{Leveque2004} are considered
\begin{align}
  \Psi^\mtext{LW}(\iota) &\defeq 1\\
  \Psi^\mtext{Fromm}(\iota) &\defeq \frac{1 + \iota}{2}\\
  \Psi^\mtext{MC}(\iota) &\defeq \max\roundpar{0, \min\roundpar{\Psi^\mtext{Fromm}(\iota), 2, 2\iota}}\\
  \Psi^\mtext{Upwind}(\iota) &\defeq 0,
\end{align}
where the LW flux simply interpolates between $\tanginterp_0\avarstag$ and $\tanginterp_1\avarstag$ and coincides with \cref{eqn:centred:flux_interp:interp_lw}, the Fromm flux uses the average of the central and upwind gradient, the monotonised central (MC) flux limits the Fromm flux to the TVD region of the Sweby diagram~\citep{Sweby1984} (or normalised variable diagram) and the upwind flux simply uses the upwind value $\tanginterp_1\avarstag$.
Note that only the upwind and MC fluxes are really limited.

\subsubsection{Flux interpolant at the interface}\label{sec:momentum:flux_interp_ctu}
Using the previously introduced flux interpolant \cref{eqn:centred:unsplit:hi_res_interpolant} near the interface may easily result in an unbounded value of $\avarstag^\pi$ when the control volume is drained of the $\pi$-phase, even if the upwind interpolant is used.

One of the culprits of such unboundedness is the fact that fluid that is merely in transit still affects the final value of $\avarstag^\pi$.
This could be prevented by explicitly computing the volume of the fluid that is in transit and introducing a diagonal flux which bypasses the control volume $c$ altogether.
However this requires intersection of neighbouring DRs, which would be cumbersome and expensive.

To this end we now introduce the corner transport upwind (CTU) flux interpolant, for which the resulting advection method is reminiscent of the CTU method~\citep{Leveque1996a}.
The CTU flux interpolant is defined as follows
\begin{equation}\label{eqn:centred:ctu_flux}
  (\stagger\fluxinterp^\mtext{CTU} \avarstag^{\pi,(n)})_g \defeq \frac{\sum_{k \in \mathcal{F}^2(g)} \volumefluxOneStag^{\pi,(n)}_{g,k}\avarstag^{\pi,(n)}_k}{\volumefluxOneStag^{\pi,(n)}_g},
\end{equation}
where the volume fluxes $\volumefluxOneStag^{\pi,(n)} \in \mathcal{G}^h$ are interpolated as (cf.~\cref{eqn:momentum:onefluid:staggered_massflux})
\begin{equation}\label{eqn:centred:volfluxstag}
  \volumefluxOneStag^{\pi,(n)} \defeq \interpolantflux \volumefluxOne^{\pi,(n)}, 
\end{equation}
and $\mathcal{F}^2(g) \subset \mathcal{F}$ denotes the set of faces whose corresponding staggered control volumes share at least one vertex with the face $g \in \mathcal{G}$ (hence resulting in $2 \times 3^{d-1}$ faces on a rectilinear mesh).
The partial volume fluxes (from~\cref{eqn:centred:volume_flux_per_cell}) are interpolated in a similar fashion.
This results in a $\avarstag^\pi$-flux which is simply the sum over all the signed areas multiplied by their respective $\avarstag^\pi$-values.
In particular this means that if neighbouring DRs overlap with opposite relative orientation, then their combined contribution does not contain any contribution from this overlapping region.

The following lemma shows that the CTU flux interpolant always results in boundedness, under the assumption that the DRs do not contain any flux overlap nor transit errors.
Here the set $\mathcal{F}(f)$ denotes those faces whose corresponding staggered control volume shares at least one vertex with $\omega_f$ (hence resulting in $3^d$ neighbouring faces on a rectilinear mesh).
\begin{restatable}[Boundedness of the CTU flux interpolant]{lemma}{lemmactuboundedness}\label{lem:momentum:ctu_boundedness}
  The CTU flux interpolant~\cref{eqn:centred:ctu_flux} results in boundedness of $\avarstag^{\pi,(n+1)}$ in the following sense
  \begin{eqnarray*}
    \abs{\avarstag_f^{\pi,(n+1)} - \avarstag_f^{\pi,(n)}} \le \max_{k\in\mathcal{F}(f)} \abs{\avarstag^{\pi,(n)}_k - \avarstag_f^{\pi,(n)}},
  \end{eqnarray*}
  provided that the approximate DRs do not contain any flux overlap nor transit errors, and that the CFL constraint given by~\cref{eqn:interface:onefluid:dr:cfl} is satisfied for $\wycflval \le 1$.
\end{restatable}
A proof is found in~\cref{sec:app:dr_analysis}.
Note that no such result can be shown without making use of the partial volume fluxes, since for the usual volume fluxes $\volumefluxOne_f^{\pi,(n)}$ no equivalent of~\cref{thm:mass:bounded_outflow} can be shown due to the presence of fluid which is merely in transit.
It is exactly this fluid that is in transit that cancels out due to swapping the order of summation in~\cref{eqn:centred:advection_reorderedsummation}, and this cancellation allows for~\cref{thm:mass:bounded_outflow}, and thus~\cref{lem:momentum:ctu_boundedness}, to be proven.


\subsubsection{The modified flux interpolant}\label{sec:centred:flux_interp_mod}
Using the CTU flux interpolant introduced previously alleviates any issues regarding boundedness, provided that no flux overlap and transit errors are committed, but it is a rather dissipative and only first-order accurate interpolant.
To this end we will now introduce the modified flux interpolant, which switches between the second-order accurate interpolant from~\cref{sec:momentum:unsplit_flux_interp} and the CTU flux interpolant from~\cref{sec:momentum:flux_interp_ctu} in a way that the latter is only used whenever it is needed to ensure boundedness.
The modified flux interpolant is defined as
\begin{equation}\label{eqn:centred:unsplit:modified}
  (\stagger\fluxinterp^{\mtext{Method}^*} \avarstag)^{\pi,(n)}_g \defeq \left\lbrace\begin{array}{rl}
    (\stagger\fluxinterp^\mtext{CTU} \avarstag)^{\pi,(n)}_g & \text{if } \min_{f \in \mathcal{F}^\omega(g)}\volfracstag^{\pi,(n+1)}_f < \minfracval\\
    (\stagger\fluxinterp^{\mtext{Method}} \avarstag)^{\pi,(n)}_g & \text{otherwise}
  \end{array}\right.,
\end{equation}
for some parameter $\minfracval > 0$.
If a staggered control volume $\omega_f$ results in $\volfracstag^{\pi,(n+1)}_f < \minfracval$, then each of the modified interpolant fluxes through the boundary faces $g \in \mathcal{G}(\omega_f)$ will reduce to the CTU flux, resulting in boundedness via~\cref{lem:momentum:ctu_boundedness}.
And if a staggered control volumes results in $\volfracstag^{\pi,(n+1)}_f \ge \minfracval$, then the following result shows that the modified interpolant still yields some form of boundedness.
\begin{restatable}[Boundedness of the modified flux interpolant]{lemma}{lemmamodboundedness}\label{lem:momentum:mod_boundedness}
  Assuming that $\volfracstag^{\pi,(n+1)}_f \ge \minfracval$, we find that the temporal update in $\avarstag^\pi$ is bounded in the following sense
  \begin{eqnarray*}
    \abs{\avarstag^{\pi,(n+1)}_f - \avarstag^{\pi,(n)}_f} \le \frac{\stagger W^{+,\pi}_f + \stagger W^{-,\pi}_f}{\minfracval} \max_{g \in \mathcal{G}(\omega_f)} \abs{(\fluxinterpstag^{\mtext{Method}^*} \avarstag^{\pi,(n)})_g - \avarstag^{\pi,(n)}_f},
  \end{eqnarray*}
  where the in- and outgoing volume, without re-ordering the summation over the partial volume fluxes, is given by
  \begin{eqnarray*}
    \stagger W^{\pm,\pi}_f \defeq \pm\frac{\dt}{|\omega_f|}\sum_{g\in\mathcal{G}(\omega_f)} \squarepar{-\stagger\orientation_{f,g} |g|{\volumefluxOneStag_g^{\pi,(n)}}}^\pm.
  \end{eqnarray*}
\end{restatable}

A proof is found in~\cref{sec:app:dr_analysis}.
By not using the partial volume fluxes we cannot re-order the summation over the partial volume fluxes, and therefore fluid that is merely in transit does not cancel (as was the case with the previously discussed CTU flux interpolant). 
This prohibits the possibility of obtaining a bound for $\stagger W^{+,\pi}_f$ in terms of $\volfracstag^{\pi,(n+1)}_f$, and therefore we obtain an upper bound which is inversely proportional to the assumed lower bound of $\volfracstag^{\pi,(n+1)}_f$ (denoted by $\minfracval$).
Moreover, if each of the DRs do not self intersect then we can directly make use of~\cref{eqn:interface:onefluid:dr:cfl_bound}, resulting in ${\stagger W^{\pm,\pi}} \le \interpolantsbp\mathcfl$, such that the resulting bound from~\cref{lem:momentum:mod_boundedness} can be somewhat simplified
\begin{equation}\label{eqn:momentum:unsplit:modified_simplified}
  \abs{\avarstag^{\pi,(n+1)}_f - \avarstag^{\pi,(n)}_f} \le \frac{2 (\interpolantsbp\mathcfl)_f}{\minfracval}\max_{g \in \mathcal{G}(\omega_f)} \abs{(\fluxinterpstag^{\mtext{Method}^*} \avarstag^{\pi,(n)})_g - \avarstag^{\pi,(n)}_f}.
\end{equation}
In practise this bound holds with $\interpolantsbp\mathcfl$ replaced by $\wycflval$, since DRs are permitted to self intersect.
We will use $\minfracval = \half$ (unless indicated otherwise), hence the CTU flux interpolant is used if a staggered control volume becomes less than half full.

\subsection{The \onefluid formulation}
When the \onefluid formulation is considered we assume continuity of the velocity field.
However, the momentum transport method proposed here leads to two velocity fields to be defined at the interface after each time step (we let $\avarstag^{\pi} = \onevelo^{}$ in~\cref{eqn:momentum:onefluid:transport})
\begin{equation}
  u^{\pi,(n+1)} = \frac{(\volfracstag \rho)^{\pi,(n)} \onevelo^{(n)} - \dt \advectionstag{\massfluxOne^{\pi,(n)}}\onevelo^{(n)}}{(\volfracstag\rho)^{\pi,(n+1)}},
\end{equation}
which in general will not be equal $u^{l,(n+1)} \neq u^{g,(n+1)}$.
Here we follow~\citet{Arrufat2021} and choose to define a single continuous velocity as the sum (over the phases) of the momentums divided by the sum of the masses
\begin{equation}\label{eqn:momentum:onefluid}
  \onevelo^{(n+1)} = \frac{\mean{\volfracstag \rho u}^{(n+1)}}{\mean{\volfracstag \rho}^{(n+1)}},
\end{equation}
which results in
\begin{equation}
  \frac{(\mean{\volfracstag \rho}\onevelo)^{(n+1)} - (\mean{\volfracstag \rho} \onevelo)^{(n)}}{\dt} + \mean{\advectionstag{\massfluxOne}\onevelo}^{(n)} = 0,
\end{equation}
and thus yields exact conservation of the sum (over the phases) of the total linear momentum.

\subsection{The \twofluid formulation}
We will use the mass fluxes $\massfluxTwo^l, \massflux^g$, as defined in~\cref{sec:mass:twofluid}, for the transport of mass and momentum in the \twofluid formulation, for the liquid and gas phase respectively.
For the liquid phase this results in the advection method
\begin{equation}\label{eqn:momentum:liq_twofluid}
  \frac{(\volfracstag\rho u)^{l,(n+1)} - (\volfracstag\rho u)^{l,(n)}}{\dt} + \advection{\massfluxTwo^{l,(n)}} u^{l,(n)} = 0.
\end{equation}
The transport of the staggered mass follows from application of the interpolant $\interpolantsbp$ to the centred mass transport~\cref{eqn:mass:twofluid:mass_cons}, and can equivalently be obtained from~\cref{eqn:momentum:liq_twofluid} by letting $u^l \equiv 1$ by making use of~\cref{eqn:mass:onefluid:advection_property}.
It follows that liquid momentum is exactly conserved as the staggered mass need not be redefined.

On the other hand, transport of gas momentum is achieved using the mass fluxes $\massflux^g$, resulting in
\begin{equation}\label{eqn:momentum:gas_twofluid}
  \frac{\volfracstag^{g,*}(\rho u)^{g,(n+1)} - (\volfracstag\rho u)^{g,(n)}}{\dt} + \advection{\massflux^{g,(n)}} u^{g,(n)} = 0.
\end{equation}
Contrary to the transport of liquid mass and momentum, letting $u^g \equiv 1$ in~\cref{eqn:momentum:gas_twofluid} now does not yield the same staggered gas mass transport as one would obtain using the interpolant $\interpolantsbp$ applied to~\cref{eqn:mass:twofluid:gas_mass}.
Instead, as previously discussed in~\cref{sec:mass:twofluid}, letting $u^g \equiv 1$ in~\cref{eqn:momentum:gas_twofluid} yields the same staggered gas mass transport as one would obtain when applying the interpolant $\interpolantsbp$ to~\cref{eqn:mass:twofluid:mass_cons_gas}, which uses the same gas mass fluxes $\massflux^g$ as used in~\cref{eqn:momentum:gas_twofluid}.
For obtaining a bounded value of the gas velocity we must divide the gas momentum by the gas mass which results from the same gas mass fluxes $\massflux^g$, hence the one defined in~\cref{eqn:mass:twofluid:mass_cons_gas}.
Hence we will reset the gas mass according to~\cref{eqn:mass:twofluid:gas_mass} after every time step, resulting in a loss of momentum conservation in the gas phase.
We deem this to be acceptable because the gas density is much smaller than that of the liquid $\rho^g \ll \rho^l$, and therefore the loss or gain of momentum is comparatively small.

\subsection{Discussion}
\begin{table}
  \centering
  \begin{tabular}{lllllll}
    \toprule
    & \rotatebox{90}{Unsplit?} & \rotatebox{90}{Conservative?} & \rotatebox{90}{\# volume fluxes} & \rotatebox{90}{\# reconstructions} & \rotatebox{90}{\# velocities} & Flux interpolation\\ \midrule
    \citet{Rudman1998}            & N & $l+g$ & $d2^{d-1}$  & $d(2^d-1)$ & 1 & Flux limiting \\ 
    \citet{Zuzio2020}             & N & $l+g$ & $d2^{d-1}$  & $0$        & 1 & WENO5\\ 
    \citet{Arrufat2021}           & N &       & $d^2$       & $d^2$      & 1 & QUICK/Superbee \\ 
    \citet{Owkes2017}             & Y & $l+g$ & $d2^{d-1}$  & ?          & 1 & CTU near interface \\ 
    \textbf{Proposed \onefluid}   & Y & $l+g$ & $0$         & 0          & 1 & CTU {if needed} \\ \midrule
    \textbf{Proposed \twofluid}   & Y & $l$   & $d$         & 0          & 2 & CTU {if needed} \\ \bottomrule
  \end{tabular}
  \caption{Overview of considered methods for the transport of a staggered momentum field, where the bold-faced methods are proposed in this paper.
  Conservation of `$l + g$' means that the sum of liquid and gas momentum is conserved, whereas `$l$' means that the liquid momentum is conserved (and the gas momentum is not).
  We show the \emph{additional} number of volume fluxes and interface reconstructions that are needed per centred interface control volume for the transport of momentum (excluding those needed for the advection of the interface: $d$ for a split method and $1$ for an unsplit method).}
  \label{tab:momentum:overview}
\end{table}
We give an overview of momentum transport methods found in the literature, restricted to staggered formulations using a DR-based geometric VOF method, in~\cref{tab:momentum:overview}.
Our proposed advection method for both the one- and \twofluid formulation has been included.
All of the methods under consideration conserve the liquid as well as gas mass.
Note that for the split methods often no arguments regarding boundedness were made in the respective references, which is most likely because for a dimensionally split method, a simple upwind flux corresponds to what the CTU flux would give, and therefore sufficient limiting implies boundedness.
In~\citet{Arrufat2021} the authors mention the use of a very low CFL number ($0.03$) and they still report the `blow up' of the solution, which is likely to be related to not taking boundedness into account.

For the dimensionally unsplit method from~\citet{Owkes2017} we find that the CTU flux interpolant is used in a neighbourhood around the interface.
This ensures boundedness, but we find that using it everywhere near the interface results in an unnecessarily dissipative flux interpolant, and we would always suggest using the modified flux interpolant as proposed in~\cref{sec:centred:flux_interp_mod} to ensure that the CTU flux interpolant is used only as much as required for ensuring boundedness.

Our proposed use of interpolated mass (or equivalently, volume) fluxes results in (to the best of our knowledge) the only, dimensionally split or unsplit, staggered advection method which does not require any additional volume flux computation. 
Moreover the staggered advection method results in exact mass and momentum conservation (for the \onefluid formulation), sharply models the interface and ensures that quadratic invariants are preserved in the semi-discrete limit $\dt \rightarrow 0$ if the LW flux interpolant is used (see~\cref{sec:app:quadratic}).

  \section{Validation}\label{sec:validation}
We will now validate the proposed staggered advection methods for both the one- and \twofluid formulation.
Since we have thus far not discussed how we impose~\cref{eqn:intro:normal_smoothness}, we will use a single and therefore continuous advecting velocity field which will be denoted by $\Onevelo$.
This means that for now we will use $u^g = u^l = \onevelo$ as the advecting velocity field when the \twofluid formulation is considered.
In a future paper, where the discretisation of~\cref{eqn:intro:normal_smoothness} will be discussed, we will consider a multitude of test cases for which $\jump{u} \neq 0$.

The advected field, which we will again denote by $\avarstag^\pi \in \mathcal{F}^h$, will be discontinuous to mimic the presence of an unresolved shear layer
\begin{equation}\label{eqn:validation:advection:field} 
  \avarstag_x^l(t_0, \+x) = \prod_{i=1}^d \sin(4\pi x_i),\quad
  \avarstag_x^g(t_0, \+x) = \prod_{i=1}^d \cos(2\pi x_i),
\end{equation}
where the remaining $y, z$ components are zero.
The advection of this staggered field, for the \onefluid formulation, is performed as if it were momentum: after each time step the two fields are merged according to~\cref{eqn:momentum:onefluid}, where $\ratio{\rho} = 10^{-3}$. 

At all times we will use the modified flux interpolant as given by~\cref{eqn:centred:unsplit:modified}.

\begin{figure}
  \subcaptionbox{\Twofluid (LW).\label{fig:validation:advection:example:lw2}}
  [\threefigwidth]{
    \includegraphics[scale=0.24,trim=0 150 0 0,clip=true]{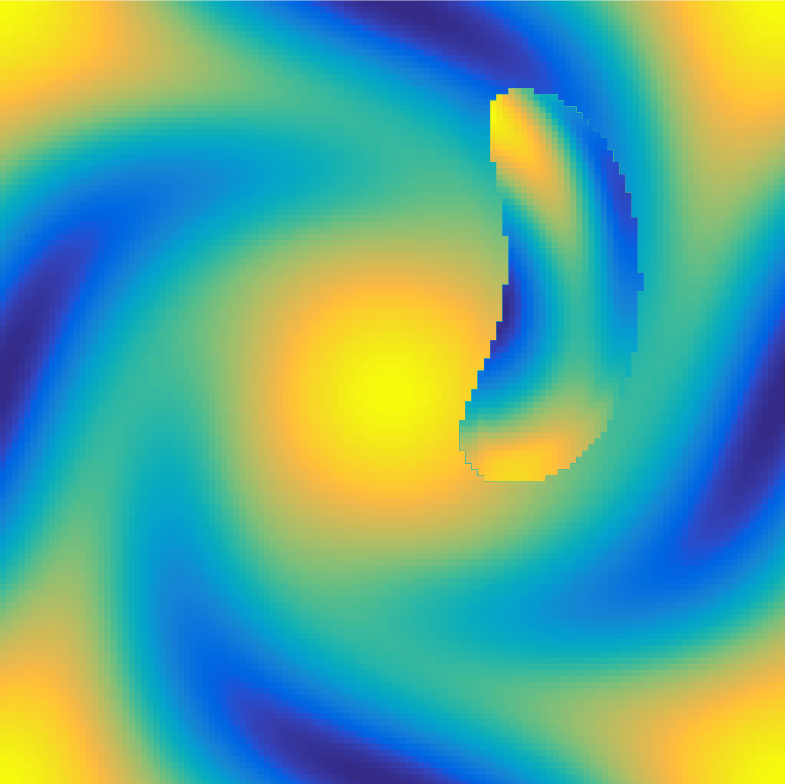}
  }\hfill
  \subcaptionbox{\Onefluid (LW).\label{fig:validation:advection:example:lw1}}
  [\threefigwidth]{
    \includegraphics[scale=0.24,trim=0 150 0 0,clip=true]{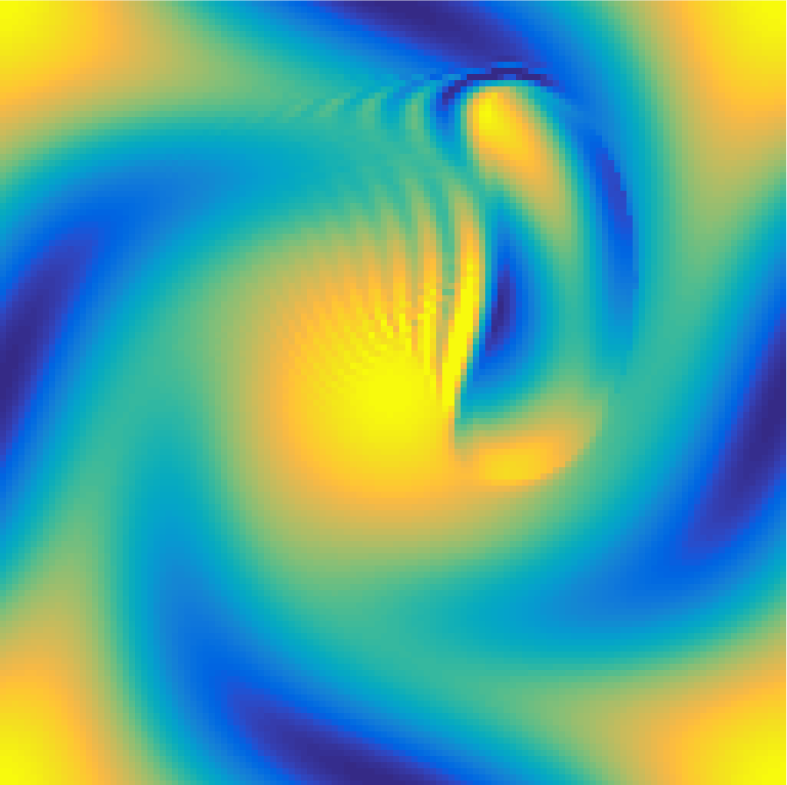}
  }\hfill
  \subcaptionbox{\Onefluid (Fromm).\label{fig:validation:advection:example:fromm}}
  [\threefigwidth]{
    \includegraphics[scale=0.24,trim=0 150 0 0,clip=true]{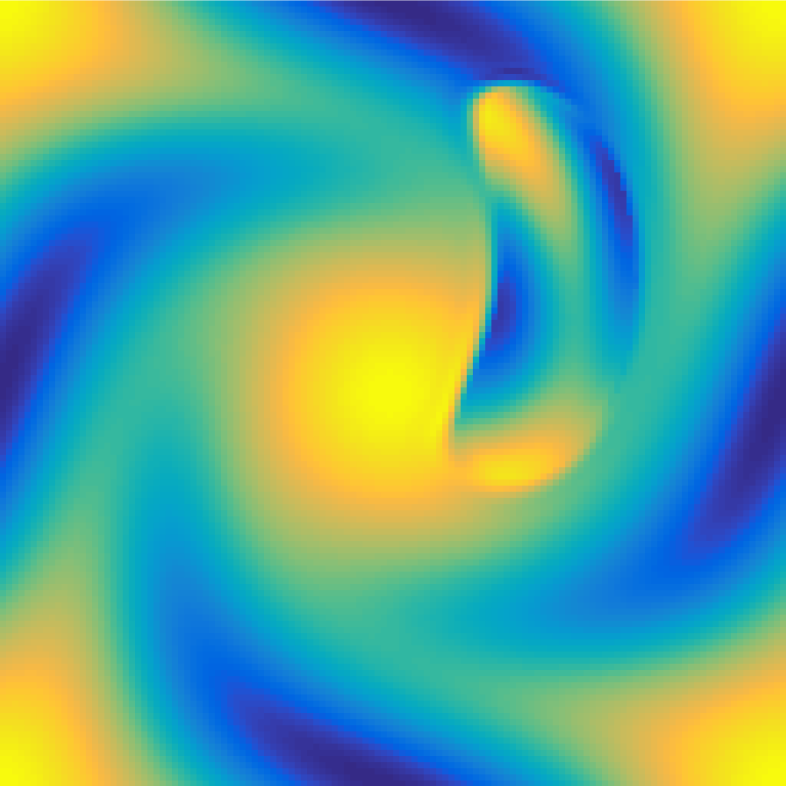}
  } 
  \caption{Example intermediate solutions to the 2D vortex reverse problem, where the field~\cref{eqn:validation:advection:field} is passively advected along with the flow.
  Here $h / R \approx 1/19$.}
  \label{fig:validation:advection:examples}
\end{figure}
In 2D we consider the classical vortex reverse problem~\citep{Bell1989,Rider1997} where the interface undergoes a reversible deformation defined by the stream function
\begin{equation}
  \bar\psi(t, \+x) = \frac{\cos(\pi t/T)}{\pi} \sin(\pi x)^2 \sin(\pi y)^2.
\end{equation}
The initial interface is a circle of radius $R = 0.15$ centred at $\+x_0 = \begin{bmatrix}0.5& 0.75\end{bmatrix}^T$.
We let $T = 1$ such that the interface remains easy to resolve.

The deformation field proposed by~\citet{Leveque1996a} is considered in 3D
\begin{equation}
  \Onevelo(t, \+x) = \cos(\pi t/T)\begin{bmatrix}
    2\sin(\pi x)^2\sin(2\pi y)\sin(2\pi z)\\
    -\sin(2\pi x)\sin(\pi y)^2\sin(2\pi z)\\
    -\sin(2\pi x)\sin(2\pi y)\sin(\pi z)^2
  \end{bmatrix},
\end{equation}
where a sphere of radius $R = 0.15$ is initially located at $\+x_0 = \begin{bmatrix}0.35& 0.35& 0.35\end{bmatrix}^T$.
The period is again taken as $T = 1$.

We will now explicitly denote an approximation to $\avarstag$ as $\approximate{\avarstag}$.

\subsection{Comparison of the one- and \twofluid formulation}
First we will compare the one- and \twofluid formulations in terms of the advection of the discontinuous field given by~\cref{eqn:validation:advection:field}.
In~\cref{fig:validation:advection:example:lw2,fig:validation:advection:example:lw1} we show the intermediate solutions of the 2D vortex reverse problem when using the LW flux interpolant with the two- and \onefluid formulation respectively.
The LW flux results in the appearance of wiggles in the lighter gas phase when the \onefluid formulation is used, the \twofluid formulation however yields satisfactory results.
Using the Fromm flux results in an absence of wiggles also in the \onefluid formulation, as shown in~\cref{fig:validation:advection:example:fromm}.
Of course the shear layer is still diffuse, contrary to the sharp shear layer found with the \twofluid formulation.

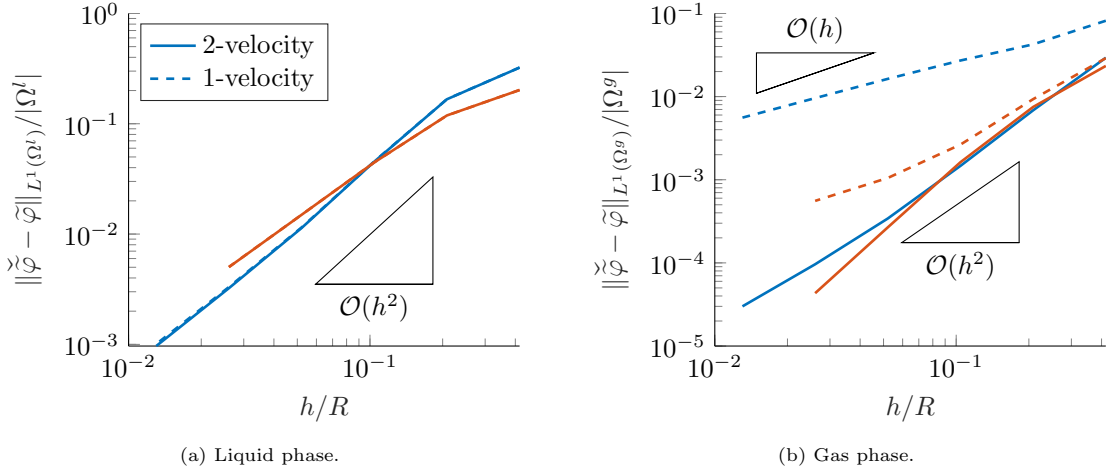
\begin{figure} 
  \subcaptionbox{Liquid phase.\label{fig:validation:advection:adveccomp:liq}}
  [\twofigwidth]{
    \def\tikzWidth{\textwidth*0.35}
    \def\tikzHeight{\textwidth*0.3}
    \tikzsetnextfilename{advection_paper_adveccomp_Fromm_l1ErrorL}
%
%
\definecolor{mycolor1}{rgb}{0.00000,0.44700,0.74100}%
\definecolor{mycolor2}{rgb}{0.85000,0.32500,0.09800}%
\begin{tikzpicture}

\begin{axis}[%
width=12.206in,
height=9.878in,
at={(2.047in,1.333in)},
scale only axis,
unbounded coords=jump,
xmode=log,
xmin=0.01,
xmax=0.416666666666667,
xminorticks=true,
xlabel style={font=\color{white!15!black}},
xlabel={$h / R$},
ymode=log,
ymin=0.000966011877192583,
ymax=1,
yminorticks=true,
ylabel style={font=\color{white!15!black}},
ylabel={$\|\approximate{\avarstag} - \avarstag\|_{L^1(\Omega^l)} / |\Omega^l|$},
axis background/.style={fill=white},
axis x line*=bottom,
axis y line*=left,
legend style={at={(0.03,0.97)}, anchor=north west, legend cell align=left, align=left, draw=white!15!black},
width=\tikzWidth,
height=\tikzHeight,
at={(0,0)}
]
\addplot [color=mycolor1, line width=1.0pt]
  table[]{/home/ronald/git_repos/phd_project/papers/2021_2fluid_pt1_transport/tikz/advection_paper_adveccomp_Fromm_l1ErrorL-1.tsv};
\addlegendentry{2-velocity}

\addplot [color=mycolor1, dashed, line width=1.0pt]
  table[]{/home/ronald/git_repos/phd_project/papers/2021_2fluid_pt1_transport/tikz/advection_paper_adveccomp_Fromm_l1ErrorL-2.tsv};
\addlegendentry{1-velocity}

\addplot [color=mycolor2, line width=1.0pt, forget plot]
  table[]{/home/ronald/git_repos/phd_project/papers/2021_2fluid_pt1_transport/tikz/advection_paper_adveccomp_Fromm_l1ErrorL-3.tsv};
\addplot [color=mycolor2, dashed, line width=1.0pt, forget plot]
  table[]{/home/ronald/git_repos/phd_project/papers/2021_2fluid_pt1_transport/tikz/advection_paper_adveccomp_Fromm_l1ErrorL-4.tsv};
\addplot [color=black, forget plot]
  table[]{/home/ronald/git_repos/phd_project/papers/2021_2fluid_pt1_transport/tikz/advection_paper_adveccomp_Fromm_l1ErrorL-5.tsv};
\node[below, align=center]
at (axis cs:1.042e-01,3.508e-03) {$\mathcal{O}(h^{2})$};
\end{axis}
\end{tikzpicture}%
  }\hfill
  \subcaptionbox{Gas phase.\label{fig:validation:advection:adveccomp:gas}}
  [\twofigwidth]{
    \def\tikzWidth{\textwidth*0.35}
    \def\tikzHeight{\textwidth*0.3}
    \tikzsetnextfilename{advection_paper_adveccomp_Fromm_l1ErrorG}
%
%
\definecolor{mycolor1}{rgb}{0.00000,0.44700,0.74100}%
\definecolor{mycolor2}{rgb}{0.85000,0.32500,0.09800}%
\begin{tikzpicture}

\begin{axis}[%
width=12.206in,
height=9.878in,
at={(2.047in,1.333in)},
scale only axis,
unbounded coords=jump,
xmode=log,
xmin=0.01,
xmax=0.416666666666667,
xminorticks=true,
xlabel style={font=\color{white!15!black}},
xlabel={$h / R$},
ymode=log,
ymin=1e-05,
ymax=0.1,
yminorticks=true,
ylabel style={font=\color{white!15!black}},
ylabel={$\|\approximate{\avarstag} - \avarstag\|_{L^1(\Omega^g)} / |\Omega^g|$},
axis background/.style={fill=white},
axis x line*=bottom,
axis y line*=left,
width=\tikzWidth,
height=\tikzHeight,
at={(0,0)}
]
\addplot [color=mycolor1, line width=1.0pt, forget plot]
  table[]{/home/ronald/git_repos/phd_project/papers/2021_2fluid_pt1_transport/tikz/advection_paper_adveccomp_Fromm_l1ErrorG-1.tsv};
\addplot [color=mycolor1, dashed, line width=1.0pt, forget plot]
  table[]{/home/ronald/git_repos/phd_project/papers/2021_2fluid_pt1_transport/tikz/advection_paper_adveccomp_Fromm_l1ErrorG-2.tsv};
\addplot [color=mycolor2, line width=1.0pt, forget plot]
  table[]{/home/ronald/git_repos/phd_project/papers/2021_2fluid_pt1_transport/tikz/advection_paper_adveccomp_Fromm_l1ErrorG-3.tsv};
\addplot [color=mycolor2, dashed, line width=1.0pt, forget plot]
  table[]{/home/ronald/git_repos/phd_project/papers/2021_2fluid_pt1_transport/tikz/advection_paper_adveccomp_Fromm_l1ErrorG-4.tsv};
\addplot [color=black, forget plot]
  table[]{/home/ronald/git_repos/phd_project/papers/2021_2fluid_pt1_transport/tikz/advection_paper_adveccomp_Fromm_l1ErrorG-5.tsv};
\node[above, align=center]
at (axis cs:2.604e-02,3.354e-02) {$\mathcal{O}(h)$};
\addplot [color=black, forget plot]
  table[]{/home/ronald/git_repos/phd_project/papers/2021_2fluid_pt1_transport/tikz/advection_paper_adveccomp_Fromm_l1ErrorG-6.tsv};
\node[below, align=center]
at (axis cs:1.042e-01,1.754e-04) {$\mathcal{O}(h^{2})$};
\end{axis}
\end{tikzpicture}%
  } 
  \caption{Comparison of the one- and \twofluid formulations in terms of the $L^1$-norm of the error at $t = T$ for the 2D vortex reverse problem (blue) and 3D deformation problem (red).
  The Fromm flux was used.}
  \label{fig:validation:advection:adveccomp} 
\end{figure}
These observations are reflected in~\cref{fig:validation:advection:adveccomp} where we show the accuracy of the advection method in terms of the $L^1$-norm of the error at $t = T$, for each of the phases separately.
We combine the results from the 2D vortex reverse as well as 3D deformation problem.
Because the liquid phase is heavier, it is favoured by~\cref{eqn:momentum:onefluid}, and therefore the accuracy in the liquid phase is the same for the two formulations, which is close to second-order.
In the gas phase we however find that the \onefluid formulation results in less than first-order accuracy, whereas for the \twofluid formulation the accuracy is again of second-order.

\subsection{Comparison of flux interpolants (2D)}
\begin{figure}
  \centering
  \tikzsetnextfilename{advec_fluxcomp_legend}
  \definecolor{mycolor1}{rgb}{0.00000,0.44700,0.74100}%
\definecolor{mycolor2}{rgb}{0.85000,0.32500,0.09800}%
\definecolor{mycolor3}{rgb}{0.46600,0.67400,0.18800}%
\definecolor{mycolor4}{rgb}{0.49400,0.18400,0.55600}%
\begin{tikzpicture}

\begin{axis}[%
  hide axis,
  width=2cm,
  height=2cm,
  ymin=0.990,
  ymax=1,
  legend style={at={(0.5,0.5)}, anchor=south, legend columns=4, legend cell align=left, align=left, draw=white!15!black},
  ]

  \addplot [line width=1.0pt, color=mycolor1, domain=-0.1:0.1,samples=2, draw=none] {1};
  \addlegendentry{\gls{lw}}
  \addplot [line width=1.0pt, color=mycolor2, domain=-0.1:0.1,samples=2, draw=none] {1};
  \addlegendentry{Fromm}
  \addplot [line width=1.0pt, color=mycolor3, domain=-0.1:0.1,samples=2, draw=none] {1};
  \addlegendentry{\gls{mc}}
  \addplot [line width=1.0pt, color=mycolor4, domain=-0.1:0.1,samples=2, draw=none] {1};
  \addlegendentry{Upwind}

\end{axis}
\end{tikzpicture}%

  \subcaptionbox{No interface ($\Omega^l = \Omega$).\label{fig:validation:advection:fluxcomp_noint}}
  [\twofigwidth]{
    \def\tikzWidth{\textwidth*0.33}
    \def\tikzHeight{\textwidth*0.3}
    \tikzsetnextfilename{advection_paper_nointerface_fluxcomp_l1ErrorL}
%
%
\definecolor{mycolor1}{rgb}{0.00000,0.44700,0.74100}%
\definecolor{mycolor2}{rgb}{0.85000,0.32500,0.09800}%
\definecolor{mycolor3}{rgb}{0.46600,0.67400,0.18800}%
\definecolor{mycolor4}{rgb}{0.49400,0.18400,0.55600}%
\begin{tikzpicture}

\begin{axis}[%
width=12.206in,
height=9.878in,
at={(2.047in,1.333in)},
scale only axis,
xmode=log,
xmin=0.01,
xmax=0.416666666666667,
xminorticks=true,
xlabel style={font=\color{white!15!black}},
xlabel={$h / R$},
ymode=log,
ymin=1e-06,
ymax=1,
yminorticks=true,
ylabel style={font=\color{white!15!black}},
ylabel={$\|\approximate{\avarstag} - \avarstag\|_{L^1(\Omega)} / |\Omega|$},
axis background/.style={fill=white},
axis x line*=bottom,
axis y line*=left,
width=\tikzWidth,
height=\tikzHeight,
at={(0,0)}
]
\addplot [color=mycolor1, line width=1.0pt, forget plot]
  table[]{/home/ronald/git_repos/phd_project/papers/2021_2fluid_pt1_transport/tikz/advection_paper_nointerface_fluxcomp_l1ErrorL-1.tsv};
\addplot [color=mycolor2, line width=1.0pt, forget plot]
  table[]{/home/ronald/git_repos/phd_project/papers/2021_2fluid_pt1_transport/tikz/advection_paper_nointerface_fluxcomp_l1ErrorL-2.tsv};
\addplot [color=mycolor3, line width=1.0pt, forget plot]
  table[]{/home/ronald/git_repos/phd_project/papers/2021_2fluid_pt1_transport/tikz/advection_paper_nointerface_fluxcomp_l1ErrorL-3.tsv};
\addplot [color=mycolor4, line width=1.0pt, forget plot]
  table[]{/home/ronald/git_repos/phd_project/papers/2021_2fluid_pt1_transport/tikz/advection_paper_nointerface_fluxcomp_l1ErrorL-4.tsv};
\addplot [color=black, forget plot]
  table[]{/home/ronald/git_repos/phd_project/papers/2021_2fluid_pt1_transport/tikz/advection_paper_nointerface_fluxcomp_l1ErrorL-5.tsv};
\node[below, align=center]
at (axis cs:2.604e-02,1.480e-02) {$\mathcal{O}(h)$};
\addplot [color=black, forget plot]
  table[]{/home/ronald/git_repos/phd_project/papers/2021_2fluid_pt1_transport/tikz/advection_paper_nointerface_fluxcomp_l1ErrorL-6.tsv};
\node[below, align=center]
at (axis cs:1.042e-01,1.126e-04) {$\mathcal{O}(h^{3})$};
\end{axis}
\end{tikzpicture}%
  }\hfill
  \subcaptionbox{With interface (\twofluid model).\label{fig:validation:advection:fluxcomp}}
  [\twofigwidth]{
    \def\tikzWidth{\textwidth*0.33}
    \def\tikzHeight{\textwidth*0.3}
    \tikzsetnextfilename{advection_paper_fluxcomp_l1ErrorL}
%
%
\definecolor{mycolor1}{rgb}{0.00000,0.44700,0.74100}%
\definecolor{mycolor2}{rgb}{0.85000,0.32500,0.09800}%
\definecolor{mycolor3}{rgb}{0.46600,0.67400,0.18800}%
\definecolor{mycolor4}{rgb}{0.49400,0.18400,0.55600}%
\begin{tikzpicture}

\begin{axis}[%
width=12.206in,
height=9.878in,
at={(2.047in,1.333in)},
scale only axis,
xmode=log,
xmin=0.01,
xmax=0.416666666666667,
xminorticks=true,
xlabel style={font=\color{white!15!black}},
xlabel={$h / R$},
ymode=log,
ymin=0.000973917013588004,
ymax=1,
yminorticks=true,
ylabel style={font=\color{white!15!black}},
ylabel={$\|\approximate{\avarstag} - \avarstag\|_{L^1(\Omega^l)} / |\Omega^l|$},
axis background/.style={fill=white},
axis x line*=bottom,
axis y line*=left,
width=\tikzWidth,
height=\tikzHeight,
at={(0,0)}
]
\addplot [color=mycolor1, line width=1.0pt, forget plot]
  table[]{/home/ronald/git_repos/phd_project/papers/2021_2fluid_pt1_transport/tikz/advection_paper_fluxcomp_l1ErrorL-1.tsv};
\addplot [color=mycolor2, line width=1.0pt, forget plot]
  table[]{/home/ronald/git_repos/phd_project/papers/2021_2fluid_pt1_transport/tikz/advection_paper_fluxcomp_l1ErrorL-2.tsv};
\addplot [color=mycolor3, line width=1.0pt, forget plot]
  table[]{/home/ronald/git_repos/phd_project/papers/2021_2fluid_pt1_transport/tikz/advection_paper_fluxcomp_l1ErrorL-3.tsv};
\addplot [color=mycolor4, line width=1.0pt, forget plot]
  table[]{/home/ronald/git_repos/phd_project/papers/2021_2fluid_pt1_transport/tikz/advection_paper_fluxcomp_l1ErrorL-4.tsv};
\addplot [color=black, forget plot]
  table[]{/home/ronald/git_repos/phd_project/papers/2021_2fluid_pt1_transport/tikz/advection_paper_fluxcomp_l1ErrorL-5.tsv};
\node[below, align=center]
at (axis cs:2.604e-02,2.673e-02) {$\mathcal{O}(h)$};
\addplot [color=black, forget plot]
  table[]{/home/ronald/git_repos/phd_project/papers/2021_2fluid_pt1_transport/tikz/advection_paper_fluxcomp_l1ErrorL-6.tsv};
\node[below, align=center]
at (axis cs:1.042e-01,7.506e-03) {$\mathcal{O}(h^{2})$};
\end{axis}
\end{tikzpicture}%
  }\hfill
  \caption{Comparison of the flux interpolants in terms of the $L^1$-norm of the error at $t = T$ for the 2D vortex reverse problem.}

  \vspace*{\floatsep}

  \subcaptionbox{No interface ($\Omega^l = \Omega$).\label{fig:validation:advection:fluxcomp_noint_energy}}
  [\twofigwidth]{
    \def\tikzWidth{\textwidth*0.33}
    \def\tikzHeight{\textwidth*0.3}
    \tikzsetnextfilename{advection_paper_nointerface_fluxcomp_energyL}
%
%
\definecolor{mycolor1}{rgb}{0.00000,0.44700,0.74100}%
\definecolor{mycolor2}{rgb}{0.85000,0.32500,0.09800}%
\definecolor{mycolor3}{rgb}{0.46600,0.67400,0.18800}%
\definecolor{mycolor4}{rgb}{0.49400,0.18400,0.55600}%
\begin{tikzpicture}

\begin{axis}[%
width=12.206in,
height=9.878in,
at={(2.047in,1.333in)},
scale only axis,
xmode=log,
xmin=0.01,
xmax=0.416666666666667,
xminorticks=true,
xlabel style={font=\color{white!15!black}},
xlabel={$h / R$},
ymode=log,
ymin=1e-06,
ymax=1,
yminorticks=true,
ylabel style={font=\color{white!15!black}},
ylabel={$\abs{\Delta E_k} / E_k(0)$},
axis background/.style={fill=white},
axis x line*=bottom,
axis y line*=left,
width=\tikzWidth,
height=\tikzHeight,
at={(0,0)}
]
\addplot [color=mycolor1, line width=1.0pt, forget plot]
  table[]{/home/ronald/git_repos/phd_project/papers/2021_2fluid_pt1_transport/tikz/advection_paper_nointerface_fluxcomp_energyL-1.tsv};
\addplot [color=mycolor2, line width=1.0pt, forget plot]
  table[]{/home/ronald/git_repos/phd_project/papers/2021_2fluid_pt1_transport/tikz/advection_paper_nointerface_fluxcomp_energyL-2.tsv};
\addplot [color=mycolor3, line width=1.0pt, forget plot]
  table[]{/home/ronald/git_repos/phd_project/papers/2021_2fluid_pt1_transport/tikz/advection_paper_nointerface_fluxcomp_energyL-3.tsv};
\addplot [color=mycolor4, line width=1.0pt, forget plot]
  table[]{/home/ronald/git_repos/phd_project/papers/2021_2fluid_pt1_transport/tikz/advection_paper_nointerface_fluxcomp_energyL-4.tsv};
\addplot [color=black, forget plot]
  table[]{/home/ronald/git_repos/phd_project/papers/2021_2fluid_pt1_transport/tikz/advection_paper_nointerface_fluxcomp_energyL-5.tsv};
\node[below, align=center]
at (axis cs:2.604e-02,3.058e-02) {$\mathcal{O}(h)$};
\addplot [color=black, forget plot]
  table[]{/home/ronald/git_repos/phd_project/papers/2021_2fluid_pt1_transport/tikz/advection_paper_nointerface_fluxcomp_energyL-6.tsv};
\node[below, align=center]
at (axis cs:1.042e-01,1.500e-04) {$\mathcal{O}(h^{3})$};
\end{axis}
\end{tikzpicture}%
  }\hfill
  \subcaptionbox{With interface (\twofluid model).\label{fig:validation:advection:fluxcomp_energy}}
  [\twofigwidth]{
    \def\tikzWidth{\textwidth*0.33}
    \def\tikzHeight{\textwidth*0.3}
    \tikzsetnextfilename{advection_paper_fluxcomp_energyL}
%
%
\definecolor{mycolor1}{rgb}{0.00000,0.44700,0.74100}%
\definecolor{mycolor2}{rgb}{0.85000,0.32500,0.09800}%
\definecolor{mycolor3}{rgb}{0.46600,0.67400,0.18800}%
\definecolor{mycolor4}{rgb}{0.49400,0.18400,0.55600}%
\begin{tikzpicture}

\begin{axis}[%
width=12.206in,
height=9.878in,
at={(2.047in,1.333in)},
scale only axis,
xmode=log,
xmin=0.01,
xmax=0.416666666666667,
xminorticks=true,
xlabel style={font=\color{white!15!black}},
xlabel={$h / R$},
ymode=log,
ymin=0.0001,
ymax=0.123612714567334,
yminorticks=true,
ylabel style={font=\color{white!15!black}},
ylabel={$\abs{\Delta E_k^l} / E^l_k(0)$},
axis background/.style={fill=white},
axis x line*=bottom,
axis y line*=left,
width=\tikzWidth,
height=\tikzHeight,
at={(0,0)}
]
\addplot [color=mycolor1, line width=1.0pt, forget plot]
  table[]{/home/ronald/git_repos/phd_project/papers/2021_2fluid_pt1_transport/tikz/advection_paper_fluxcomp_energyL-1.tsv};
\addplot [color=mycolor2, line width=1.0pt, forget plot]
  table[]{/home/ronald/git_repos/phd_project/papers/2021_2fluid_pt1_transport/tikz/advection_paper_fluxcomp_energyL-2.tsv};
\addplot [color=mycolor3, line width=1.0pt, forget plot]
  table[]{/home/ronald/git_repos/phd_project/papers/2021_2fluid_pt1_transport/tikz/advection_paper_fluxcomp_energyL-3.tsv};
\addplot [color=mycolor4, line width=1.0pt, forget plot]
  table[]{/home/ronald/git_repos/phd_project/papers/2021_2fluid_pt1_transport/tikz/advection_paper_fluxcomp_energyL-4.tsv};
\addplot [color=black, forget plot]
  table[]{/home/ronald/git_repos/phd_project/papers/2021_2fluid_pt1_transport/tikz/advection_paper_fluxcomp_energyL-5.tsv};
\node[below, align=center]
at (axis cs:2.604e-02,1.368e-02) {$\mathcal{O}(h)$};
\addplot [color=black, forget plot]
  table[]{/home/ronald/git_repos/phd_project/papers/2021_2fluid_pt1_transport/tikz/advection_paper_fluxcomp_energyL-6.tsv};
\node[below, align=center]
at (axis cs:1.042e-01,2.833e-03) {$\mathcal{O}(h^{2})$};
\end{axis}
\end{tikzpicture}%
  }\hfill
  \caption{Comparison of the flux interpolants in terms of the loss of kinetic energy at $t = T$ for the 2D vortex reverse problem.}
\end{figure}

We now consider the effect of the flux interpolant on the accuracy of the advection method, in 2D.
To separate the effect of the \gls{ctu} flux interpolant from the higher-order flux interpolants, we also consider a simulation without an interface ($\Omega^l = \Omega$).
The accuracy in terms of the $L^1$-norm, in the absence of an interface, is shown in~\cref{fig:validation:advection:fluxcomp_noint}.
We find that the higher-order flux interpolants (excluding the upwind interpolant) are third-order accurate in space and time (the \gls{cfl} constraint is fixed at $\wycflval = \half$, see also~\cref{eqn:interface:onefluid:dr:cfl}).
In~\cref{fig:validation:advection:fluxcomp} we show the accuracy in the presence of an interface, where the \twofluid formulation is used.
We only show the accuracy of the liquid phase, as the accuracy in the gas phase was found to be similar.
The results show that the flux interpolants, except for the upwind interpolant, are close to second-order accurate in the $L^1$-norm and of very similar accuracy.
Moreover we find that the error is two orders of magnitude larger than what we found in the absence of an interface, which suggests that the first-order error at the interface, due to the \gls{ctu} flux interpolant, dominates the overall accuracy.

We furthermore consider how well quadratic invariants are conserved (see~\cref{sec:app:quadratic}), where we define the `kinetic energy' of the passively advected scalar as
\begin{equation}
  E^\pi_k \defeq \half\sum_\setextrusion{F} \volfracstag^\pi \rho^\pi (\avarstag^\pi)^2.
\end{equation}
The change in kinetic energy is defined as
\begin{equation}
  \Delta E^\pi_k(t) \defeq E^\pi_k(t) - E^\pi_k(0).
\end{equation}
The resulting change in kinetic energy at $t = T$, relative to the initial kinetic energy, is shown in~\cref{fig:validation:advection:fluxcomp_noint_energy,fig:validation:advection:fluxcomp_energy} for the simulation without and with the interface respectively.
In the absence of an interface, we find that the higher-order methods are third-order accurate, where the \gls{lw} flux interpolant yields the smallest change in kinetic energy, as can be expected from the discussion in~\cref{sec:app:quadratic}.
When an interface is included, however, we find that significantly more energy is lost, which is attributed to the use of the first-order accurate \gls{ctu} flux interpolant at the interface.

\subsection{Effect of the parameters $\wycflval$ and $\minfracval$ (2D)}
Finally we consider the effect of the \gls{cfl} limit $\wycflval$ (see also~\cref{eqn:interface:onefluid:dr:cfl}) and the parameter $\minfracval$ (as required in~\cref{eqn:centred:unsplit:modified}, see also~\cref{lem:momentum:mod_boundedness,eqn:momentum:unsplit:modified_simplified}) on the accuracy as well as conservation of kinetic energy.
Recall that the parameter $\minfracval$ determines when we use the \gls{ctu} flux interpolant: if $\volfracstag_f^{\pi,(n+1)} < \minfracval$ then the \gls{ctu} flux interpolant is used on all faces of $\omega_f$.
As the parameter $\minfracval$ only has an influence at the interface, and results in the use of the first-order accurate \gls{ctu} flux interpolant, we now consider the $L^\infty$-norm of the error at $t = T$ such that we essentially obtain the accuracy at the interface.

\begin{figure} 
  \centering
  \tikzexternaldisable
  \definecolor{mycolor1}{rgb}{0.00000,0.44700,0.74100}%
\definecolor{mycolor2}{rgb}{0.85000,0.32500,0.09800}%
\definecolor{mycolor3}{rgb}{0.92900,0.69400,0.12500}%

\begin{tikzpicture}
    \begin{axis}[
        hide axis,
        width=2.cm,
        height=2cm,
        ymin=0.99,
        ymax=1,
    ]
        \addplot [dotted, line width=1.0pt, color=mycolor2, domain=-0.1:0.1,samples=2, draw=none] {1};
            \label{plot:advect_paper:cflcomp:line1}
        \addplot [dashed, line width=1.0pt, color=mycolor2, domain=-0.1:0.1,samples=2, draw=none] {1};
            \label{plot:advect_paper:cflcomp:line2}
        \addplot [line width=1.0pt, color=mycolor2, domain=-0.1:0.1,samples=2, draw=none] {1};
            \label{plot:advect_paper:cflcomp:line3}
        \addplot [dotted, line width=1.0pt, color=mycolor1, domain=-0.1:0.1,samples=2, draw=none] {1};
            \label{plot:advect_paper:cflcomp:line4}
        \addplot [dashed, line width=1.0pt, color=mycolor1, domain=-0.1:0.1,samples=2, draw=none] {1};
            \label{plot:advect_paper:cflcomp:line5}
        \addplot [line width=1.0pt, color=mycolor1, domain=-0.1:0.1,samples=2, draw=none] {1};
            \label{plot:advect_paper:cflcomp:line6}

        %
        \coordinate (legend) at (axis description cs:0.5,0.5);
    \end{axis}

    %
    \draw (0,0) node{
      \begin{tabular}{|clccc|}
        \hline
         \multicolumn{2}{|r}{\small $\wycflval$:}                   & $0.1$                             & $0.5$ & $1.0$\\
        \hline
        \multirow{2}{*}{{\small $\minfracval:$}}& $0.1$  & \ref{plot:advect_paper:cflcomp:line4}   & \ref{plot:advect_paper:cflcomp:line5} & \ref{plot:advect_paper:cflcomp:line6}\\[0.15cm]
        & $0.5$  & \ref{plot:advect_paper:cflcomp:line1}   & \ref{plot:advect_paper:cflcomp:line2} & \ref{plot:advect_paper:cflcomp:line3}\\
        \hline
    \end{tabular}
    };

\end{tikzpicture}
  \tikzexternalenable

  \subcaptionbox{The $L^\infty$-norm of the error in the gas phase.\label{fig:validation:advection:cflcomp_linf}}
  [\twofigwidth]{
    \def\tikzWidth{\textwidth*0.35}  
    \def\tikzHeight{\textwidth*0.3}
    \tikzsetnextfilename{advection_paper_cflcomp_lInfErrorG}
%
%
\definecolor{mycolor1}{rgb}{0.85000,0.32500,0.09800}%
\definecolor{mycolor2}{rgb}{0.00000,0.44700,0.74100}%
\begin{tikzpicture}

\begin{axis}[%
width=12.206in,
height=9.878in,
at={(2.047in,1.333in)},
scale only axis,
xmode=log,
xmin=0.01,
xmax=0.416666666666667,
xminorticks=true,
xlabel style={font=\color{white!15!black}},
xlabel={$h / R$},
ymode=log,
ymin=0.01,
ymax=0.261593565636194,
yminorticks=true,
ylabel style={font=\color{white!15!black}},
ylabel={$\|\approximate{\avarstag} - \avarstag\|_{L^\infty(\Omega^g)}$},
axis background/.style={fill=white},
axis x line*=bottom,
axis y line*=left,
width=\tikzWidth,
height=\tikzHeight,
at={(0,0)}
]
\addplot [color=mycolor1, line width=1.0pt, forget plot]
  table[]{/home/ronald/git_repos/phd_project/papers/2021_2fluid_pt1_transport/tikz/advection_paper_cflcomp_lInfErrorG-1.tsv};
\addplot [color=mycolor1, dashed, line width=1.0pt, forget plot]
  table[]{/home/ronald/git_repos/phd_project/papers/2021_2fluid_pt1_transport/tikz/advection_paper_cflcomp_lInfErrorG-2.tsv};
\addplot [color=mycolor1, dotted, line width=1.0pt, forget plot]
  table[]{/home/ronald/git_repos/phd_project/papers/2021_2fluid_pt1_transport/tikz/advection_paper_cflcomp_lInfErrorG-3.tsv};
\addplot [color=mycolor2, dashed, line width=1.0pt, forget plot]
  table[]{/home/ronald/git_repos/phd_project/papers/2021_2fluid_pt1_transport/tikz/advection_paper_cflcomp_lInfErrorG-4.tsv};
\addplot [color=mycolor2, dotted, line width=1.0pt, forget plot]
  table[]{/home/ronald/git_repos/phd_project/papers/2021_2fluid_pt1_transport/tikz/advection_paper_cflcomp_lInfErrorG-5.tsv};
\addplot [color=mycolor2, line width=1.0pt, forget plot]
  table[]{/home/ronald/git_repos/phd_project/papers/2021_2fluid_pt1_transport/tikz/advection_paper_cflcomp_lInfErrorG-6.tsv};
\addplot [color=white!50!mycolor2, dashed, line width=1.0pt, mark size=2.5pt, mark=+, mark options={solid, white!50!mycolor2}, forget plot]
  table[]{/home/ronald/git_repos/phd_project/papers/2021_2fluid_pt1_transport/tikz/advection_paper_cflcomp_lInfErrorG-7.tsv};
\addplot [color=white!50!mycolor2, dotted, line width=1.0pt, mark size=2.5pt, mark=+, mark options={solid, white!50!mycolor2}, forget plot]
  table[]{/home/ronald/git_repos/phd_project/papers/2021_2fluid_pt1_transport/tikz/advection_paper_cflcomp_lInfErrorG-8.tsv};
\addplot [color=black, forget plot]
  table[]{/home/ronald/git_repos/phd_project/papers/2021_2fluid_pt1_transport/tikz/advection_paper_cflcomp_lInfErrorG-9.tsv};
\node[below, align=center]
at (axis cs:1.042e-01,2.101e-02) {$\mathcal{O}(h)$};
\end{axis}
\end{tikzpicture}%
  }\hfill
  \subcaptionbox{The loss of kinetic energy in the gas phase.\label{fig:validation:advection:cflcomp_energy}}
  [\twofigwidth]{
    \def\tikzWidth{\textwidth*0.35} 
    \def\tikzHeight{\textwidth*0.3} 
    \tikzsetnextfilename{advection_paper_cflcomp_energyG}
%
%
\definecolor{mycolor1}{rgb}{0.85000,0.32500,0.09800}%
\definecolor{mycolor2}{rgb}{0.00000,0.44700,0.74100}%
\begin{tikzpicture}

\begin{axis}[%
width=12.206in,
height=9.878in,
at={(2.047in,1.333in)},
scale only axis,
xmode=log,
xmin=0.01,
xmax=0.416666666666667,
xminorticks=true,
xlabel style={font=\color{white!15!black}},
xlabel={$h / R$},
ymode=log,
ymin=1e-06,
ymax=0.1,
yminorticks=true,
ylabel style={font=\color{white!15!black}},
ylabel={$\abs{\Delta E_k^g} / E^g_k(0)$},
axis background/.style={fill=white},
axis x line*=bottom,
axis y line*=left,
width=\tikzWidth,
height=\tikzHeight,
at={(0,0)}
]
\addplot [color=mycolor1, line width=1.0pt, forget plot]
  table[]{/home/ronald/git_repos/phd_project/papers/2021_2fluid_pt1_transport/tikz/advection_paper_cflcomp_energyG-1.tsv};
\addplot [color=mycolor1, dashed, line width=1.0pt, forget plot]
  table[]{/home/ronald/git_repos/phd_project/papers/2021_2fluid_pt1_transport/tikz/advection_paper_cflcomp_energyG-2.tsv};
\addplot [color=mycolor1, dotted, line width=1.0pt, forget plot]
  table[]{/home/ronald/git_repos/phd_project/papers/2021_2fluid_pt1_transport/tikz/advection_paper_cflcomp_energyG-3.tsv};
\addplot [color=mycolor2, dashed, line width=1.0pt, forget plot]
  table[]{/home/ronald/git_repos/phd_project/papers/2021_2fluid_pt1_transport/tikz/advection_paper_cflcomp_energyG-4.tsv};
\addplot [color=mycolor2, dotted, line width=1.0pt, forget plot]
  table[]{/home/ronald/git_repos/phd_project/papers/2021_2fluid_pt1_transport/tikz/advection_paper_cflcomp_energyG-5.tsv};
\addplot [color=mycolor2, line width=1.0pt, forget plot]
  table[]{/home/ronald/git_repos/phd_project/papers/2021_2fluid_pt1_transport/tikz/advection_paper_cflcomp_energyG-6.tsv};
\addplot [color=white!50!mycolor2, dashed, line width=1.0pt, mark size=2.5pt, mark=+, mark options={solid, white!50!mycolor2}, forget plot]
  table[]{/home/ronald/git_repos/phd_project/papers/2021_2fluid_pt1_transport/tikz/advection_paper_cflcomp_energyG-7.tsv};
\addplot [color=white!50!mycolor2, dotted, line width=1.0pt, mark size=2.5pt, mark=+, mark options={solid, white!50!mycolor2}, forget plot]
  table[]{/home/ronald/git_repos/phd_project/papers/2021_2fluid_pt1_transport/tikz/advection_paper_cflcomp_energyG-8.tsv};
\addplot [color=black, forget plot]
  table[]{/home/ronald/git_repos/phd_project/papers/2021_2fluid_pt1_transport/tikz/advection_paper_cflcomp_energyG-9.tsv};
\node[below, align=center]
at (axis cs:1.042e-01,3.355e-05) {$\mathcal{O}(h^{2})$};
\end{axis}
\end{tikzpicture}%
  }
  \caption{Comparison of the parameters $\wycflval$ and $\minfracval$ for the 2D vortex reverse problem, using the \gls{lw} flux interpolant.
  The markers (shown for only for two pairs of $(\wycflval, \minfracval)$) result from using the Fromm flux interpolant.
  The \twofluid formulation was used.}
  \label{fig:validation:advection:cflcomp}
\end{figure}
We consider mesh refinement for $\wycflval \in \{0.1, 0.5, 1.0\}$ and $\minfracval \in \{0.1, 0.5\}$, as shown in~\cref{fig:validation:advection:cflcomp}.
In~\cref{fig:validation:advection:cflcomp_linf} the resulting $L^\infty$-norm of the error in the gas phase is shown, where we made use of the \gls{lw} flux interpolant (no markers) as well as the Fromm flux interpolant (markers, shown only for two pairs of $(\wycflval, \minfracval)$).
We find that varying the parameters hardly affects the accuracy at the interface when the \gls{lw} flux interpolant is used.
What little effect there is, consistently shows that taking larger time steps (i.e. considering larger values of $\wycflval$) yields a more accurate solution for a fixed value of $\minfracval$.
When the Fromm flux interpolant is used, we find that using $\minfracval = 0.1$ yields divergence of the solution under time step refinement (i.e. $\wycflval \rightarrow 0$), suggesting that $\minfracval$ should not be taken too small when the Fromm flux interpolant is used.
Note that we have shown the results using the Fromm flux interpolant only for those two pairs of $(\wycflval, \minfracval)$ which yield divergence of the solution.
Exactly why this occurs with the Fromm flux interpolant, and not with the \gls{lw} flux interpolant, is unclear.


The loss in kinetic energy is shown in~\cref{fig:validation:advection:cflcomp_energy}. 
A lower value of $\minfracval$ does lead to improved conservation of kinetic energy when the \gls{lw} flux interpolant is used, as can be expected because the \gls{ctu} flux interpolant is used less frequently.
Moreover we find that taking a smaller time step, resulting in a smaller value of $\wycflval$, similarly leads to improved conservation of kinetic energy, as can be expected from the discussion in~\cref{sec:app:quadratic}.

\begin{figure}
  \centering
  
  \subcaptionbox{Change in gas kinetic energy as function of time.\label{fig:validation:advection:show_energy}}
  [\twofigwidth]{
    \tikzsetnextfilename{advection_show_energy_legend}
%
%
\definecolor{mycolor1}{rgb}{0.00000,0.44700,0.74100}%
\definecolor{mycolor2}{rgb}{0.85000,0.32500,0.09800}%
\definecolor{mycolor3}{rgb}{0.46600,0.67400,0.18800}%
\definecolor{mycolor4}{rgb}{0.49400,0.18400,0.55600}%
\begin{tikzpicture}

\begin{axis}[%
  hide axis,
  width=2cm,
  height=2cm,
  ymin=0.99,
  ymax=1,
  legend style={at={(0.5,1.03)}, anchor=south, legend columns=5, legend cell align=left, align=left, draw=white!15!black},
  ]
  \addlegendimage{empty legend}
  \addlegendentry{$\minfracval = \wycflval$: }
  \addplot [line width=1.0pt, color=mycolor1, domain=-0.1:0.1,samples=2, draw=none] {1};
  \addlegendentry{$10^{0}$}
  \addplot [line width=1.0pt, color=mycolor2, domain=-0.1:0.1,samples=2, draw=none] {1};
  \addlegendentry{$10^{-1}$}
  \addplot [line width=1.0pt, color=mycolor3, domain=-0.1:0.1,samples=2, draw=none] {1};
  \addlegendentry{$10^{-2}$}
  \addplot [line width=1.0pt, color=mycolor4, domain=-0.1:0.1,samples=2, draw=none] {1};
  \addlegendentry{$10^{-3}$}
  
\end{axis}
\end{tikzpicture}%

    \def\tikzWidth{\textwidth*0.33} 
    \def\tikzHeight{\textwidth*0.3} 
    \tikzsetnextfilename{advection_paper_show_energy}
%
%
\definecolor{mycolor1}{rgb}{0.00000,0.44700,0.74100}%
\definecolor{mycolor2}{rgb}{0.85000,0.32500,0.09800}%
\definecolor{mycolor3}{rgb}{0.46600,0.67400,0.18800}%
\definecolor{mycolor4}{rgb}{0.49400,0.18400,0.55600}%
\begin{tikzpicture}

\begin{axis}[%
width=8.159in,
height=10.014in,
at={(1.369in,1.352in)},
scale only axis,
xmin=0,
xmax=1,
xlabel style={font=\color{white!15!black}},
xlabel={$t / T$},
ymode=log,
ymin=1e-06,
ymax=0.01,
yminorticks=true,
ylabel style={font=\color{white!15!black}},
ylabel={$\abs{\Delta E^g} / E^g(0)$},
axis background/.style={fill=white},
axis x line*=bottom,
axis y line*=left,
legend style={at={(0.5,0.03)}, anchor=south, legend cell align=left, align=left, draw=white!15!black},
width=\tikzWidth,
height=\tikzHeight,
at={(0,0)}
]
\addplot [color=mycolor1, line width=1.0pt]
  table[]{/home/ronald/git_repos/phd_project/papers/2021_2fluid_pt1_transport/tikz/advection_paper_show_energy-1.tsv};
\addlegendentry{LW}

\addplot [color=mycolor2, line width=1.0pt, forget plot]
  table[]{/home/ronald/git_repos/phd_project/papers/2021_2fluid_pt1_transport/tikz/advection_paper_show_energy-2.tsv};
\addplot [color=mycolor3, line width=1.0pt, forget plot]
  table[]{/home/ronald/git_repos/phd_project/papers/2021_2fluid_pt1_transport/tikz/advection_paper_show_energy-3.tsv};
\addplot [color=mycolor4, line width=1.0pt, forget plot]
  table[]{/home/ronald/git_repos/phd_project/papers/2021_2fluid_pt1_transport/tikz/advection_paper_show_energy-4.tsv};
\addplot [color=white!50!mycolor1, line width=1.0pt, forget plot]
  table[]{/home/ronald/git_repos/phd_project/papers/2021_2fluid_pt1_transport/tikz/advection_paper_show_energy-5.tsv};
\addplot [color=white!50!mycolor1, line width=1.0pt, draw=none, mark=+, mark options={solid, mycolor1}]
  table[]{/home/ronald/git_repos/phd_project/papers/2021_2fluid_pt1_transport/tikz/advection_paper_show_energy-6.tsv};
\addlegendentry{Fromm}

\addplot [color=white!50!mycolor2, line width=1.0pt, forget plot]
  table[]{/home/ronald/git_repos/phd_project/papers/2021_2fluid_pt1_transport/tikz/advection_paper_show_energy-7.tsv};
\addplot [color=mycolor2, line width=1.0pt, draw=none, mark=+, mark options={solid, mycolor2}, forget plot]
  table[]{/home/ronald/git_repos/phd_project/papers/2021_2fluid_pt1_transport/tikz/advection_paper_show_energy-8.tsv};
\addplot [color=white!50!mycolor3, line width=1.0pt, forget plot]
  table[]{/home/ronald/git_repos/phd_project/papers/2021_2fluid_pt1_transport/tikz/advection_paper_show_energy-9.tsv};
\addplot [color=mycolor3, line width=1.0pt, draw=none, mark=+, mark options={solid, mycolor3}, forget plot]
  table[]{/home/ronald/git_repos/phd_project/papers/2021_2fluid_pt1_transport/tikz/advection_paper_show_energy-10.tsv};
\addplot [color=white!50!mycolor4, line width=1.0pt, forget plot]
  table[]{/home/ronald/git_repos/phd_project/papers/2021_2fluid_pt1_transport/tikz/advection_paper_show_energy-11.tsv};
\addplot [color=mycolor4, line width=1.0pt, draw=none, mark=+, mark options={solid, mycolor4}, forget plot]
  table[]{/home/ronald/git_repos/phd_project/papers/2021_2fluid_pt1_transport/tikz/advection_paper_show_energy-12.tsv};
\end{axis}
\end{tikzpicture}%
  }\hfill
  \subcaptionbox{The $L^\infty$-norm of the error at $t = T$.\label{fig:validation:advection:show_energy_linf}}
  [\twofigwidth]{
    \def\tikzWidth{\textwidth*0.33} 
    \def\tikzHeight{\textwidth*0.3} 
    \tikzsetnextfilename{advection_paper_show_energy_linf}
%
%
\definecolor{mycolor1}{rgb}{0.00000,0.44700,0.74100}%
\definecolor{mycolor2}{rgb}{0.85000,0.32500,0.09800}%
\definecolor{mycolor3}{rgb}{0.46600,0.67400,0.18800}%
\definecolor{mycolor4}{rgb}{0.49400,0.18400,0.55600}%
\begin{tikzpicture}

\begin{axis}[%
scale only axis,
xmode=log,
xmin=0.001,
xmax=1,
xminorticks=true,
xlabel style={font=\color{white!15!black}},
xlabel={$\wycflval = \minfracval$},
ymode=log,
ymin=0.01,
ymax=10,
yminorticks=true,
ylabel style={font=\color{white!15!black}},
ylabel={$\|\approximate{\avarstag} - \avarstag\|_{L^\infty(\Omega^g)}$},
axis background/.style={fill=white},
axis x line*=bottom,
axis y line*=left,
legend style={legend cell align=left, align=left, draw=white!15!black},
width=\tikzWidth,
height=\tikzHeight,
at={(0,0)}
]
\addplot [color=black, dashed, line width=1.0pt, forget plot]
  table[]{/home/ronald/git_repos/phd_project/papers/2021_2fluid_pt1_transport/tikz/advection_paper_show_energy_linf-1.tsv};
\addplot [color=mycolor1, line width=1.0pt, draw=none, mark size=0.8pt, only marks, mark=*, mark options={solid, mycolor1}]
  table[]{/home/ronald/git_repos/phd_project/papers/2021_2fluid_pt1_transport/tikz/advection_paper_show_energy_linf-2.tsv};
\addlegendentry{LW}

\addplot [color=mycolor2, line width=1.0pt, draw=none, mark size=0.8pt, only marks, mark=*, mark options={solid, mycolor2}, forget plot]
  table[]{/home/ronald/git_repos/phd_project/papers/2021_2fluid_pt1_transport/tikz/advection_paper_show_energy_linf-3.tsv};
\addplot [color=mycolor3, line width=1.0pt, draw=none, mark size=0.8pt, only marks, mark=*, mark options={solid, mycolor3}, forget plot]
  table[]{/home/ronald/git_repos/phd_project/papers/2021_2fluid_pt1_transport/tikz/advection_paper_show_energy_linf-4.tsv};
\addplot [color=mycolor4, line width=1.0pt, draw=none, mark size=0.8pt, only marks, mark=*, mark options={solid, mycolor4}, forget plot]
  table[]{/home/ronald/git_repos/phd_project/papers/2021_2fluid_pt1_transport/tikz/advection_paper_show_energy_linf-5.tsv};
\addplot [color=gray, dashed, line width=1.0pt, forget plot]
  table[]{/home/ronald/git_repos/phd_project/papers/2021_2fluid_pt1_transport/tikz/advection_paper_show_energy_linf-6.tsv};
\addplot [color=mycolor1, line width=1.0pt, draw=none, mark size=2.5pt, only marks, mark=+, mark options={solid, mycolor1}]
  table[]{/home/ronald/git_repos/phd_project/papers/2021_2fluid_pt1_transport/tikz/advection_paper_show_energy_linf-7.tsv};
\addlegendentry{Fromm}

\addplot [color=mycolor2, line width=1.0pt, draw=none, mark size=2.5pt, only marks, mark=+, mark options={solid, mycolor2}, forget plot]
  table[]{/home/ronald/git_repos/phd_project/papers/2021_2fluid_pt1_transport/tikz/advection_paper_show_energy_linf-8.tsv};
\addplot [color=mycolor3, line width=1.0pt, draw=none, mark size=2.5pt, only marks, mark=+, mark options={solid, mycolor3}, forget plot]
  table[]{/home/ronald/git_repos/phd_project/papers/2021_2fluid_pt1_transport/tikz/advection_paper_show_energy_linf-9.tsv};
\addplot [color=mycolor4, line width=1.0pt, draw=none, mark size=2.5pt, only marks, mark=+, mark options={solid, mycolor4}, forget plot]
  table[]{/home/ronald/git_repos/phd_project/papers/2021_2fluid_pt1_transport/tikz/advection_paper_show_energy_linf-10.tsv};
\end{axis}
\end{tikzpicture}%
  }
  
  \caption{Reducing the time-step size ($\wycflval \rightarrow 0$, where we let $\minfracval = \wycflval$), for the 2D vortex reverse problem, where the mesh width is fixed and given by $h / R \approx 0.1$.
  We only show results from the gas phase, as results from the liquid phase show the same trends.
  The \twofluid formulation was used.}
\end{figure}
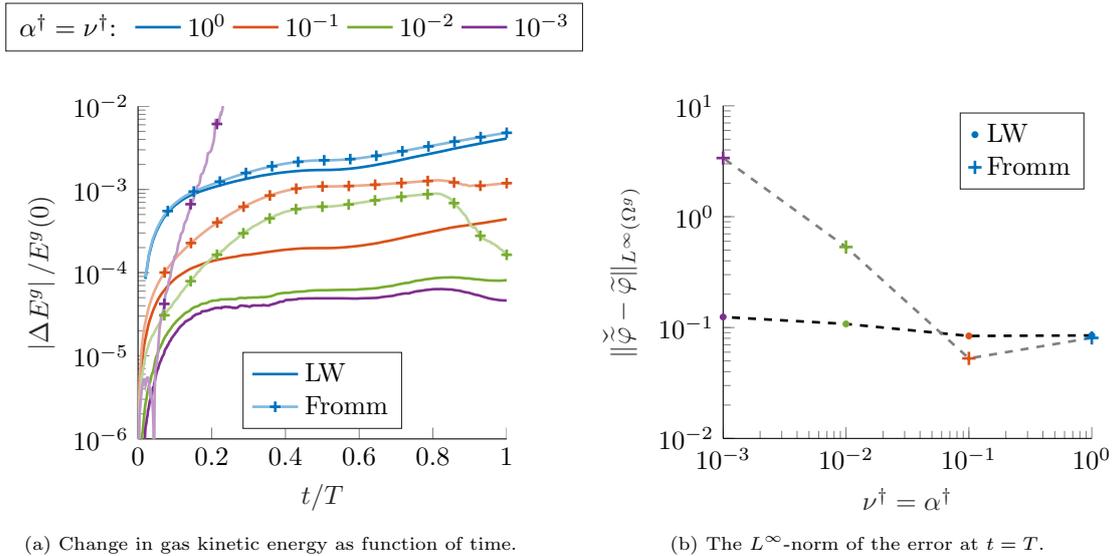
This leads to the question of how much the conservation of kinetic energy can be improved by letting both $\wycflval$ and $\minfracval$ tend to zero.
Based on the upper bound given by~\cref{eqn:momentum:unsplit:modified_simplified}, where the ratio $\wycflval / \minfracval$ appears, we now let $\minfracval = \wycflval$ such that we use the \gls{ctu} flux interpolant as little as possible, while still having boundedness of the solution.
We let $h / R \approx 0.1$, and consider four different values of $\minfracval = \wycflval \in \{10^0, 10^{-1}, 10^{-2}, 10^{-3}\}$.
The resulting evolution of kinetic energy in the gas phase is shown in~\cref{fig:validation:advection:show_energy}, where we consider both the \gls{lw} and Fromm flux interpolant.
For the \gls{lw} flux interpolant we initially find a significant improvement in the conservation of kinetic energy when decreasing the value of $\minfracval = \wycflval$.
Hence even though we use the \emph{modified} \gls{lw} flux interpolant, which uses the dissipative \gls{ctu} flux interpolant (which does not satisfy~\cref{eqn:app:quadratic:lwlimit}) at the interface, the analysis presented in~\cref{sec:app:quadratic} still holds in some sense when $\minfracval = \wycflval \rightarrow 0$.
On the other hand, the Fromm flux interpolant does not yield much of an improvement in the conservation of kinetic energy and moreover yields divergence of the solution, as is illustrated in~\cref{fig:validation:advection:show_energy_linf}.
It's unclear why the Fromm flux interpolant leads to divergence of the solution when $\minfracval = \wycflval \rightarrow 0$.

  \section{Discussion}\label{sec:discussion}
In an effort to model the shear layer at the interface, we have proposed the use of a \twofluid formulation of the two-phase Navier--Stokes equations.
For now we have focussed our attention on the transport of mass and momentum in such a \twofluid formulation.

To conserve mass and momentum in a sharp and accurate way, we follow the approach originally proposed in~\citet{Rudman1998} and use the same approximate space-time integration of the advection equation for momentum as is used for the advection of mass.
The dimensionally unsplit advection method used for the transport of mass and momentum relies on the construction of donating regions.
We have derived sufficient conditions on these donating regions which ensure boundedness of the resulting volume fraction (\cref{cor:interface:onefluid:boundedness}).

For the advection of a staggered momentum field we propose a simple averaging of the same mass fluxes that are used for advecting the centred mass.
For our proposed one- and \twofluid formulations this implies that, besides the conservation of mass, the total and liquid linear momentum respectively are conserved exactly, without the need for computing additional mass fluxes.
We find that such an interpolation yields a direct relation between the centred and staggered advection methods (as given by~\cref{eqn:mass:onefluid:advection_property}).
Furthermore we show that the advection method in semi-discrete form conserves quadratic invariants (i.e. kinetic energy), provided that the Lax--Wendroff flux interpolant is used.

The velocity (per phase) can be computed by dividing the momentum with the corresponding mass, however such a division is not always well-defined if the corresponding mass nearly vanishes during a single time step.
We have introduced a modified flux interpolant for which we can guarantee that the division is well defined, under the condition that the donating regions are absent of any flux overlap and transit errors (\cref{lem:momentum:ctu_boundedness,lem:momentum:mod_boundedness}).

The proposed methods are shown to converge under mesh refinement for two- and three-dimensional reversible deformation test cases.
The passive transport of a discontinuous  scalar shows that the \onefluid formulation yields low accuracy for the lighter of the two phases. This is because the advection of the velocity field in the \onefluid formulation favours the heavier liquid phase.
To the contrary, using the \twofluid formulation results in the same accuracy for both phases, regardless of the density ratio.

In a future paper we will focus our attention on coupling the two phases in the Navier--Stokes equations via a novel jump condition that is included in the pressure Poisson problem.
For future work it would moreover be interesting to see if a less dissipative flux interpolant can be constructed at the interface, resulting in second-order accuracy in the $L^\infty$-norm, while still guaranteeing boundedness.
The proposed methods are compatible with adaptive mesh refinement (AMR).
The restriction and prolongation operators between the refinement levels are however not yet momentum conservative, and it would therefore be interesting to develop this and obtain exact momentum conservation also when AMR is used.

  \section*{Acknowledgements}
This work is part of the research programme SLING, which is (partly) financed by the Netherlands Organisation for Scientific Research (NWO).
We would like to thank the Center for Information Technology of the University of Groningen for their support and for providing access to the Peregrine high performance computing cluster.
Moreover we thank Dr. Joaquín López (Universidad Politécnica de Cartagena) for kindly providing the VoFTools library.

  \appendix
  \setcounter{figure}{0}
  \section{Operator connection}\label{sec:app:connection}
\begin{figure}
  \subcaptionbox{The face $g$ corresponds to a diagonal tensor component.}
  [\twofigwidth]{
    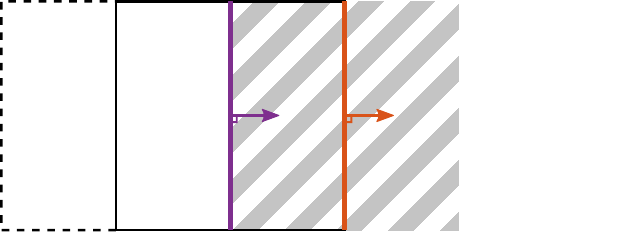
  }\hfill
  \subcaptionbox{The face $g$ corresponds to an off-diagonal tensor component.}
  [\twofigwidth]{
    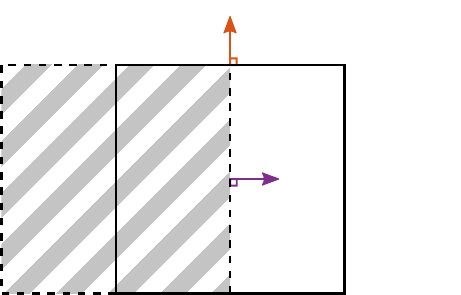
  }
  \caption{We can uniquely identify a centred control volume $\connectionc \in \mathcal{C}(f)$ to the pair $f, h$, provided that $g \in \mathcal{G}(\omega_f)$, $h \in \mathcal{F}(g)$ and $f \neq h$.
  The faces $f, h$ both belong to the boundary of the hatched centred control volume $\connectionc$ and therefore $f, h \in \mathcal{F}(\connectionc)$.}\label{fig:app:connection:illustrate}
\end{figure}
\lemmaoperatorconnection*
\begin{proof}
  We start by applying the left-hand side of~\caref{eqn:notation:notation:operator_connection} to some field $\massflux \in \mathcal{F}^h$
  \begin{equation}\label{eqn:notation:notation:operator_connection:intermediate}
    |\omega_f|(\stagger \divh \interpolantflux \massflux)_f = \half\sum_{g \in \mathcal{G}(\omega_f)}  \sum_{h\in\mathcal{F}(g)} \stagger\orientation_{f,g} (\stagger{\+n}_g \cdot \+n_h) |h| \massflux_h,
  \end{equation} 
  where we have substituted~\cref{eqn:notation:notation:divergence_stag,eqn:notation:notation:interpolantsbp_stag}.
  The union over the faces $\mathcal{F}(g)$, for all faces $g \in \mathcal{G}(\omega_f)$, is the same as the union over the faces $\mathcal{F}(c)$, for all control volumes $c \in \mathcal{C}(f)$ (see also~\cref{fig:notation:notation:set_stag})
  \begin{equation}\label{eqn:app:connection:samesets}
    \bigcup_{g \in \mathcal{G}(\omega_f)} \mathcal{F}(g) = \bigcup_{c \in \mathcal{C}(f)} \mathcal{F}(c).
  \end{equation}
  This identity can be exploited to rewrite the double summation in~\cref{eqn:notation:notation:operator_connection:intermediate}, but in doing so we must get rid of the dependence on $g \in \mathcal{G}(\omega_f)$ in the $\stagger\orientation_{f,g} (\stagger{\+n}_g \cdot \+n_h)$ term, and instead express it in terms of $c \in \mathcal{C}(f)$.

  In order to do so, we uniquely identify a centred control volume to the pair $f, h$, which we denote by $\connectionc$, and for which it holds that $f, h \in \mathcal{F}(\connectionc)$, see also~\cref{fig:app:connection:illustrate}.
  Provided that $f \neq h$, it holds that the outward pointing normal of the staggered control volume $\stagger\omega_f$ coincides with the outward pointing normal of the centred control volume $c \in \mathcal{C}(f)$
  \begin{equation}\label{eqn:app:connection:centred_orientation}
    \stagger\orientation_{f,g} \stagger{\+n}_g = \orientation_{\connectionc,h} \+n_h \quad \implies \quad \stagger\orientation_{f,g} (\stagger{\+n}_g \cdot \+n_h) = \orientation_{\connectionc,h},
  \end{equation}
  as illustrated in~\cref{fig:app:connection:illustrate}.
  Note that the two contributions in~\cref{eqn:notation:notation:operator_connection:intermediate} for $f = h$ cancel, and therefore it does not matter that the above identity only holds for $f \neq h$.
  Substitution of~\cref{eqn:app:connection:centred_orientation} into~\cref{eqn:notation:notation:operator_connection:intermediate}, and subsequently utilising~\cref{eqn:app:connection:samesets}, results in
  \begin{equation}
    |\omega_f|(\stagger \divh \interpolantflux \massflux)_f = \half\sum_{g \in \mathcal{G}(\omega_f)}  \sum_{h\in\mathcal{F}(g)} \orientation_{\connectionc,h} |h| \massflux_h
    \stackrel{\eqref{eqn:app:connection:samesets}}{=} \half\sum_{c \in \mathcal{C}(f)}  \sum_{h\in\mathcal{F}(c)} \orientation_{c,h} |h| \massflux_h
    = |\omega_f|(\interpolantsbp \divh \massflux)_f\label{eqn:notation:notation:operator_connection:final},
  \end{equation} 
  where we have used the definition of the divergence and interpolation operator given by~\cref{eqn:notation:notation:interpolantsbp,eqn:notation:notation:divergence}.
  Equation~\eqref{eqn:notation:notation:operator_connection:final} holds for any $\massflux \in \mathcal{F}^h$ and therefore we find that~\caref{eqn:notation:notation:operator_connection} indeed holds true.
\end{proof}

\section{Analysis of DR methods}\label{sec:app:dr_analysis}
\lemmactuboundedoutflow*
\begin{proof}
  We will omit the superscript $(n)$.
  Substitution of the definition of the partial volume fluxes~\cref{eqn:centred:volume_flux_per_cell} into the definition of the outgoing volume~\cref{eqn:centred:advection_inoutgoing_volume} results in
  \begin{equation}
    |c|\totalfluxctu^{-,\pi}_c = -\sum_{b \in \mathcal{C}(c)} \negflux{-\sum_{f \in \mathcal{F}(c)} \orientation_{c,f} M_0(b \cap \Omega^{\pi} \cap \drOne_f)},
  \end{equation}
  where we recall that $b^{\pi} = b \cap \Omega^{\pi}$.
  By making use of the notation of scalar multiplication, as defined in~\cref{eqn:interfacee:dr:mult}, of an oriented DR (where the scalar must be plus or minus one), we can move the multiplication by $-\orientation_{c,f}$ inside $M_0(\ldots)$ to the oriented set $\drOne_f$, resulting in
  \begin{equation}
    |c|\totalfluxctu^{-,\pi}_c = -\sum_{b \in \mathcal{C}(c)} \negflux{\sum_{f \in \mathcal{F}(c)}  M_0(b \cap \Omega^{\pi} \cap (-\orientation_{c,f}\drOne_f))}.
  \end{equation}
  Note that an absence of flux overlap errors (see~\cref{eqn:interface:onefluid:dr:overlap}) implies that the summation over the faces of the volumes can equivalently be written as the volume of the union of the sets, that is
  \begin{equation}\label{eqn:app:bounded_outflow:outflow}
    |c|\totalfluxctu^{-,\pi}_c = -\sum_{b \in \mathcal{C}(c)} \negflux{M_0(b \cap \Omega^{\pi} \cap \drOne_c)}
  \end{equation}
  where we have defined the oriented union of the DRs as
  \begin{equation}\label{eqn:app:bounded_outflow:union}
    \drOne_c \defeq \bigcup_{f\in\mathcal{F}(c)} -\orientation_{c,f}\drOne_f.
  \end{equation}

  For any two non-oriented sets $A, B$ it holds that $A$ can be written as the union of the following two non-overlapping sets: $A = (A \setminus B) \cup (A \cap B)$.
  The definition of our oriented set allows for a similar result
  \begin{equation}\label{eqn:app:bounded_outflow:split}
    \drOne_c = (\drOne_c \setminus -\drOne_c) \cup (\drOne_c \cap -\drOne_c),
  \end{equation}
  where $\drOne_c \cap -\drOne_c$ results from overlapping DRs with opposite relative orientation, resulting in the phase volume that is merely in transit as shown by the hatched regions in~\cref{fig:interface:mass_transport}.
  Moreover, the intersection volume of any non-oriented set $A$ with the second bracketed term in the right-hand side of~\cref{eqn:app:bounded_outflow:split} vanishes
  \begin{equation}\label{eqn:app:bounded_outflow:vanishing_intersection}
    M_0(A \cap (\drOne_c \cap -\drOne_c)) = |A \cap (\drOne_c^+ \cap \drOne_c^-)| - |A \cap (\drOne_c^- \cap \drOne_c^+)| = 0,
  \end{equation}
  which follows from~\cref{eqn:interface:dr:signed_vol,eqn:interfacee:dr:mult}, and reflects the fact that fluid that is merely in transit does not contribute to the change in volume within the control volume $c$.

  The two bracketed sets in~\cref{eqn:app:bounded_outflow:split} are non-overlapping and therefore the outgoing volume can be written as
  \begin{align}
    |c|\totalfluxctu^{-,\pi}_c &= -\sum_{b \in \mathcal{C}(c)} \negflux{M_0(b \cap \Omega^{\pi} \cap (\drOne_c \setminus -\drOne_c))}\\
    &= -\sum_{b \in \mathcal{C}(c)} \negflux{|b \cap \Omega^{\pi} \cap (\drOne^{+}_c \setminus \drOne^{-}_c)| - |b \cap \Omega^{\pi} \cap (\drOne^{-}_c \setminus \drOne^{+}_c)|},\label{eqn:app:bounded_outflow:outgoing_volume_intermediate}
  \end{align}
  where we made use of~\cref{eqn:app:bounded_outflow:outflow,eqn:app:bounded_outflow:split,eqn:app:bounded_outflow:vanishing_intersection,eqn:interface:dr:signed_vol}.
  Recall from~\cref{sec:interface:oriented_dr} that the orientation of a part of the DR $\drOne_f$ is defined as negative if the face normal $\+n_f$ points into it, and therefore the positively orientated part $\drOne^+_c$ of the oriented union of the DRs, as given by~\cref{eqn:app:bounded_outflow:union}, lies outside of $c$: $\drOne^+_c \cap c = \emptyset$.
  From this observation it follows that $b \cap \drOne^+_c \setminus \drOne^-_c = \emptyset$ if $b = c$ and moreover assuming an absence of flux transit errors (see~\cref{eqn:interface:onefluid:dr:transit}) implies that $b \cap \drOne^-_c \setminus \drOne^+_c = \emptyset$ if $b \neq c$. 
  It follows that~\cref{eqn:app:bounded_outflow:outgoing_volume_intermediate} can be written as
  \begin{equation}
    |c|\totalfluxctu^{-,\pi}_c = -\sum_{b \in \mathcal{C}(c)\setminus c} \negflux{|b \cap \Omega^{\pi} \cap (\drOne^{+}_c \setminus \drOne^{-}_c)|} - \negflux{ - |c \cap \Omega^{\pi} \cap (\drOne^{-}_c \setminus \drOne^{+}_c)|} = |c \cap \Omega^{\pi} \cap (\drOne^{-}_c \setminus \drOne^{+}_c)|,
  \end{equation}
  and therefore
  \begin{equation}
    |c|\totalfluxctu^{-,\pi}_c \le |c \cap \Omega^{\pi}| = |c|\volfrac^{\pi}_c,
  \end{equation}
  which coincides with the bounded outflow condition~\caref{eqn:centred:bounded_outflow}.
\end{proof}

\lemmactuboundedness*
\begin{proof}
  We substitute $\volfracstag^{\pi,(n)} = \volfracstag^{\pi,(n+1)} + \dt \stagger \divh\volumefluxOneStag^{\pi,(n)}$ (as follows from~\cref{eqn:momentum:onefluid:transport,eqn:mass:advection_stagdef,eqn:centred:volfluxstag,eqn:mass:onefluid:massflux_def} with $\avarstag \equiv 1$) into~\cref{eqn:centred:approx_dr_solve} and divide out the constant density $\rho^\pi$ which appears in both the numerator and denominator, this results in
  \begin{equation}
    \avarstag^{\pi,(n+1)} = \avarstag^{\pi,(n)} + \dt \frac{(\stagger \divh\volumefluxOneStag^{\pi,(n)})\avarstag^{\pi,(n)} - \stagger \divh(\volumefluxOneStag^{\pi,(n)}\fluxinterpstag^\mtext{\gls{ctu}}\avarstag^{\pi,(n)})}{\volfracstag^{\pi,(n+1)}},
  \end{equation}
  where we have made use of~\cref{eqn:mass:advection_stagdef} with $\stagger\fluxinterp = \stagger\fluxinterp^\mtext{\gls{ctu}}$.
  Subsequently we combine both terms in the numerator (recall that the staggered divergence operator $\stagger \divh$ is given by~\cref{eqn:notation:notation:divergence_stag})
  \begin{equation}\label{eqn:app:ctuboundedness:intermediate_form}
    \avarstag^{\pi,(n+1)}_f = \avarstag^{\pi,(n)}_f + \dt \frac{\sum_{g\in\mathcal{G}(\omega_f)} \stagger\orientation_{f,g}|g| \volumefluxOneStag^{\pi,(n)}_g \roundpar{\avarstag^{\pi,(n)}_f - (\fluxinterpstag^\mtext{\gls{ctu}}\avarstag^{\pi,(n)})_g}}{|\omega_f|\volfracstag^{\pi,(n+1)}_f},
  \end{equation}
  and substitute the definition of the \gls{ctu} flux interpolant~\cref{eqn:centred:ctu_flux}
  \begin{equation}\label{eqn:app:ctuboundedness:approx_ctuform}
    \avarstag^{\pi,(n+1)}_f = \avarstag^{\pi,(n)}_f + \dt \frac{ \sum_{k \in \mathcal{F}(f)} \sum_{g\in\mathcal{G}(\omega_f)} \stagger\orientation_{f,g}|g| \volumefluxOneStag^{\pi,(n)}_{g,k}\roundpar{\avarstag^{\pi,(n)}_f - \avarstag^{\pi,(n)}_k}}{|\omega_f|\volfracstag_f^{\pi,(n+1)}},
  \end{equation}
  where we have swapped the order of summation.
  We now split the term in the numerator of~\cref{eqn:app:ctuboundedness:approx_ctuform} into the parts corresponding to the in- and outgoing flow 
  \begin{equation}\label{eqn:app:ctuboundedness:ctu_numerator_split}
    \frac{\avarstag^{\pi,(n+1)}_f - \avarstag^{\pi,(n)}_f}{\dt} = \frac{N^{+,\pi}_{f} - N^{-,\pi}_{f}}{\volfracstag_f^{\pi,(n+1)}},
  \end{equation}
  where the terms in the numerator are given by
  \begin{equation}\label{eqn:app:ctuboundedness:ctu_numerator_flux}
    N^{\pm,\pi}_{f} \defeq \pm \frac{1}{|\omega_f|}\sum_{k \in \mathcal{F}(f)} \underbrace{\squarepar{-\sum_{g\in\mathcal{G}(\omega_f)} \stagger\orientation_{f,g}|g| \volumefluxOneStag^{\pi,(n)}_{g,k}}^\pm}_\text{fluid that flows out of $\omega_f$ ($-$), or from $\omega_k$ into $\omega_f$ ($+$).} \roundpar{\avarstag^{\pi,(n)}_k - \avarstag^{\pi,(n)}_f},
  \end{equation}
  corresponding to an outflow of $\avarstag^{\pi,(n)}$ given by $N^{-,\pi}_{f}$, and an inflow into $\omega_f$ given by $N^{+,\pi}_{f}$.
  Note that swapping the order of summation in~\cref{eqn:app:ctuboundedness:approx_ctuform}, combined with the placement of $\squarepar{\ldots}^\pm$ in~\cref{eqn:app:ctuboundedness:ctu_numerator_flux}, exactly  results in the cancellation of any fluxes resulting from neighbouring DRs which are overlapping with an opposite relative orientation.
  This step is essential in being able to separately bound each of the terms in the numerator of~\cref{eqn:app:ctuboundedness:ctu_numerator_split}, and explicitly relies on the use of the partial volume fluxes.

  \def\connectione{\hat b(g,h,k)}
  \def\connectionei{\hat b(g,h_i,k)}
  \def\connectionec{b(c,f,k)}
  \begin{figure}
    \def\be{b}
    \subcaptionbox{The face $g$ corresponds to a diagonal tensor component.
    Here $f = h_2$.}
    [0.45\textwidth]{
      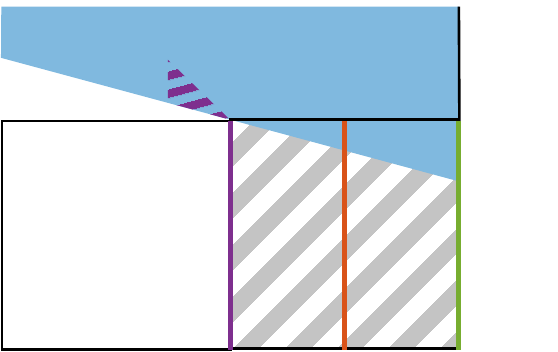
    }\hfill
    \subcaptionbox{The face $g$ corresponds to an off-diagonal tensor component.}
    [0.5\textwidth]{
      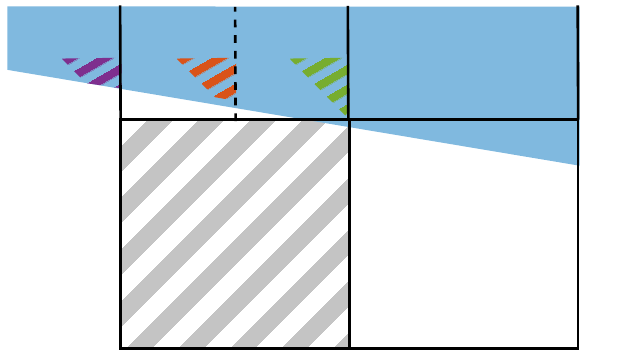
    }
    \caption{Illustration of the interpolated partial volume flux as defined in~\cref{eqn:app:ctuboundedness:interp_partial_flux}.
    The centred control volumes (solid lines) $b_i$ are positioned the same, relative to faces $h_i$, as the staggered control volume $\omega_k$ is positioned relative to $g$.
    In this example we find $b_i = \connectionei \in \mathcal{C}(k)$.
    The gray hatched region corresponds to the control volume $c$, where, for example, $b_1 = \connectionec$: $b_1$ is a diagonal neighbour of $c$ in the same way that $\omega_k$ is a diagonal neighbour of $\omega_{f}$.}\label{fig:app:ctuboundedness:connectione}
  \end{figure}
  For the remainder we need an explicit definition of the interpolated partial volume flux~\cref{eqn:centred:volume_flux_per_cell}, as illustrated in~\cref{fig:app:ctuboundedness:connectione} and given by (cf.~\cref{eqn:notation:notation:interpolantsbp_stag,eqn:centred:volfluxstag})
  \begin{equation}\label{eqn:app:ctuboundedness:interp_partial_flux}
    |g|\volumefluxOneStag^{\pi,(n)}_{g,k} \defeq \half \sum_{h \in \mathcal{F}(g)} (\stagger{\+n}_g \cdot \+n_h) |h| \volumefluxOne^{\pi,(n)}_{h,\connectione},
  \end{equation}
  where the centred control volume $\connectione \in \mathcal{C}(k)$ is as illustrated in~\cref{fig:app:ctuboundedness:connectione} and is positioned relative to $h$ in the same way as $\omega_k$ is positioned relative to $g$.
  Using this definition of the interpolated partial volume flux we can write~\cref{eqn:app:ctuboundedness:ctu_numerator_flux} as
  \begin{align}
    N^{\pm,\pi}_{f} &= \frac{\pm1}{2|\omega_f|}\sum_{k \in \mathcal{F}(f)} \squarepar{-\sum_{g\in\mathcal{G}(\omega_f)} \sum_{h \in \mathcal{F}(g)} \stagger\orientation_{f,g}(\stagger{\+n}_g \cdot \+n_h) |h| \volumefluxOne^{\pi,(n)}_{h,\connectione}}^\pm \roundpar{\avarstag^{\pi,(n)}_k - \avarstag^{\pi,(n)}_f}\\
    &= \frac{\pm1}{2|\omega_f|}\sum_{k \in \mathcal{F}(f)} \squarepar{-\sum_{c\in\mathcal{C}(f)} \sum_{h \in \mathcal{F}(c)} \orientation_{c,h} |h| \volumefluxOne^{\pi,(n)}_{h,\connectionec}}^\pm \roundpar{\avarstag^{\pi,(n)}_k - \avarstag^{\pi,(n)}_f},\label{eqn:app:ctuboundedness:inoutflow_rewritten}
  \end{align}
  where we have furthermore utilised~\cref{eqn:app:connection:samesets,eqn:app:connection:centred_orientation} and have expressed $\connectione$ as $\connectionec$.
  That is, $\omega_k$ is a neighbour of $\omega_f$ in the same way that $\connectionec$ is a neighbour of $c$, see also~\cref{fig:app:ctuboundedness:connectione}.

  We will now first show that the outflow term vanishes identically: $N^{-,\pi}_{f} = 0$.
  The arguments used are identical to those used in the proof of~\cref{thm:mass:bounded_outflow}. 
  Substitution of the definition of the partial volume flux given by~\cref{eqn:centred:volume_flux_per_cell} into~\cref{eqn:app:ctuboundedness:inoutflow_rewritten} results in
  \begin{equation}
    N^{-,\pi}_{f} = - \frac{1}{2\dt|\omega_f|} \sum_{k \in \mathcal{F}(f)} \squarepar{-\sum_{c\in\mathcal{C}(f)} \sum_{h \in \mathcal{F}(c)} \orientation_{c,h} M_0(\drOne^{(n)}_h \cap \connectionec^{\pi,(n)})}^- \roundpar{\avarstag^{\pi,(n)}_k - \avarstag^{\pi,(n)}_f}.
  \end{equation}
  We then move the multiplication by the orientation $-\orientation_{c,h}$ into the signed volume, and moreover make use of the assumption that neighbouring \glsplural{dr} do not overlap, which permits replacing the summation over $h \in \mathcal{F}(c)$ with a union of the \glsplural{dr} (as defined in~\cref{eqn:app:bounded_outflow:union}), resulting in
  \begin{equation}
    N^{-,\pi}_{f} = - \frac{1}{2\dt|\omega_f|} \sum_{k \in \mathcal{F}(f)} \squarepar{\sum_{c\in\mathcal{C}(f)}  M_0(\drOne^{(n)}_c \cap \connectionec^{\pi,(n)})}^- \roundpar{\avarstag^{\pi,(n)}_k - \avarstag^{\pi,(n)}_f}.
  \end{equation}
  Subsequently we make use of~\cref{eqn:app:bounded_outflow:split,eqn:app:bounded_outflow:vanishing_intersection}, from which it follows that the outflow term can be written as
  \begin{multline}
    N^{-,\pi}_{f} = - \frac{1}{2\dt|\omega_f|} \\ \sum_{k \in \mathcal{F}(f)} \squarepar{\sum_{c\in\mathcal{C}(f)} M_0(\connectionec^{\pi,(n)} \cap (\drOne^{(n)}_c \setminus -\drOne^{(n)}_c))}^- \roundpar{\avarstag^{\pi,(n)}_k - \avarstag^{\pi,(n)}_f}.\label{eqn:app:bounded_outflow:outflowterm}
  \end{multline}
  Note that only $k \neq f$ will contribute to~\cref{eqn:app:bounded_outflow:outflowterm}, and therefore by the definition of $\connectionec$ we can assume $\connectionec \neq c$.

  Recall that $\drOne^{(n)}_c$ is an oriented \gls{dr}, and therefore
  \begin{equation}\label{eqn:app:bounded_outflow:recall_signed_volume}
    M_0(b^\pi \cap (\drOne^{(n)}_c \setminus -\drOne^{(n)}_c)) \stackrel{\eqref{eqn:interface:dr:signed_vol}}{=} \abs{b^\pi \cap (\drOne^{+,(n)}_c \setminus \drOne^{-,(n)}_c)} - \abs{b^\pi \cap (\drOne^{-,(n)}_c \setminus \drOne^{+,(n)}_c)},
  \end{equation} 
  where we note that $\drOne^{-,(n)}_c\setminus \drOne^{+,(n)}_c \cap b = \emptyset$ if $b \neq c$, thanks to an absence of flux transit errors.
  This implies that~\eqref{eqn:app:bounded_outflow:recall_signed_volume} can be written as (assuming $b \neq c$)
  \begin{equation}
    M_0(b^\pi \cap (\drOne^{(n)}_c \setminus -\drOne^{(n)}_c)) = \abs{b^\pi \cap (\drOne^{+,(n)}_c \setminus \drOne^{-,(n)}_c)}.
  \end{equation}
  This result can be used to write~\cref{eqn:app:bounded_outflow:outflowterm} as
  \begin{multline}
    N^{-,\pi}_{f} = - \frac{1}{2\dt|\omega_f|} \\ \sum_{k \in \mathcal{F}(f)} \squarepar{\sum_{c\in\mathcal{C}(f)} \abs{\connectionec^{\pi,(n)} \cap (\drOne^{+,(n)}_c \setminus \drOne^{-,(n)}_c)}}^- \roundpar{\avarstag^{\pi,(n)}_k - \avarstag^{\pi,(n)}_f} = 0,\label{eqn:app:ctuboundedness:ctu_numerator_minus}
  \end{multline}
  where we made use of $[x]^- = 0$ if $x \geq 0$.
  Hence thanks to an assumed absence of flux overlap and transit errors, as well as the use of the \gls{ctu} flux interpolant, we can guarantee that the outflow $N^{-,\pi}_f$ does not affect the change in $\avarstag^{\pi}_f$ in~\cref{eqn:app:ctuboundedness:ctu_numerator_split}.
  
  The inflow, which was the first term in the numerator of the right-hand side of~\cref{eqn:app:ctuboundedness:ctu_numerator_split}, can be bounded in the following way
  \begin{equation}\label{eqn:app:ctuboundedness:ctu_numerator_plus}
    |N^{+,\pi}_{f}| \le \frac{\stagger V^{+,\pi}_f}{\dt} \max_{k\in\mathcal{F}(f)}\abs{\avarstag^{\pi,(n)}_k - \avarstag^{\pi,(n)}_f},
  \end{equation}
  where we made use of~\cref{eqn:app:ctuboundedness:inoutflow_rewritten} and we define the staggered ingoing flow (cf.~\cref{eqn:centred:advection_inoutgoing_volume}) as
  \begin{equation}\label{eqn:app:ctuboundedness:stagger_outvolume}
    \stagger V^{+,\pi}_f \defeq \frac{\dt}{2|\omega_f|}\sum_{k \in \mathcal{F}} \squarepar{-\sum_{c\in\mathcal{C}(f)} \sum_{h \in \mathcal{F}(c)} \orientation_{c,h} |h| \volumefluxOne^{\pi,(n)}_{h,\connectionec}}^+.
  \end{equation}
  In~\cref{eqn:app:ctuboundedness:stagger_outvolume} we have replaced the summation over $k \in \mathcal{F}(f)$ (summing over the neighbours of $\omega_f$) by a summation over $k \in \mathcal{F}$ (assuming a \gls{cfl} constraint there will be no contributions from $\mathcal{F}\setminus \mathcal{F}(f)$).
  We now note that summing over all $k \in \mathcal{F}$ results in all centred control volumes $\connectionec$ to be considered, which implies that the staggered ingoing flow can equivalently be written as
  \begin{equation}
    \stagger V^{+,\pi}_f \defeq \frac{\dt}{2|\omega_f|}\sum_{b \in \mathcal{C}} \squarepar{-\sum_{c\in\mathcal{C}(f)} \sum_{h \in \mathcal{F}(c)} \orientation_{c,h} |h| \volumefluxOne^{\pi,(n)}_{h,b}}^+.
  \end{equation}
  The staggered ingoing flow can now be bounded in the following way
  \begin{equation}
    \stagger V^{+,\pi}_f \le \frac{\dt}{2|\omega_f|}\sum_{c\in\mathcal{C}(f)}\sum_{b\in\mathcal{C}} \squarepar{- \sum_{h\in\mathcal{F}(c)} \orientation_{c,h}|h| \volumefluxOne^{\pi,(n)}_{h,b}}^+ \stackrel{\eqref{eqn:centred:advection_inoutgoing_volume}}{=} (\interpolantsbp V^{+,\pi})_f,\label{eqn:app:ctuboundedness:stagger_outvolume_bound}
  \end{equation}
  where we have used the algebraic identity $\posflux{a_1 + a_2 + \ldots} \le \posflux{a_1} + \posflux{a_2} + \ldots$ and finally have substituted the definition of the centred outgoing volume~\cref{eqn:centred:advection_inoutgoing_volume} as well as the interpolant given by~\cref{eqn:notation:notation:interpolantsbp}.

  Since we have assumed that the \glsplural{dr} do not commit any flux overlap nor transit errors, we may use the result of~\cref{thm:mass:bounded_outflow} which implies that the interpolated ingoing volume flow is bounded by the staggered volume fraction at $t = t^{(n+1)}$
  \begin{equation}\label{eqn:app:ctuboundedness:ctu_bounded_inflow}
    \interpolantsbp V^{+,\pi} \stackrel{\eqref{eqn:centred:advection_reorderedsummation}}{=} \interpolantsbp\squarepar{\volfrac^{\pi,(n+1)} - (\volfrac^{\pi,(n)} -  V^{-,\pi})} \le \interpolantsbp{\volfrac^{\pi,(n+1)}} = \volfracstag^{\pi,(n+1)}.
  \end{equation}
  Combining the results of~\cref{eqn:app:ctuboundedness:ctu_numerator_split,eqn:app:ctuboundedness:ctu_numerator_minus,eqn:app:ctuboundedness:ctu_numerator_plus,eqn:app:ctuboundedness:stagger_outvolume_bound,eqn:app:ctuboundedness:ctu_bounded_inflow} then concludes the proof.
\end{proof}

\lemmamodboundedness*
\begin{proof}
  As with the proof of~\cref{lem:momentum:ctu_boundedness}, we can write~\cref{eqn:momentum:onefluid:transport} as (cf.~\cref{eqn:app:ctuboundedness:ctu_numerator_split})
  \begin{equation}\label{eqn:momentum:flux_interp_mod:numerator_split}
    \frac{\avarstag_f^{\pi,(n+1)} - \avarstag_f^{\pi,(n)}}{\dt}  = \frac{M^{+,\pi}_{f} - M^{-,\pi}_{f}}{\volfracstag_f^{\pi,(n+1)}}
  \end{equation}
  where the terms in the numerator are defined as (cf.~\cref{eqn:app:ctuboundedness:ctu_numerator_flux})
  \begin{equation}
    M^{+,\pi}_{f} \defeq \pm\frac{1}{|\omega_f|}\sum_{g\in\mathcal{G}(\omega_f)} \squarepar{-\stagger\orientation_{f,g} |g| \volumefluxOneStag_g^{\pi,(n)}}^\pm\roundpar{(\fluxinterpstag^{\mtext{Method}^*}\avarstag^{\pi,(n)})_g - \avarstag_f^{\pi,(n)}}.
  \end{equation}
  Each of the terms in the numerator of the right-hand side of~\cref{eqn:momentum:flux_interp_mod:numerator_split} can be bounded as follows
  \begin{equation}\label{eqn:momentum:flux_interp_mod:numerator_bound}
    \abs{M^{\pm,\pi}_{f}} \le \frac{\stagger W^{\pm,\pi}_f}{\dt} \max_{g\in\mathcal{G}(\omega_f)} \abs{(\fluxinterpstag^{\mtext{Method}^*}\avarstag^{\pi,(n)})_g - \avarstag_f^{\pi,(n)}}.
  \end{equation}
  Combining~\cref{eqn:momentum:flux_interp_mod:numerator_split,eqn:momentum:flux_interp_mod:numerator_bound} as well as the assumption $\volfracstag^{\pi,(n+1)}_f \ge \minfracval$ then yields the desired result.
\end{proof}

\section{Semi-discrete conservation of quadratic invariants}\label{sec:app:quadratic}
We will use the following short-hand notation for discrete integration
\begin{equation}\label{eqn:sec:app:quadratic:integrate}
  \sum_\setextrusion{F} \avarstag \defeq \sum_{f\in\mathcal{F}} |\omega_f|\avarstag_f, \quad \sum_\mathcal{G} \tvar \defeq \sum_{g\in\mathcal{G}} |\stagger\omega_g|\tvar_g.
\end{equation}
Whenever we refer to the adjoint of an operator, e.g. $\divh^T$, the use of the $L^2$ inner product induced by the previously defined integral functionals in~\cref{eqn:sec:app:quadratic:integrate} is implied.
We consider the \onefluid formulation.

Conservation of quadratic invariants is usually studied in semi-discrete form~\citep{Veldman2019,Verstappen2003}.
Our discretisation of~\cref{eqn:intro:conservation_eqn} is however based on approximate space-time integration, resulting in~\cref{eqn:momentum:onefluid:transport}, and we therefore initially do not have a semi-discrete equation to study.
We can however take the limit $\dt \rightarrow 0$ of~\cref{eqn:momentum:onefluid:transport}, and obtain a semi-discretisation via that way.
In this limit the volume flux can be written as
\begin{equation}
  \lim_{\dt \rightarrow 0} \volumefluxOne^{\pi,(n)} = a^{\pi,(n)} \onevelo^{(n)},
\end{equation}
and therefore the space-time integration approach discussed in~\cref{sec:momentum}, resulting in~\cref{eqn:momentum:onefluid:transport}, can be viewed as a specialised time integration of the following semi-discretisation
\begin{equation}\label{eqn:app:quadratic:semidisc}
  \frac{d}{dt} (\volfracstag \rho \avarstag)^{\pi} + \advectionstag{\massfluxOne^\pi}\avarstag^{\pi} = 0,
\end{equation}
where the mass flux is now given by
\begin{equation}
  \massfluxOne^\pi = \rho^\pi a^{\pi} \onevelo.
\end{equation}
Moreover, in the same limit we find that the LW flux interpolant~\cref{eqn:centred:flux_interp:interp_lw} reduces to a central spatial discretisation (see~\cref{fig:momentum:flux_interp:example}, in this limit the centroid of the DR coincides with the face centroid, and therefore $\tanginterp_1\avarstag$ will coincide with the upwind centred value of $\avarstag$)
\begin{equation}\label{eqn:app:quadratic:lwlimit}
  \lim_{\dt \rightarrow 0} \fluxinterpstag^{\mtext{LW}} \avarstag = \half \tanginterp_0\avarstag + \half\tanginterp_1\avarstag = \interpolanthalfstag{\avarstag},
\end{equation}
where the equal weight interpolant $\interpolanthalfstag: \mathcal{F}^h \rightarrow \mathcal{G}^h$ is defined as
\begin{equation}
  (\interpolanthalfstag{\avarstag})_g \defeq \half \sum_{f\in\mathcal{F}^\omega(g)} \avarstag_f.
\end{equation}

We will now show that the semi-discretisation~\cref{eqn:app:quadratic:semidisc} conserves kinetic energy if the LW interpolant is used.
For ease of notation we will omit the superscript $\pi$ on $\volfrac, \rho, \avar$ and $\massfluxOne$.
The temporal evolution of the kinetic energy density is given by
\begin{equation}\label{eqn:app:quadratic:kinetic_evo}
  \half\frac{d}{dt} (\rho \volfracstag \avarstag^2) = \half\frac{d}{dt} \squarepar{\frac{(\rho \volfracstag \avarstag)^2}{\rho \volfracstag}} = \avarstag \frac{d}{dt} (\rho \volfracstag \avarstag) - \half \avarstag^2 \frac{d}{dt} (\rho \volfracstag) = -\avarstag \advectionstag{\massfluxOne}\avarstag + \half \avarstag^2 \advectionstag{\massfluxOne}1,
\end{equation}
where we have used the quotient rule, product rule as well as~\cref{eqn:app:quadratic:semidisc} and again~\cref{eqn:app:quadratic:semidisc} with $\avarstag=1$.

Integration of the first term that appears in the operator in the right-hand side of~\cref{eqn:app:quadratic:kinetic_evo} results in
\begin{equation}\label{eqn:app:quadratic:energy_firstterm}
  -\sum_\setextrusion{F} \avarstag \advectionstag{\massfluxOne} \avarstag = -\sum_\setextrusion{F} \avarstag \stagger \divh\roundpar{(\interpolantflux \massfluxOne) (\fluxinterpstag \avarstag)} = \sum_\mathcal{G} (\stagger \gradh\avarstag) {(\interpolantflux \massfluxOne) (\fluxinterpstag \avarstag)},
\end{equation}
where we have substituted~\cref{eqn:mass:advection_stagdef} and used that the staggered gradient operator is the skew-adjoint of the staggered divergence operator.
At this point we need that the flux interpolant $\fluxinterpstag$ is given by the LW interpolant, which reduces to the equal weight interpolant $\interpolanthalfstag$ in the semi-discrete limit $\dt \rightarrow 0$.
This is needed because the staggered gradient $\stagger \gradh$ and the interpolant $\interpolanthalfstag$ satisfy the discrete equivalent of the following special case of the product rule
\begin{equation}\label{eqn:app:quadratic:product_rule}
  \avar \frac{d}{dx}\avar = \half \frac{d}{dx} \avar^2 \quad \xrightarrow[\text{equivalent}]{\text{discrete}} \quad (\interpolanthalfstag\avarstag) (\stagger \gradh \avarstag) = \half \stagger \gradh \avarstag^2.
\end{equation}
which follows from the algebraic relation (for $a, b \in \mathbb{R}$)
\begin{equation}
  \half (a+b) (b-a) = \half (b^2 - a^2).
\end{equation}
Substitution of~\cref{eqn:app:quadratic:product_rule} into~\cref{eqn:app:quadratic:energy_firstterm} (with $\fluxinterpstag=\interpolanthalfstag$) yields
\begin{equation}
  -\sum_\setextrusion{F} \avarstag \advectionstag{\massfluxOne} \avarstag = \half\sum_\mathcal{G}  (\interpolantflux \massfluxOne) (\stagger \gradh \avarstag^2) = -\half\sum_\setextrusion{F}  \avarstag^2 \stagger \divh(\interpolantflux \massfluxOne) = -\half\sum_\setextrusion{F}  \avarstag^2 \advectionstag{\massfluxOne}1,
\end{equation}
which coincides, up to the sign, with the second term in the right-hand side of~\cref{eqn:app:quadratic:kinetic_evo}.
Hence when combined with~\cref{eqn:app:quadratic:kinetic_evo} this shows that kinetic energy is indeed conserved
\begin{equation}
  \half\frac{d}{dt} \sum_\setextrusion{F} \rho \volfracstag \avarstag^2 = 0.
\end{equation}

  \bibliography{library}
\end{document}